\documentclass[10pt]{article}
\usepackage{amssymb,amsmath,textcomp}
\usepackage{enumerate}
\usepackage{hyperref}
\usepackage{mathrsfs}
\usepackage{amsthm}
\usepackage{xcolor}
\newtheorem{satz}{Theorem}[section]
\newtheorem{lemma}{Lemma}[section]

\newtheorem{bemerk1}{Remark}[section]
\newtheorem{bei}{Example}[section]
\newtheorem{korollar}{Corollary}[section]
\newtheorem{prop}{Proposition}[section]
\newcommand{\bvr}{\overline{\Phi}(r)}
\newcommand{\iR}{\mathbb{R}}
\newcommand{\iN}{\mathbb{N}}

\newcommand{\oH}{\hspace*{0.39em}\raisebox{0.6ex}{\textdegree}\hspace{-0.72em}H}
\DeclareMathOperator*{\eosc}{ess\,osc}
\DeclareMathOperator*{\esup}{ess\,sup}
\DeclareMathOperator*{\einf}{ess\,inf}
\newcommand{\dd}{d\mu(\alpha)}

\newcommand{\ki}{k_{1}}
\newcommand{\kjr}{k_{1}(\Phi(r))}
\newcommand{\tl}{\tilde{l}}

\newcommand{\supp}{\mathop{\mathrm{supp}}\limits}
\newcommand{\vs}{\varsigma}
\newcommand{\ka}{\kappa}

\newcommand*{\norm}[1]{\left\Vert{#1}\right\Vert}
\newcommand*{\abs}[1]{\left\vert{#1}\right\vert}
\newcommand*{\om}{\omega}
\newcommand*{\Om}{\Omega}

\newcommand*{\izi}{\int_{0}^{\infty}}
\newcommand*{\izt}{\int_{0}^{t}}
\newcommand*{\izj}{\int_{0}^{1}}

\newcommand*{\al}{\alpha}

\newcommand*{\mg}{\frac{\mu(\alpha)}{\Gamma(1-\alpha)}}

\newcommand*{\vf}{\varphi}

\newcommand*{\ve}{\varepsilon}

\def\divv{\operatorname {div}}

\newcommand{\eqq}[2]{\begin{equation}  #1  \label{#2}\end{equation}    }
\newcommand{\hd}{\hspace{0.2cm}}

\newcommand{\no}{\noindent}
\newcommand{\m}[1]{\mbox{#1}}

\newcommand{\R}{\mathbb{R}}

\newcommand{\jd}{\frac{1}{2}}

\newcommand*{\esssup}{\mathop{\mathrm{ess\hspace{0.05cm}sup}}\limits}
\newcommand*{\essinf}{\mathop{\mathrm{ess\hspace{0.05cm}inf}}\limits}

\newcommand*{\gmb}{\gamma_{-}}

\newcommand{\vr}{\Phi(r)}
\newcommand{\vdr}{\Phi(2r)}

\newcommand{\eqns}[1]{
\begin{eqnarray*}
\begin{split}
#1
\end{split}
\end{eqnarray*}
}

\newcommand{\eqnsl}[2]{
\begin{equation}
\label{#2}
\begin{split}
#1
\end{split}
\end{equation}
}

\newcommand{\vsi}{\varsigma}

\newcommand{\muf}{\tilde{\nu}}

\newcommand{\ti}[1]{\tilde{#1}}
\newcommand{\tss}{t_{**}}
\newcommand{\tit}{\ti{t}}

\newcommand{\vkrl}{\overline{\Phi}(2^{-l}r)}
\newcommand{\vkrll}{\overline{\Phi}(2^{-(l+1)}r)}
\newcommand{\blrj}{B_{2^{-l}r}(x_{1})}
\newcommand{\bllrj}{B_{2^{-(l+1)}r}(x_{1})}
\newcommand{\bllrt}{B_{2^{-(l+1)}\cdot \frac{3}{2}r}(x_{1})}

\newcommand{\tw}{\tilde{w}}
\newcommand{\tv}{\tilde{v}}
\newcommand{\Qdl}{Q(2^{-l})}
\newcommand{\Qdll}{Q(2^{-(l+1)})}
\newcommand{\dml}{2^{-l}r}

\newcommand{\nic}[1]{ }

\newcommand{\Gf}{\tilde{G}}
\newcommand{\lf}{\tilde{l}}
\newcommand{\Pk}{\overline{\Phi}}
\newcommand{\Pkr}{\overline{\Phi}(r)}
\newcommand{\Pkro}{\overline{\Phi}(\rok)}

\newcommand{\tQ}{\tilde{Q}}
\newcommand{\rok}{\tilde{\rho}}

\newcommand{\nuNj}{\nu_{N+1}}
\newcommand{\nuN}{\nu_{N}}

\begin{document}
\begin{center}
{\bf\Large H\"older regularity for nonlocal in time subdiffusion equations with general kernel}
\renewcommand{\thefootnote}{\fnsymbol{footnote}}
\end{center}
\vspace{0.7em}
\begin{center}
Adam Kubica*, Katarzyna Ryszewska*, Rico Zacher${}^{\diamond}$
\end{center}

\vspace{1cm}
{\footnotesize \noindent \nic{{\bf address:} }*Department of Mathematics and Information Sciences\\
Warsaw University of Technology\\
Koszykowa 75, 00-662 Warsaw, Poland \\
Adam.Kubica@pw.edu.pl \\
Katarzyna.Ryszewska@pw.edu.pl \\
\\
${}^{\diamond}$Institute of Applied Analysis \\
University of Ulm \\
89069 Ulm, Germany \\
rico.zacher@uni-ulm.de \\
} \vspace{0.7em}
\begin{abstract}
We study the regularity of weak solutions to nonlocal in time subdiffusion equations
for a wide class of weakly singular kernels appearing in the generalised fractional derivative operator.
We prove a weak Harnack inequality for nonnegative weak
supersolutions and Hölder continuity of weak solutions to such problems. Our results
substantially extend the results from our previous work \cite{harnackdistr} by leaving the framework of distributed order
fractional time derivatives and considering a general $\mathscr{PC}$ kernel and by also allowing for an inhomogeneity in the PDE from a Lebesgue space of mixed type.
\end{abstract}
\vspace{0.7em}
\begin{center}
\end{center}

\section{Introduction and main results}
The aim of this article is to establish a weak Harnack inequality for nonnegative weak supersolutions and H\"older continuity of weak solutions to subdiffusion equations of the
form

\begin{equation} \label{MProb}
\partial_t\big(k*(u-u_0)\big)-\mbox{div}\,\big(A(t,x)Du\big)=f,\quad t\in (0,T),\,x\in
\Omega.
\end{equation}
Here $T > 0$, $\Omega$ is a bounded domain in $\mathbb{R}^{\mathbb{N}}$, $N \geq 1$, the functions $u_0$ (the initial data), $f$ and $A=(a_{ij})$ (which takes values in $\mathbb{R}^{N\times N}$) are given functions. By $Du$ we denote the gradient w.r.t.\ to the spatial variables and $f_1\ast f_2$ stands for the
convolution on the positive half-line w.r.t.\ time, that is
$(f_1\ast f_2)(t)=\int_0^t f_1(t-\tau)f_2(\tau)\,d\tau$, $t\ge 0$. Throughout this paper,
the coefficients $A(t,x)$ are merely assumed to be measurable, bounded and uniformly elliptic.
For the kernel $k$ we assume that it is of type $\mathscr{PC}$, i.e. $k\in L_{1,\,loc}(\iR_+)$ is nonnegative and nonincreasing, and there
exists a nonnegative kernel $l\in L_{1,\,loc}(\iR_+)$ such that $k\ast l=1$ in
$(0,\infty)$. We call $(k,l)$ a $\mathscr{PC}$ pair.

The results established in this paper are substantial generalisations of the results from  \cite{harnackdistr}, where the authors considered the problem (\ref{MProb}) with a {\em bounded} function $f$ and a kernel $k$ leading to a {\em distributed order fractional time derivative}, that is $k$ is of the special form
\begin{equation} \label{distorder}
k(t) = \izj \frac{t^{-\al}}{\Gamma(1-\al)}\dd,\quad t>0,
\end{equation}
which in turn generalises the single order fractional time derivative case, where
\begin{equation} \label{sing}
k(t)=\frac{t^{-\al}}{\Gamma(1-\al)},\quad t>0,\quad \mbox{with some}\; \alpha\in (0,1).
\end{equation}
In the present paper, we develop the theory from \cite{harnackdistr} further in two directions. Firstly, we leave the framework of the distributed order derivative and consider a {\em general} $\mathscr{PC}$ kernel. The wider perspective allows us to recognize the key properties of the kernel $k$ which are essential for deriving the desired regularity results. This systematic treatment gives us better insight into the mechanisms of the De Giorgi-Nash-Moser techniques applied in the nonlocal framework. Secondly, we consider the problem with a function $f$ belonging merely to a Lebesgue space of mixed type, $L_{q_1}(L_{q_2})$, which is of great importance
with regard to possible applications, such as the global solvability of related quasilinear problems.

Since De Giorgi-Nash-Moser techniques rely strongly on scaling properties of the equation, the choice of a suitable geometry for the local sets (time-space cylinders) used in the iteration procedures of De Giorgi and Moser is of fundamental importance. In the first two contributions
\cite{base} (weak Harnack inequality) and \cite{Zhol} (H\"older continuity) on local regularity estimates for nonlocal in time subdiffusion equations, the author studies the case \eqref{sing} of a fractional time derivative of order $\alpha\in (0,1)$ and uses time-space cylinders of the form $(t_0,t_0+\Phi(r))\times B_{r}(x_0)$ with $\Phi(r) = r^{\frac{2}{\al}}$, which is a natural choice since it reflects the fact that the time-fractional diffusion
problem admits a scaling with the similarity variable $s=|x|^{2}t^{-\al}$. At this point we would also like to mention important contributions for the problem with
single order fractional time derivative and fractional diffusion in space; Hölder continuity
of weak solutions was proved in \cite{Caf}, and a weak Harnack inequality was established
in \cite{JPY}. Again, the equation has a scaling property, which determines the geometry of the local time-space cylinders. However, the problem with general kernel $k$ lacks a scaling. In the case of a distributed order fractional derivative \eqref{distorder} treated in \cite{harnackdistr}, this crucial difficulty has been overcome by introducing the scaling function $\Phi(r)$ (for sufficiently small $r$) as the solution to $\izj (\Phi(r))^{-\al}\dd = r^{-2}$, which boils down to $r^{\frac{2}{\al}}$ in the single order case, where
$\mu=\delta(\cdot-\alpha)$ is the Dirac measure at $\alpha$. In the present paper, looking at the problem from a wider perspective, we are able to understand even better than in \cite{harnackdistr} how the time-space cylinders should be selected. It turns out that
(up to a multiplicative constant) the function $\Phi(r)$ shall be chosen as the solution to the functional equation
\begin{equation} \label{PHIEQUA}
(1*l)(\Phi(r)) = r^2,
\end{equation}
for sufficiently small $r>0$. Note that in the case of a single order fractional derivative the
$\mathscr{PC}$ pair is given by $(k,l) = (\frac{t^{-\al}}{\Gamma(1-\al)},\frac{t^{\al-1}}{\Gamma(\al)})$ and we then obtain $\Phi(r) = r^\frac{2}{\al}(\Gamma(\al+1))^{\frac{1}{\al}}$, which coincides up to a constant with the previously discussed choice $\Phi(r) = r^\frac{2}{\al}$.
The new definition of the scaling function $\Phi$ also leads to the same results in the general distributed order fractional derivative case, since for sufficiently small $t>0$ we obtained for $l$ associated with $k$ given by
\eqref{distorder}
 that
\[
c_1 \izj t^{-\al}\dd \leq \frac{1}{(1*l)(t)} \leq c_{2} \izj t^{-\al}\dd,
\]
with some constants $c_1,c_2>0$ (see Lemma 2.1 and formula (38) in \cite{harnackdistr}).

We point out that the problem of determining a suitable intrinsic scale depending on properties of the kernel has also been studied in the context of nonlocal in space operators structurally defined like the fractional Laplacian but with a more general kernel, see (\cite{KaMi}).

Although, the basic structure of our arguments follows the one introduced in \cite{base} and developed further in \cite{harnackdistr}, the proofs given in this paper are much more involved. Surprisingly, it turns out that establishing the weak Harnack inequality for nonnegative weak supersolutions to  (\ref{MProb}) with unbounded function $f$ is much more challenging than in the case with bounded inhomogeneity. It is worth mentioning that this phenomenon does not appear in the classical parabolic case. The main difficulty in case of unbounded $f$ arises in the logarithmic estimates, where we now have to
treat inequalities of the form
\[
\partial_{t}(k*v)(t) + \theta(t) v(t) \geq g(t),
\]
with a {\em time-dependent coefficient} $\theta$ belonging merely to some $L_{q}$ - space with $q \in (1,\infty)$. The theory for such evolution problems with time-dependent $\theta$ is much less developed than in the time-independent case. Hence, we were forced to find completely new and more involved arguments than in \cite{base} and \cite{harnackdistr}.

Before formulating the main results we describe the assumptions on the considered pairs of kernels $(k,l)$.
${}$\\

{\em 1. $\mathscr{PC}$ property, regularity and monotonicity:}
\eqq{
 k,l \in L_{1,\,loc}(\iR_+)\cap C^{1}((0,\infty)), \hd  k,l \m{ are nonnegative and nonincreasing, and } k*l=1.
\tag{K0}}{pc}

{\em 2. Higher integrability of $l$ and comparability with the average:} There exist $p_{0}>1$ and $t_0 > 0,\overline{c} \geq 1 $ such that
\eqq{l \in L_{p_0}((0,t_0)) \hd \m{ and }\;
\frac{1}{t} \int_{0}^{t}\big(l(s)\big)^{p_0} ds \leq \overline{c} \big(l(t)\big)^{p_0},\quad 0<t\le t_0.\tag{K1}}{ak1}

{\em 3. Upper estimate of $k$ via its derivative:} There exist $\tilde{c}\in (0,1)$ and $\tilde{t}_0 > 0$ such that
\eqq{
- t\dot{k}(t) \geq \tilde{c} {k(t)},\quad t \in (0,\tilde{t}_0). \tag{K2}
}{ak2}

{\em 4. Additional assumption for the proof of H\"older continuity:}
There exist a constant $\beta > 0$ and a nondecreasing function $c:(0,\infty)\to (0,\infty)$ such that for all $D>0$
\eqq{-xy\dot{k}(xy) \leq c(D) \frac{k_1(y)}{x^{\beta}},\quad 0<y \leq 1,\; 1\leq  x \leq \frac{D}{y}, \hd \tag{K3}}{asholder}
where $k_1$ is defined as
\eqq{
\ki(t) = \frac{1}{(1*l)(t)},\quad t>0.
}{kidef}

The following examples of $\mathscr{PC}$ kernels satisfy the assumptions (\ref{pc}) - (\ref{asholder}), see also the appendix.
\begin{itemize}
\item the time fractional derivative with exponential weight
\[
k(t) = \frac{t^{-\al}}{\Gamma(1-\al)}e^{-\gamma t}, \hd l(t) = \frac{t^{\al-1}}{\Gamma(\al)}e^{-\gamma t} + \gamma \izt e^{-\gamma \tau}\frac{\tau^{\al-1}}{\Gamma(\al)}d\tau, \hd \hd \gamma \geq 0, \hd \al \in (0,1),
\]
\item the distributed order fractional derivative
\[
k(t) = \izj \frac{t^{-\al}}{\Gamma(1-\al)} d\mu(\al), \hd l(t) = \frac{1}{\pi}\izi e^{-pt}\frac{\izj p^{\al} \sin (\pi \al)\dd}{(\izj p^{\al} \sin (\pi \al)\dd)^{2} + (\izj p^{\al} \cos (\pi \al)\dd)^{2}}dp,
\]
where the measure $\mu$ is of the form $d\mu=\sum_{n=1}^M q_nd\delta(\cdot-\alpha_n)+wd\nu_1$, $\alpha_n\in (0,1)$, $q_n\ge 0$ for all $n=1,\ldots,M$,
$w\in L_1((0,1))$ is nonnegative and $\mu \not \equiv 0$ ($\delta(\cdot - \al_{n})$ is the Dirac measure at $\alpha_n$ and $\nu_1$ denotes the one-dimensional Lebesgue measure)
\item the operator with $l$ decaying exponentially
\[
k(t) =\frac{t^{\al-1}}{\Gamma(\al)}e^{-\gamma t} + \gamma \izt e^{-\gamma \tau}\frac{\tau^{\al-1}}{\Gamma(\al)}d\tau, \hd l(t) = \frac{t^{-\al}}{\Gamma(1-\al)}e^{-\gamma t}, \hd \hd \gamma > 0, \hd \al \in (0,1),
\]
\item the distributed kernel case with switched kernels under the assumption $\supp \mu \subset [0,\al_{*}], \al_{*} \in (0,1)$
\[
k(t) = \frac{1}{\pi}\izi e^{-pt}\frac{\izj p^{\al} \sin (\pi \al)\dd}{(\izj p^{\al} \sin (\pi \al)\dd)^{2} + (\izj p^{\al} \cos (\pi \al)\dd)^{2}}dp, \hd l(t) = \izj \frac{t^{-\al}}{\Gamma(1-\al)} d\mu(\al).
\]
\end{itemize}

Let us next introduce the basic assumptions imposed on $A,u_0$ and $f$. Denoting $\Omega_T=(0,T)\times \Omega$ we  assume that
\begin{itemize}
\item [{(H1)}] $A\in L_\infty(\Omega_T;\iR^{N\times
N})$, and
$\sum_{i,j=1}^N|a_{ij}(t,x)|^2\le \Lambda^2,\quad \mbox{for
a.a.}\;(t,x)\in \Omega_T.$

\item [{(H2)}] There exists $\nu>0$ such that
\[
\big(A(t,x)\xi|\xi\big)\ge \nu|\xi|^2,\quad\mbox{for a.a.}\;
(t,x)\in\Omega_T,\; \mbox{and all}\;\xi\in \iR^N.
\]
\item [{(H3)}] $u_0\in L_2(\Omega)$, $f \in L_{q_1}((0,T);L_{q_{2}}(\Omega))$, where
\[
\frac{p_{0}'}{q_{1}} + \frac{N}{2q_{2}} = 1-d,
\]
\[
q_1 \geq \frac{p_{0}'}{(1-d)},\hd q_2 \geq \frac{N}{2(1-d)}, \hd d \in (0,1)  \hd \m{ for } \hd N \geq 2,
\]
and
\[
 q_1 \in\big[\frac{p_0'}{1-d},\frac{2p'_0}{(1-2d)}\big], \hd q_2 \geq 1, \hd d \in (0,\frac{1}{2}) \hd  \m{ for } \hd N = 1.
\]
Here $p_0$ comes from (\ref{ak1}) and
$\frac{1}{p_0}+\frac{1}{p_0'}=1$.
\end{itemize}

We say that a function $u$ is a {\em weak solution (subsolution,
supersolution)} of equation (\ref{MProb}) in $\Omega_T$, if $u$ belongs to
the space
\[
Z:=\{v \in L_{2}((0,T);H^1_2(\Omega)):\, k*(u-u_{0}) \in C([0,T];L_{2}(\Omega))\;\mbox{and}\; (k*u)|_{t=0} = 0\}
\]
and for any nonnegative test function
\[
\eta\in \oH^{1,1}_2(\Omega_T):=H^1_2((0,T);L_2(\Omega))\cap
L_2((0,T);\oH^1_2(\Omega)) \quad\quad
\Big(\oH^1_2(\Omega):=\overline{C_0^\infty(\Omega)}\,{}^{H^1_2(\Omega)}\Big)
\]
with $\eta|_{t=T}=0$ we have
\begin{equation} \label{BWF}
\int_{0}^{T} \int_\Omega \Big(-\eta_t [k\ast (u-u_0)]+
(ADu|D \eta)\Big)\,dxdt= \,(\le,\,\ge )\, \int_{0}^{T} \int_\Omega f \eta\, dxdt.
\end{equation}

To formulate our main result, let $B(x,r)$ denote the open ball with
radius $r>0$ centered at $x\in \iR^N$. The $N$-dimensional Lebesgue measure will be denoted by
 $\nu_{N}$. We set
\eqq{
r_0 = \left(\izi l(t)dt \right)^{\frac{1}{2}}.
}{r0}
Note that if $l$ is not integrable on $\mathbb{R}_{+}$, then $r_0 = \infty$. If we write $\frac{1}{r_0}$ in the following, we use the convention $\frac{1}{\infty} = 0$.
Let $\Phi$ be the function from Lemma~\ref{fi},
which is defined by \eqref{PHIEQUA}. For $\delta\in(0,1)$, $t_0\ge 0$,
$\tau>0$, $r \in (0,r_0/2)$ and $x_0\in \iR^N$, define the boxes
\eqnsl{
Q_-(t_0,x_0,r, \delta)&=(t_0,t_0+\delta\tau \Phi(2r))\times B(x_0,\delta r),\\
Q_+(t_0,x_0,r,\delta)&=(t_0+(2-\delta)\tau \Phi(2r),t_0+2\tau
\Phi(2r))\times B(x_0,\delta r).
}{defQpm}

\noindent Our result on the weak Harnack inequality for nonnegative supersolutions reads as follows.

\begin{satz} \label{localweakHarnack}
Let $T>0, N \geq 1$, $\Omega\subset \iR^N$ be a bounded
domain. Suppose the assumptions (H1)--(H3) and
(\ref{pc})-- (\ref{ak2}) are satisfied.
Let $\delta\in(0,1)$ and $\Phi \in C([0,r_0))\cap C^{1}((0,r_0))$ be the function
from Lemma~\ref{fi}, which satisfies \eqref{PHIEQUA}.
There exist constants  $r^* \in (0,r_0/2)$ and  $\tau^{*} \in (0,1)$ with $ \tau^{*}=\tau^{*}(\delta,q_1, p_0, N,\nu,\Lambda,\overline{c}) > 0$ such that for every $\tau \in (0,\tau^{*}]$ and for any $0<p < \frac{2p_0+N(p_0-1)}{2+N(p_0-1)}$,  $t_0\ge 0$, $r\in (0,r^*]$  with $t_{0}+2\tau \Phi(2r) \leq T$ and any ball $B(x_{0},r) \subset \Omega$ and any nonnegative weak
supersolution $u$ of (\ref{MProb}) in $(0,t_0+2\tau
\Phi(2r))\times B(x_0, r)$ with $u_0\ge 0$ in $B(x_0,
r)$ there holds
\begin{equation} \label{localwHarnackF}
\Big(\frac{1}{\nuNj\big(Q_-\big)}\,\int_{Q_-}u^p\,d\nuNj\Big)^{1/p}
\le C \left( \einf_{Q_+} u + r^{2-\frac{N}{q_{2}}} (\Phi(2r))^{-\frac{1}{q_{1}}}\norm{f}_{L_{q_{1}}((0,T);L_{q_{2}}(\Omega))}\right),
\end{equation}
where $Q_-=Q_-(t_0,x_0,r, \delta)$, $Q_+=Q_+(t_0,x_0,r, \delta)$ and $C=C(\nu,\Lambda,\delta,\tau,N,p,p_0,q_1,q_2,\overline{c},\tilde{c})$.
\end{satz}

\begin{bemerk1}
Since in Theorem \ref{localweakHarnack} we work with supersolutions, it is enough to assume only that $f^{-} \in L_{q_{1}}((0,T);L_{q_{2}}(\Omega))$, where  $f^{-}$ denotes the negative part of $f$. One may replace $f$ by $f^{-}$ in (\ref{localwHarnackF}).
\end{bemerk1}
\begin{bemerk1}
Note that in the single order case \eqref{sing}, we can choose in Theorem \ref{localweakHarnack} any $p_0<\frac{1}{1-\alpha}$.
Thus the weak Harnack inequality holds for all $0<p<\frac{2+N\alpha}{2+N\alpha-2\alpha}$, in accordance with the main result from \cite{base}, which is optimal with  respect to the critical exponent. Note that as $\alpha\to 1$ we obtain the number $1+\frac{2}{N}$, which is the sharp exponent in the classical parabolic case. The critical exponent for $p$ in Theorem \ref{localweakHarnack} further coincides with the one
from \cite{harnackdistr} in the distributed order case, see also Example \ref{distexamp}.
\end{bemerk1}

Similarly as in \cite{harnackdistr}, we apply the weak Harnack estimate to deduce H\"older regularity of weak solutions to (\ref{MProb}).

\begin{satz}\label{holder}
Let $T>0, N \geq 1$ and $\Omega\subset \iR^N$ be a bounded
domain. Suppose the assumptions (H1)--(H3) and
(\ref{pc}) -- (\ref{asholder}) are satisfied and assume that $u_{0} \in L_{\infty}(\Omega)$. If $u$ is a bounded weak solution to \eqref{MProb} in $\Omega_T$,
then for any $V \subset \Omega_{T}$ separated from the parabolic  boundary of $\Omega_{T}$ by a positive distance, there exist $C > 0$ and $\ve \in (0,1)$ depending only on $V$, $\Lambda$, $\nu$, $N$, $q_1$, $q_2$, $p_0$, $\overline{c}$ and  $\tilde{c}$ such that
\eqq{
\|u\|_{C^{0,\ve}(V)} \leq C\big(\|u\|_{L_{\infty}(\Omega_{T})} +\|u_{0}\|_{L_{\infty}(\Omega)} +\|f\|_{ L_{q_1}((0,T);L_{q_2}(\Omega))}\big).
}{holderkoniec}
\end{satz}
\nic{\eqq{
\partial_t (k*(u-u_{0}))-\divv\,\big(A(t,x)Du\big)=f,\quad t\in (0,T),\,x\in
\Omega
}{holdeq}
}

As another application of Theorem \ref{localweakHarnack} we immediately obtain the strong maximum principle for weak subsolutions. The proof is the same as for Theorem 3.4 in \cite{harnackdistr}.
\begin{korollar} \label{strongmax}
Let $T>0$, $N \geq 1$ and $\Omega\subset \iR^N$ be a bounded
domain. Suppose the assumptions (H1)--(H3) and (\ref{pc})--(\ref{ak2}) are fulfilled. Let $u\in
Z$ be a weak subsolution of (\ref{MProb}) in $\Omega_T$ with $f\equiv 0$ such that $0\le \esup_{\Omega_T}u<\infty$ and $\esup_{\Omega}
u_0\le \esup_{\Omega_T}u$. Then, if for some cylinder
$Q=(t_0,t_0+\tau \vdr)\times B(x_0,r/2)\subset \Omega_T$ with
$t_0,\tau,r>0$ and $\overline{B(x_0,r)}\subset \Omega$ we have
\begin{equation} \label{strrel}
\esup_{Q}u \,=\,\esup_{\Omega_T}u,
\end{equation}
then the function $u$ is constant on $(0,t_0)\times \Omega$.
\end{korollar}

Let us now give some motivation for defining the function $\Phi$ via the relation \eqref{PHIEQUA}. Let $r>0$ be sufficiently small and consider the equation
\begin{equation}
\partial_{t}(k*u) - \Delta u = g,\quad t\in (0,\Phi(r))\times B(x_0,r),
\end{equation}
together with the boundary condition
$u|_{\partial B(x_0,r)} = 0$, $t\in (0,\Phi(r))$.
Here $g$ might also contain a term involving an initial value for $u$.
Testing the equation with $u$, integrating over $B_r:=B(x_0,r)$ and by parts, applying the fundamental identity (see \eqref{fundidentity} below) and convolving  with $l$ formally leads to
\begin{equation} \label{moti1}
\frac{1}{2}\int_{B_r} u^{2} \,dx + l* \int_{B_r} |Du|^{2}\, dx
\leq l*\int_{B_r} g u\, dx.
\end{equation}
We will show that after integration in time, the two terms on the left-hand side of \eqref{moti1} are up to some sharp inequalities (in different directions) ''in balance'' with respect to the behaviour in $r$. In fact, integrating \eqref{moti1} over
$(0,\Phi(r))$, the first term becomes $J_1:=\frac{1}{2}\|u\|_{L_2((0,\Phi(r))\times B_r)}^2$ while for the second we obtain
\[
(1\ast l* \int_{B_r} |Du|^{2}\, dx)(\Phi(r)),
\]
which can be estimated from below by the Poincar\'e inequality on balls by a term of the form
\[
J_2:=\frac{C}{r^2}(1\ast l* \int_{B_r} u^{2}\, dx)(\Phi(r))
\]
with $C>0$ only depending on $N$. To compare $J_2$ with
$J_1$, note that by Young's inequality and the relation \eqref{PHIEQUA} we have
\[
J_2 \leq (1*l)(\Phi(r))\frac{C}{r^2}\int_0^{\Phi(r)}\int_{B_r} u^{2}\, dx\,dt
= 2C J_1.
\]
This should provide some motivation for the chosen geometry of the local boxes. The use of certain inequalities in these considerations seems to be unavoidable as the equation lacks a scaling property.

The paper is organised as follows. Section 2 is of preliminary character. We provide important properties of the involved kernels and the function $\Phi$ and recall
some auxiliary tools such as Moser's iteration technique, the crucial Bombieri-Giusti lemma as well as the fundamental identity for operators of the form $\frac{d}{dt}k*\cdot$. Section 3 contains the proof of the weak Harnack inequality and Section 4 is devoted to H\"older continuity
of weak solutions. Finally, Section 5 is an appendix which
deals with several examples of admissible kernels.



\section{Preliminaries and auxiliary results}
\subsection{Properties of the kernels and cylinders}
In the whole paper we work in the general framework of kernels satisfying (\ref{pc}).
In this subsection we look at various properties of the $\mathscr{PC}$ pair $(k,l)$ which are induced by the additional assumptions (\ref{ak1})-(\ref{asholder}). Note that the majority of the results in this part are alone a
consequence of (\ref{ak1}), which is thus a crucial assumption in our approach. Furthermore, we introduce the function $\Phi$ which induces the shape of time-space cylinders suitable for the Giorgi-Nash-Moser estimates and derive further properties of $\Phi$.

We begin with a simple remark.
\begin{bemerk1} \label{RemarkAllp}
The assumption (\ref{ak1}) together with H\"older's inequality implies
\[
\frac{1}{t} \int_{0}^{t}\big(l(s)\big)^{p} ds \leq \overline{c} \big(l(t)\big)^{p},\quad 0<t\le t_0, \hd \m{ for } \hd p \in [1,p_0].
\]
\end{bemerk1}
\begin{lemma}\label{kernels}
Assume only (\ref{pc}). Then
\[
k(t) \leq  k_1(t) =  \frac{1}{(1*l)(t)}, \quad t>0.
\]
Furthermore, if in addition (\ref{ak1}) holds, we have
\[
(1*k)(t) \leq \overline{c} t k_1(t), \quad t \in (0,t_0),
\]
where $t_0$ and $\overline{c}$ come from the assumption (\ref{ak1}).
\end{lemma}
\begin{proof}
To show the first estimate we note that since $k*l=1$ and $k$ is nonincreasing we have
\[
k(t) \izt l(\tau)d\tau \leq k*l = 1.
\]
Since $l$ is also nonincreasing, we obtain
\[
l(t)\izt k(s)ds \leq 1,
\]
thus, applying  the assumption (\ref{ak1}) with $p=1$ we obtain  for $t < t_0$
\[
(1*k)(t) \leq \frac{1}{l(t)} \leq \frac{\overline{c}t}{(1*l)(t)}.
\]
\end{proof}

We next present a simple but useful lemma.
\begin{lemma}\label{ba}
For every $x,y > 0$ we have
    \[
k_1(xy) \leq \max\{1,y^{-1}\}k_1(x).
    \]
\end{lemma}
\begin{proof}
If $y \in (0,1]$ we have
\[
k_1(xy) = \frac{1}{ (1*l)(xy)} = \frac{1}{y \int_{0}^{x}l(ys)ds} \leq \frac{1}{y(1*l)(x)}=y^{-1}k_{1}(x).
\]
On the other hand, the function $k_1$ is decreasing, because $\dot{k}_{1}=-lk^{2}_{1}$, and thus if $y>1$ then trivially $k_1(xy) \le k_1(x)$.
\end{proof}
\begin{lemma}\label{lpoint}
Assume (\ref{ak1}). Then for every $t \leq t_0$ and $a \in (0,1)$ there holds
\[
l(at) \leq \overline{c}a^{-1}l(t).
\]
\end{lemma}
\begin{proof}
Indeed, from monotonicity of $l$ and Remark \ref{RemarkAllp} we infer that
\[
l(at) =\frac{1}{at} \int_{0}^{at} l(s)ds\leq \frac{1}{at} (1*l)(at) \leq \frac{1}{at} (1*l)(t) \leq \overline{c} a^{-1}l(t).
\]
\end{proof}
Now, let us show that (\ref{ak1}) implies that the value of $l$ integrated over the smaller interval may be controlled by a small constant multiplied by the value of $l$ integrated over the bigger interval.
\begin{lemma} \label{ak3}
Assume (\ref{ak1}). Then for every $x < 1$ and $y \leq t_0$ there holds
\[
(1*l)(xy) \leq (\overline{c})^{\frac{1}{p_0}} x^{\frac{1}{p'_0}}(1*l)(y),
\]
where $\overline{c} > 0$ is the constant from (\ref{ak1}).
\end{lemma}
\begin{proof}
Using H\"older's inequality, the assumption (\ref{ak1}) and monotonicity of $l$ we see that
\[
(1*l)(xy) \leq (xy)^{\frac{1}{p'_0}} \left(\int_{0}^{y} (l(s))^{p_0}ds\right)^{\frac{1}{p_0}} \leq (\overline{c})^{\frac{1}{p_0}}x^{\frac{1}{p'_0}}yl(y) \leq (\overline{c})^{\frac{1}{p_0}} x^{\frac{1}{p'_0}} (1*l)(y).
\]
\end{proof}
\begin{bemerk1}\label{ak31}
    From Lemma \ref{ak3} one may further deduce that for every $x < 1$, $y \leq t_0$ and $p \in (1,p_0)$
    \[
    \left(\int_{0}^{xy}(l(s))^{p}ds\right)^{\frac{1}{p}} \leq \overline{c}^{\frac{1}{p}+\frac{1}{p_0}} x^{\frac{1}{p'_0}-\frac{1}{p'}}
    \left(\int_{0}^{y}(l(s))^{p}ds\right)^{\frac{1}{p}}.
    \]
    Indeed, by Remark \ref{RemarkAllp} and using monotonicity of $l$, Lemma \ref{ak3} and H\"older's inequality we have
    \[
    \left(\int_{0}^{xy}(l(s))^{p}ds\right)^{\frac{1}{p}} \leq \overline{c}^{\frac{1}{p}} (xy)^{\frac{1}{p}}l(xy) \leq \overline{c}^{\frac{1}{p}} (xy)^{\frac{1}{p}-1}(1*l)(xy) \leq \overline{c}^{\frac{1}{p}+\frac{1}{p_0}} (xy)^{-\frac{1}{p'}}x^{\frac{1}{p'_0}}(1*l)(y)
    \]
    \[
    \leq  \overline{c}^{\frac{1}{p}+\frac{1}{p_0}}x^{\frac{1}{p'_0}-\frac{1}{p'}} \left(\int_{0}^{y}(l(s))^{p}ds\right)^{\frac{1}{p}}.
    \]
\end{bemerk1}

The next lemma introduces the function $\Phi$ which defines the shape of our time-space cylinders. The proof of Lemma \ref{fi} follows the lines of the proof of \cite{harnackdistr}[Lemma 2.4]. Since the result is crucial and the argument is short, we provide a proof for the reader's convenience.

\begin{lemma}\label{fi}
Let $r_0\in (0,\infty]$ be defined by (\ref{r0}). Then there exists a unique strictly increasing function $\Phi \in C([0,r_0))\cap C^{1}((0,r_0))$ such that $\Phi(0) = 0$, $\lim\limits_{r\rightarrow r_0}\Phi(r) = \infty$ and
\eqq{
 \kjr = r^{-2}.
}{zn1}
\end{lemma}
\begin{proof}
We note that $k_{1}$ is a smooth, decreasing function on $(0,\infty)$.
Furthermore,
 $\lim\limits_{t\rightarrow 0^{+}}k_{1}(t) = \infty$ and $\lim\limits_{t\rightarrow \infty}k_{1}(t) = r_0^{-2}$.

Thus, by the Darboux property, for every $r\in (0,r_0)$ there exists $\Phi(r) >0$ such that
\[
k_{1}(\Phi(r)) = r^{-2}
\]
and $\Phi(r)$ is uniquely determined because $\dot{k}_{1}<0$. In particular, from the implicit function theorem we deduce that  $\Phi \in C^{1}((0,r_0))$ and $\Phi'(r)>0$. Moreover,
\[
\lim_{r\rightarrow r_0}\Phi(r) = \lim_{r\rightarrow r_0} k_1^{-1}(r^{-2}) = \infty.
\]
If $y_{0}:=\inf\limits_{r>0}\Phi(r)=\lim\limits_{r\rightarrow 0^{+}} \Phi(r)$ were positive, then we would have
\[
k_{1}(y_{0})= \lim_{r\rightarrow 0^{+}}k_{1}(\Phi(r))=\lim_{r\rightarrow 0^{+}} \frac{1}{r^{2}}=\infty,
\]
a contradiction. Hence $y_{0}=0$.
\end{proof}

Next, we establish further properties of the function $\Phi$, which will be needed later.

\begin{prop}\label{philambda}
    For every $r \in (0,r_0)$ and for every $\lambda \in (0,1]$ there holds
    \[
    \Phi(\lambda r) \leq \lambda^2 \vr.
    \]
\end{prop}
\begin{proof}
    From (\ref{zn1}) and the fact that $l$ is nonincreasing we have
    \begin{align*}
    k_1(\Phi(\lambda r)) & = \lambda^{-2}r^{-2} = \lambda^{-2}k_1(\vr) = \frac{1}{\lambda^{2}(1*l)(\vr)}\\
    & \geq \frac{1}{\lambda^{2}\int_{0}^{\vr}l(\lambda^2 t)dt} = \frac{1}{\int_{0}^{\lambda^{2}\vr}l(s)ds} = k_1(\lambda^2 \Phi( r)).
    \end{align*}
    Since $k_1$ is decreasing  we obtain the claim.
\end{proof}

The next lemma is a key tool in the proof of weak Harnack estimate. It corresponds to \cite{harnackdistr}[Lemma 2.5].

\begin{lemma}\label{scaling}
Assume (\ref{ak1}). Then there exists $r^{*} \in (0,r_0/2)$ such that $\Phi(2r^{*})~\leq~\min\{1,t_{0}\}$ and for every  $1 \leq  p \leq p_0$ and for every $r \in (0,r^{*}]$, there holds

\[
\norm{l}_{L_{p}(0,\Phi(2r))}^{p}(\Phi(2r))^{p-1} \leq C r^{2p},
\]
where $C=4^{p_0}\overline{c}$ and $\overline{c}$ comes from the assumption (\ref{ak1}).
\end{lemma}
\begin{proof}
We take $r^{*} \in (0,r_0/2)$ such that $\Phi(2r^{*})~\leq~1$ and  $\Phi(2r^{*})~\leq~t_{0}$, where $t_0$ comes from the assumption~(\ref{ak1}). Then, by (\ref{ak1}) and monotonicity of $l$ we have
\[
\int_{0}^{\vdr}(l(s))^{p}ds \leq \overline{c} \vdr (l(\vdr))^{p} \leq \overline{c} \vdr \frac{1}{(\vdr)^{p}}((1*l)(\vdr))^{p} = \overline{c} (\vdr)^{1-p}(k_1(\vdr))^{-p}.
\]
We recall that thanks to (\ref{zn1})
\[
(k_1(\vdr))^{-p} = r^{2p}4^{p}.
\]
Thus
\[
\norm{l}_{L_{p}(0,\vdr)}^{p}  \leq 4^{p_0} \overline{c} r^{2p} (\vdr)^{1-p} ,
\]
which finishes the proof.
\end{proof}

\nic{
\begin{lemma} \label{ak3}
For every $x < 1$ and $y \leq \vdr$, with $r \leq r^{*}(p_0)$, where $r^{*}(p_0)$ comes from Lemma \ref{scaling}, there exists positive constant such that $\bar{c} = \bar{c}(p_0)$
\[
l(xy) \leq \overline{c} x^{-\frac{1}{p_0}}l(y).
\]
\end{lemma}
\begin{proof}
We note that from (\ref{ak1}) we have
\[
l(xy) \leq l(y) - \int_{xy}^{y}\dot{l}(s)ds \leq l(y) + \frac{\delta}{p_0}\int_{xy}^{y} l(s) s^{-1}ds \leq
l(y) + \frac{\delta}{p_0} \left(\int_{xy}^{y} (l(s))^{p_0}ds\right)^{\frac{1}{p_0}} \left(\int_{xy}^{y} s^{-p'_0}ds \right)^{\frac{1}{p'_0}}
\]
Applying (\ref{lmon}) we have
\[
l(xy) \leq l(y) + \frac{\delta}{p_0}c y^{\frac{1}{p_0}}l(y) \frac{1}{(1-p_0')^{\frac{1}{p'_0}}}[(xy)^{1-p'_0} - y^{1-p'_0}]^{\frac{1}{p'_0}} =  l(y)\left(1 + \frac{\delta}{p_0}c  \frac{1}{(1-p_0')^{\frac{1}{p'_0}}}[x^{1-p'_0} -1]^{\frac{1}{p'_0}}\right) \leq \bar{c}x^{-\frac{1}{p_0}}l(y).
\]
\end{proof}
}

\begin{lemma}\label{fixy}
Assume (\ref{ak1}). Then for any  $x,y > 0$ such that $xy<r_0$, $y >1$ and $\Phi(xy) \leq t_0$, where $t_0$ comes from (\ref{ak1}),  there holds
\[
\Phi(xy) \leq \overline{c}^{\frac{1}{p_0-1}} y^{2p'_0} \Phi(x),
\]
where $\overline{c} \geq 1$ is the constant from (\ref{ak1}) .
\end{lemma}

\begin{proof}
Since $k_1$ is decreasing it is enough to prove
\[
k_{1}\left(\frac{\Phi(xy)}{\overline{c}^{\frac{1}{p_0-1}}y^{2p'_0}}\right) \geq k_{1}(\Phi(x)).
\]
We apply Lemma~\ref{ak3} with $x=\overline{c}^{-\frac{1}{p_0-1}}y^{-2p'_0}, y =\Phi(xy) $ and (\ref{zn1}) to the result
\[
k_{1}\left(\frac{\Phi(xy)}{\overline{c}^{\frac{1}{p_0-1}}y^{2p'_0}}\right) = \frac{1}{(1*l)(\overline{c}^{-\frac{1}{p_0-1}}y^{-2p'_0} \Phi(xy))}  \geq \frac{1}{\overline{c}^{\frac{1}{p_0}}\overline{c}^{-\frac{1}{(p_0-1)p'_0}}(y^{2})^{-\frac{p'_0}{p'_0}} (1*l)(\Phi(xy)) }
\]
\[
= y^{2}k_1(\Phi(xy)) = k_1(\Phi(x)),
\]
which finishes the proof.
\end{proof}

\begin{prop}\label{estiprop}
Assume (\ref{ak1}). Then there exists a constant $c \in (0,1)$ such that
\[
\vdr \geq c r^{2p'_0}, \quad r \in (0,r^*),
\]
where $r^*$ is the number from Lemma \ref{scaling}.
\end{prop}
\begin{proof}
Since $k_1$ is decreasing and in view of (\ref{zn1}), it is enough to show that there exists a positive $c\in (0,1)$ such that
\eqq{
\frac{1}{4} r^{-2} \leq k_1(c r^{2p'_0}).
}{estipropc1}
We note that by Lemma \ref{ak3} and H\"older's inequality
\[
(1*l)(cr^{2p'_0}) \leq \overline{c}^{\frac{1}{p_0}}c^{\frac{1}{p'_0}} (1*l)(r^{2p'_0})\leq  \overline{c}^{\frac{1}{p_0}}c^{\frac{1}{p'_0}}\norm{l}_{L_{p_0}((0,1))}r^{2},
\]
thus
\[
 k_1(cr^{2p'_0}) \geq  \overline{c}^{-\frac{1}{p_0}}c^{-\frac{1}{p'_0}}\norm{l}^{-1}_{L_{p_0}((0,1))}r^{-2},
\]
and (\ref{estipropc1}) is satisfied with any $0<c\leq \min\{1,(4\overline{c}^{-\frac{1}{p_0}}\norm{l}^{-1}_{L_{p_0}((0,1))})^{p'_0}\}$.
\end{proof}

\begin{bemerk1}
We note that the assumption (\ref{ak2}) is equivalent to
\eqq{k(x)-k(y) \geq \tilde{c} k(x)\frac{y-x}{y} \hd \m{ for every } \hd 0<x<y<\tilde{t}_0.}{ak2prim}
Indeed, integrating the inequality in (\ref{ak2}) from $x$ to $y$ gives
\[
-k(y)y+k(x)x + \int_{x}^{y}k(t)dt \geq \tilde{c} \int_{x}^{y}k(t)dt.
\]
Adding and subtracting $yk(x)$ leads to
\[
(1-\tilde{c})\int_{x}^{y}k(t)dt + y(k(x)-k(y)) \geq k(x)(y-x)
\]
and since $k$ is nonincreasing we obtain
\[
(1-\tilde{c})k(x)(y-x)+ y(k(x)-k(y)) \geq k(x)(y-x),
\]
which gives (\ref{ak2prim}). To show that (\ref{ak2prim}) implies (\ref{ak2}) it is enough to divide (\ref{ak2prim}) by $y-x$ and pass to the limit $y \rightarrow x$.

\end{bemerk1}

\subsection{ Moser iterations and embeddings}

Below we present two Moser iteration lemmas and an important abstract lemma by Bombieri and Giusti. The proofs of the Moser iteration lemmas may be found in  \cite{base}[Lemma 2.1, Lemma 2.2], \cite{harnackdistr}[Lemma 2.7, Lemma 2.8], (see also \cite{AS}, \cite{CZ}). For the proof of Lemma \ref{abslemma} we refer to \cite{BomGiu}, \cite[Lemma
2.2.6]{SalCoste} and \cite[Lemma 2.6]{CZ}.

We introduce the following notation. Let $U_\sigma$, $0<\sigma\le 1$ be a collection of measurable subsets of a fixed finite measure space $U_1$
endowed with the measure $\mu$, such that $U_{\sigma'}\subset
U_\sigma$ if $\sigma'\le \sigma$. For $p\in (0,\infty)$ and
$0<\sigma\le 1$, by $L_p(U_\sigma)$ we mean the Lebesgue space
$L_p(U_\sigma,d\mu)$ of all $\mu$-measurable functions
$f:U_\sigma\rightarrow \iR$ with
$\|f\|_{L_p(U_\sigma)}:=(\int_{U_\sigma}|f|^p\,d\mu)^{1/p}<\infty$.
\begin{lemma} \label{moserit1}
Let $\kappa>1$, $\bar{p}\ge 1$, $C>0$, and $a>0$. Suppose
$f$ is a $\mu$-measurable function on $U_1$ such that
\begin{equation} \label{mositer1}
\|f\|_{L_{\gamma\kappa}(U_{\sigma'})}\le
\Big(\frac{C(1+\gamma)^{a}}{(\sigma-\sigma')^{a}}\Big)^{1/\gamma}\,\|f\|_{L_{\gamma}(U_{\sigma})},
\quad 0<\sigma'<\sigma\le 1,\;\gamma>0.
\end{equation}
Then there exists a constant $M=M(a,\kappa,\bar{p})>0$  such that
\[
\esup_{U_{ \varsigma  }}{|f|} \le
\Big(\frac{M C^{\frac{\kappa}{\kappa-1}}}{(1-\varsigma)^{a_0}}\Big)^{1/p}
\|f\|_{L_{p}(U_1)}\quad \mbox{for all}\;\;\varsigma\in(0,1),\;p\in
(0,\bar{p}] ,\]
where $a_{0}= \frac{a \kappa }{\kappa -1}$.
\end{lemma}


$\mbox{}$

\begin{lemma} \label{moserit2}
Assume that $\mu(U_1)\le 1$. Let $\kappa>1$, $0<p_0<\kappa$, and
$C>0,\,a>0$. Suppose $f$ is a $\mu$-measurable function on
$U_1$ such that
\begin{equation} \label{mositer2}
\|f\|_{L_{\gamma\kappa}(U_{\sigma'})}\le
\Big(\frac{C}{(\sigma-\sigma')^{a}}\Big)^{1/\gamma}\,\|f\|_{L_{\gamma}(U_{\sigma})},
\quad 0<\sigma'<\sigma\le 1,\;0<\gamma\le \frac{p_0}{\kappa}<1.
\end{equation}
Then
\[ \|f\|_{L_{p_0}(U_{\varsigma})}\le
\Big(\frac{M}{(1-\varsigma)^{a_0}}\Big)^{1/p-1/p_0}
\|f\|_{L_{p}(U_1)}\quad \mbox{for all}\;\;\varsigma\in(0,1),\;p\in
(0,\frac{p_0}{\kappa}],\]
where $M=C^{\frac{\kappa (\kappa+1)}{\kappa-1}}\cdot 2^{\frac{a \kappa^{3}}{(\kappa-1)^{3}}}$ and $a_{0}= \frac{a \kappa (\kappa+1)}{\kappa-1}$.
\end{lemma}


$\mbox{}$

\begin{lemma} \label{abslemma}
Let $\delta,\,\eta\in(0,1)$, and let $\gamma,\,C$ be positive
constants and $0<\beta_0\le \infty$. Suppose $f$ is a positive
$\mu$-measurable function on $U_1$ which satisfies the following two
conditions:

(i)
\[
\|f\|_{L_{\beta_0}(U_{\sigma'})}\le
[C(\sigma-\sigma')^{-\gamma}\mu(U_1)^{-1}]^{1/\beta-1/\beta_0}\|f\|_{L_{\beta}(U_{\sigma})},
\]
for all $\sigma,\,\sigma',\,\beta$ such that $0<\delta\le
\sigma'<\sigma\le 1$ and $0<\beta\le \min\{1,\eta\beta_0\}$.

(ii)
\[
\mu(\{ \log f>\lambda\}) \le C\mu(U_1)\lambda^{-1}
\]
for all $\lambda>0$.

Then
\[
\|f\|_{L_{\beta_0}(U_{\delta})}\le M \mu(U_1)^{1/\beta_0},
\]
where $M$ depends only on $\delta,\,\eta,\,\gamma,\,C$, and
$\beta_0$.
\end{lemma}

\nic{
\begin{lemma} \label{abslemma}
Let $\delta,\,\eta\in(0,1)$, and let $\gamma,\,C$ be positive
constants and $0<\beta_0\le \infty$, $\beta_1,\beta_2 \geq 1$. Suppose $f$ is a positive
$\mu$-measurable function on $U_1$ which satisfies the following two
conditions:

(i)
\[
\|f\|_{L_{\beta_0 \beta_1,\beta_0 \beta_2}(U_{\sigma'})}\le
[C(\sigma-\sigma')^{-\gamma}|U_1|^{-1}]^{1/\beta-1/\beta_0}\|f\|_{L_{\beta \beta_1,\beta \beta_2}(U_{\sigma})},
\]
for all $\sigma,\,\sigma',\,\beta$ such that $0<\delta\le
\sigma'<\sigma\le 1$ and $0<\beta\le \min\{1,\eta\beta_0\}$.

(ii)
\[
|\{ \log f>\lambda\}| \le C|U_1|\lambda^{-1}
\]
for all $\lambda>0$.

Then
\[
\|f\|_{L_{\beta_0 \beta_1,\beta_0 \beta_2}(U_{\delta})}\le M |U_1|^{1/\beta_0},
\]
where $M$ depends only on $\delta,\,\eta,\,\gamma,\,C$, and
$\beta_0$.
\end{lemma}
}

We finish this chapter by recalling a parabolic embedding result and a weighted Poincar\'e inequality, which goes back to Moser.

The following proposition is a
consequence of the Gagliardo-Nirenberg and H\"older's inequality, cf.\ for example, \cite[Section 2]{Za2}, \cite[Proposition 2.1.]{VZd}.

\begin{prop}
Let $N \geq 1$, $T>0$ and $\Omega$ be a bounded domain in $\iR^N$ and assume that $\partial \Omega$ satisfies the property
of positive density. For $1<p\le
\infty$ we define the space
\begin{equation} \label{Vdef}
V_p:=V_p((0,T)\times \Omega)=L_{2p}((0,T);L_2(\Omega))\cap
L_2((0,T);H^1_2(\Omega)),
\end{equation}
endowed with the norm
\[
\|u\|_{V_p((0,T)\times \Omega)}:=\|u\|_{L_{2p}((0,T);L_2(\Omega))}
+\|Du\|_{L_2((0,T);L_2(\Omega))}.
\]
Then, if
\[
p'\left(1-\frac{2}{a}\right) + N\left(\frac{1}{2}-\frac{1}{b}\right) = 1,
\]
where $p'=\frac{p}{p-1}$ and
\[
a\in \left[ \frac{4p}{p+1},2p \right], \hd b\in[2, \infty] \hd \m{ for } \hd N=1,
\]
\[
 a \in (2,2p], \hd b \in [2,\infty) \m{ \hd for \hd  } N = 2,
\]
and
\[
a \in [2,2p], \hd b \in \left[2,\frac{2N}{N-2}\right] \hd \m{for} \hd N > 2,
\]
then
$V_p\hookrightarrow
L_{a}((0,T); L_{b}(\Omega))$ and there exists $C=C(N,b)$ such that
\eqq{
\norm{u}_{L_{a}((0,T), L_{b}(\Omega))} \leq C\norm{u}_{L_{2p}((0,T), L_{2}(\Omega))}^{1-\theta}\norm{D u}_{L_{2}((0,T), L_{2}(\Omega))}^{\theta}
}{sobolev}
for all $u\in V_{p}\cap L_{2}((0,T);\oH^1_2 (\Om))$
where $\theta = N(\frac{1}{2}-\frac{1}{b})$.
\end{prop}

The following result can be found in \cite[Lemma 3]{Moser64}, see
also \cite[Lemma 6.12]{Lm}.
\begin{prop} \label{WeiPI}
Let $\varphi\in C(\iR^N)$ with non-empty compact support of diameter
$d$ and assume that $0\le \varphi\le 1$. Suppose that the domains
$\{x\in\iR^N:\varphi(x)\ge a\}$ are convex for all $a\le 1$. Then
for any function $u\in H^{1}_2(\iR^N)$,
\[
\int_{\iR^N} \big(u(x)-u_\varphi\big)^2 \varphi(x)\,dx \le \,\frac{2
d^2\mu_N(\mbox{{\em supp}}\,\varphi)}{|\varphi|_{L_1(\iR^N)}}\,
\int_{\iR^N} |Du(x)|^2 \varphi(x)\,dx,
\]
where
\[
u_\varphi=\frac{\int_{\iR^N} u(x)\varphi(x)\,dx}{\int_{\iR^N}
\varphi(x)\,dx}.
\]
\end{prop}

\subsection{Approximation of the kernel $k$} \label{SecYos}
In this subsection, we briefly recall the Yosida approximation of the operator $\frac{d}{dt}k*\cdot$, cf.\ \cite{Grip1}, \cite{base}, \cite{harnackdistr}, \cite{VZ} and some important identities for this and related nonlocal operators.

Let  $1\le p<\infty$, $T>0$, and $X$ be a real Banach
space. Then the operator $B$ defined by
\[ B u=\,\frac{d}{dt}\,(k\ast u),\;\;D(B)=\{u\in L_p((0,T);X):\,k\ast u\in \mbox{}_0 H^1_p((0,T);X)\},
\]
where the zero means vanishing trace at $t=0$, is known to be
$m$-accretive in $L_p((0,T);X)$, see (\cite{Phil1}, \cite{CP},
\cite{Grip1}).
Let $h_{n}\in L_{1,\,loc}(\iR_+)$, $n \in \mathbb{N}$, be the
resolvent kernel of $nl$, that is, $h_n+nh_n\ast l=nl$
and set $k_n:=k*h_n$, $n\in \iN$. Then it is well known that $h_n$ is nonnegative, $k_n$ is nonnegative and nonincreasing and that $k_n\in H^1_1((0,T))$ for all $n\in \iN$. Furthermore, the Yosida approximation $B_n=nB(n+B)^{-1}$ of $B$ has the form
$B_n u=\frac{d}{dt}(k_n\ast u)$, in particular
$ \frac{d}{dt}(k_n*u)\rightarrow \frac{d}{dt}(k*u)$ in $L_p((0,T);X)$ as
$n\to \infty$ for all $u\in D(B)$ as well as
$h_{n}\ast f\to f$ in $L_p((0,T);X)$
as $n\to \infty$.

The "fundamental identity" for integro-differential
operators of the form $\frac{d}{dt}(k\ast u)$ (see also \cite{Za2},\cite{GLS}) plays a crucial role in all estimates for weak supersolutions of (\ref{MProb}).
Since the identity requires a certain regularity of the kernel, it is often applied with $k_n$ introduced above.

Suppose $k\in H^1_1((0,T))$ and $H\in C^1(U)$ where $U$ is an open subset of $\iR$. Then it follows
from a straightforward computation that for any sufficiently smooth
function $u$ on $(0,T)$ taking values in $U$ one has for a.a.\ $t\in (0,T)$,
\begin{align} \label{fundidentity}
H'(u(t))&\frac{d}{dt}\,(k \ast u)(t) =\;\frac{d}{dt}\,(k\ast
H(u))(t)+
\Big(-H(u(t))+H'(u(t))u(t)\Big)k(t) \nonumber\\
 & +\int_0^t
\Big(H(u(t-s))-H(u(t))-H'(u(t))[u(t-s)-u(t)]\Big)[-\dot{k}(s)]\,ds.
\end{align}

We will apply the fundamental identity in the same way as in \cite{base} and \cite{harnackdistr}, in particular we will again use the fact that the last term in (\ref{fundidentity}) is nonnegative in
case $H$ is convex and $k$ is nonincreasing. However, we point out that although, in Section \ref{logsection} we work only with convex or concave functions,
the full identity (\ref{fundidentity}) is required.

We conclude this section by recalling two simple but useful lemmas (see \cite{base}, \cite{harnackdistr}). The first lemma follows from integration by parts. Note that the identity is satisfied, since the monotonicity of $l$ implies $tl(t) \leq (1*l)(t) \rightarrow 0$ as $t\rightarrow 0$.

\begin{lemma} \label{comm}
Let $T>0$ and $l\in L_{1,loc}(\iR)\cap C^1((0,T))$ be nonnegative and nonincreasing. Suppose that $v\in
{}_0H^1_1([0,T])$ and $\varphi\in C^1([0,T])$. Then
\[
\big(l\ast(\varphi \dot{v}))(t)=\varphi(t)(l\ast
\dot{v})(t)+\int_0^t
v(\sigma)\partial_\sigma\big(l(t-\sigma)
[\varphi(t)-\varphi(\sigma)]\big)\,d\sigma,\;\;\mbox{a.a.}\;t\in
(0,T).
\]
If in addition $v$ is nonnegative and $\varphi$ is nondecreasing
there holds
\[
\big(l\ast(\varphi \dot{v}))(t)\ge \varphi(t)(l\ast
\dot{v})(t)-\int_0^t l(t-\sigma)
\dot{\varphi}(\sigma)v(\sigma)\,d\sigma,\;\;\mbox{a.a.}\;t\in
(0,T).
\]
\end{lemma}
\begin{lemma} \label{comm2}
Let $T>0$, $k\in H^1_1([0,T])$, $v\in L_1([0,T])$, and $\varphi\in
C^1([0,T])$. Then
\[
\varphi(t)\,\frac{d}{dt}\,(k\ast v)(t)=\,\frac{d}{dt}\,\big(k\ast
[\varphi v]\big)(t)+\int_0^t
\dot{k}(t-\tau)\big(\varphi(t)-\varphi(\tau)\big)v(\tau)\,d\tau,\;\;\mbox{a.a.}\;t\in
(0,T).
\]

\end{lemma}


\section{Proof of the weak Harnack inequality}
This section has a similar structure as corresponding sections in \cite{base} and \cite{harnackdistr}, where the weak Harnack inequality for single and distributed order time fractional derivatives, respectively, with $f=0$ was proven. However, as already mentioned in the introduction, the presence of the source term induces new difficulties, especially with the logarithmic estimates. It also makes Moser's already technically demanding approach even more complex. Thus, even though, the structure of the proof is the same as in previous results (\cite{base}, \cite{harnackdistr}), we present here
the whole argument except the final step, which is the same as in \cite{harnackdistr}).

\subsection{The regularized weak formulation and  time shifts}\label{SSS}

At first, we recall a lemma which provides a starting point for  {\em a priori} estimates
for weak (sub-/super-) solutions of (\ref{MProb}). It roughly states that one may replace the singular
kernel $k$ by its regular approximation $k_{n}$ ($n\in\iN$) in the weak formulation of (\ref{MProb}).
\begin{lemma} \label{LemmaReg}
Let  $T>0$, and $\Omega\subset \iR^N$ be a bounded
domain. Suppose the assumptions (H1)--(H3) are satisfied. Then $u\in
Z$ is a weak solution (subsolution, supersolution) of
(\ref{MProb}) in $\Omega_T$ if and only if for any nonnegative
function $\psi\in \oH^1_2(\Omega)$ one has
\begin{equation} \label{LemmaRegF}
\int_\Omega \Big(\psi \partial_t[k_{n}\ast
(u-u_0)]+(h_n\ast [ADu]|D\psi)\Big)\,dx\nonumber\\
=\,(\le,\,\ge)\,\int_{\Om}(h_{n}*f)\psi dx ,\quad\mbox{a.a.}\;t\in (0,T),\,n\in \iN.
\end{equation}
\end{lemma}
For a proof we refer to Lemma 3.1 in \cite{Za2}, where a slightly more
general situation is considered.

Let $u\in Z$ be a weak supersolution of (\ref{MProb}) in
$\Omega_T$ and assume that $u_0\ge 0$ in $\Omega$. Then Lemma~\ref{LemmaReg} and the positivity of $k_{n}$ imply that
\begin{equation} \label{u0weg}
\int_\Omega \Big(\psi \partial_t(k_{n}\ast u)+(h_n\ast
[ADu]|D\psi)\Big)\,dx \ge \,\int_{\Om}(h_{n}*f)\psi dx,\quad\mbox{a.a.}\;t\in (0,T),\,n\in
\iN,
\end{equation}
for any nonnegative function $\psi\in \oH^1_2(\Omega)$.

For the reader's convenience we recall the inequality which is satisfied by (in time) shifted positive weak supersolutions of (\ref{MProb}). It will be a starting point for all the estimates in this section.
We fix $t_1\in (0,T)$. For $t\in (t_1,T)$ we introduce the
shifted time $s=t-t_1$ and set $\tilde{g}(s)=g(s+t_1)$, $s\in
(0,T-t_1)$, for functions $g$ defined on $(t_1,T)$. From the
decomposition
\[
(k_{n}\ast u)(t,x)=\int_{t_1}^t
k_{n}(t-\tau)u(\tau,x)\,d\tau+\int_{0}^{t_1}
k_{n}(t-\tau)u(\tau,x)\,d\tau,\quad t\in (t_1,T),
\]
we then deduce that
\begin{equation} \label{shiftprop}
\partial_t(k_{n}\ast u)(t,x)=\partial_s(k_{n}\ast
\tilde{u})(s,x)+\int_0^{t_1}\dot{k}_{n}(s+t_1-\tau)u(\tau,x)\,d\tau.
\end{equation}
Assuming in addition that $u\ge 0$ on $(0,t_1)\times \Omega$ it
follows from (\ref{u0weg}), (\ref{shiftprop}), and the positivity
of $\psi$ and of $-\dot{k}_{n}$ that
\begin{equation} \label{shiftprob}
\int_\Omega \Big(\psi \partial_s(k_{n}\ast
\tilde{u})+\big((h_n\ast [ADu])\,\tilde{}\;|D\psi\big)\Big)\,dx \ge
\,\int_{\Om}(h_{n}*f)\,\tilde{} \, \psi dx,\quad\mbox{a.a.}\;s\in (0,T-t_1),\,n\in \iN,
\end{equation}
for any nonnegative function $\psi\in \oH^1_2(\Omega)$.

\subsection{Mean value inequalities} \label{mvi}
To shorten the notation we put $B_r(x):=B(x, r)$ for $r>0$ and $x\in \iR^N$. Recall that
$\nuN$ stands for the Lebesgue measure in $\iR^N$ and that we always assume (\ref{pc}).
Given $r \in (0,r_0/2)$ we introduce a shift of the function $u$ as $u_b = u + b$, where
\[ b = r^{2-\frac{N}{q_{2}}}(\vdr)^{-\frac{1}{q_1}}\norm{f}_{L_{q_1}((0,T);L_{q_2}(\Omega))}
\]
if $f \not\equiv 0$ and $b=\ve$ with arbitrarily fixed $\ve > 0$  otherwise.
\begin{satz} \label{superest1}
Let $T>0$ and $\Omega\subset \iR^N$ be a bounded
domain. Suppose the assumptions (H1)--(H3) and (\ref{ak1}) are satisfied. Let
$\eta>0$ and $\delta\in (0,1)$ be fixed, and let $r^{*}\in (0,r_0/2)$ be the number provided by Lemma \ref{scaling}. Then for any $0<r \leq r^{*}$, any $t_0\in(0,T]$
 with $t_0-\eta \vdr \ge 0$, any ball
$B_r(x_{0})\subset\Omega$, and any weak supersolution $u\ge 0$ of (\ref{MProb}) in $(0,t_0)\times B_r(x_{0})$ with $u_0\ge 0$
in $B_r(x_{0})$, we have
\[
\esup_{U_{\sigma'}}{u_b^{-1}} \le \Big(\frac{C \nuNj(U_1)^{-1}
}{(\sigma-\sigma')^{\tau_0}}\Big)^{1/\gamma}
\|u_b^{-1}\|_{L_{\gamma }(U_{\sigma})},\quad \delta\le \sigma'<\sigma\le
1,\; \gamma\in (0,1].
\]
Here $U_\sigma=(t_0-\sigma\eta \vdr,t_0)\times B_{\sigma r}(x_0)$,
$0<\sigma\le 1$
and $C=C(\nu,\Lambda,\delta,\eta,N,p_0,q_1,q_2,\overline{c})$ and
$\tau_0=\tau_0(N,p_0,d)$.
\end{satz}
\begin{proof}
The strategy of our proof is similar to the proof of \cite[Theorem 3.1]{base}, see also \cite[Theorem 3.1]{harnackdistr}. Concerning the inhomogeneity, which is new here, we adapt the methods introduced in \cite{AS}. Analogously to the argument in \cite{harnackdistr}, the key idea is to apply Lemma \ref{scaling}. We will use the  shorter notation $B_{r} := B_{r}(x_{0})$, because we only consider balls centered at the fixed $x_0$.

Let $r \in (0,r^{*}]$ and fix $\sigma'$ and $\sigma$ such that $\delta\le \sigma'<\sigma\le
1$. We set
$V_\rho=U_{\rho\sigma}$ for $\rho\in (0,1]$. For any $0<\rho'<\rho\le 1$, we introduce
$t_{1} = t_{0}-\rho\sigma\eta\vdr$ and $t_{2} = t_{0}-\rho'\sigma\eta\vdr$. Then $0\le t_1<t_2<t_0$. We
shift the time by setting ${s}=t-t_1$ and  $\tilde{g}(s)=g(s+t_1)$, $s\in
(0,t_0-t_1)$, for functions $g$ defined on $(t_1,t_0)$. Since
$u_0\ge 0$ in $B_r$ and $u$ is a positive weak supersolution of
(\ref{MProb}) in $(0,t_0)\times B_r$, we have (cf. (\ref{shiftprob}))
\begin{equation} \label{sup0}
\int_{B_r} \Big(v \partial_s(k_{n}\ast \tilde{u_{b}})+\big((h_n\ast
[ADu_b])\,\tilde{}\;|Dv\big)\Big)\,dx  \ge \,\int_{B_r} v (h_{n} * f)^{\tilde{}}dx,\quad\mbox{a.a.}\;s\in
(0,t_0-t_1),\,n\in \iN,
\end{equation}
for any nonnegative function $v\in \oH^1_2(B_r)$. In (\ref{sup0}) we choose the test function $v=\psi^2
\tilde{u_{b}}^{\beta}$, where $\beta<-1$ and $\psi\in C^1_0(B_{r\sigma})$ is a cut-off function which satisfies:
$0\le \psi\le 1$, $\psi=1$ in $B_{\rho' r\sigma}$, supp$\,\psi\subset B_{\rho r\sigma}$, and $|D \psi|\le 2/[r \sigma (\rho-\rho')]$. The fundamental
identity (\ref{fundidentity}) applied to $k=k_{n}$ and the
convex function $H(y)=-(1+\beta)^{-1}y^{1+\beta}$, $y>0$, implies that for a.a.\ $(s,x)\in (0,t_0-t_1)\times B_r$
\begin{align}
 -\tilde{u_{b}}^{\beta}\partial_{s}(k_{n}\ast \tilde{u_{b}}) & \ge -\,\frac{1}{1+\beta}\,\partial_{s}
(k_{n}\ast\tilde{u_{b}}^{1+\beta})+\Big(\frac{\tilde{u_{b}}^{1+\beta}}{1+\beta}\,-\tilde{u_{b}}^{1+\beta}\Big)k_{n}\nonumber\\
 & =
-\,\frac{1}{1+\beta}\,\partial_{s}
(k_{n}\ast\tilde{u_{b}}^{1+\beta})-\,\frac{\beta}{1+\beta}\,\tilde{u_{b}}^{1+\beta}
k_{n}. \label{sup1}
\end{align}
Using
\[ Dv=2\psi D\psi \,\tilde{u_{b}}^{\beta}+\beta\psi^2 \tilde{u_{b}}^{\beta-1}D \tilde{u_{b}}\]
 and (\ref{sup1}) it follows from (\ref{sup0}) that for
a.a.\ $s\in (0,t_0-t_1)$
\begin{align}
-\,\frac{1}{1+\beta}\,& \int_{B_{r\sigma}}\psi^2\partial_{s}
(k_{n}\ast\tilde{u_{b}}^{1+\beta})\,dx+|\beta|\int_{B_{r\sigma}}\big((h_n\ast
[ADu_b])\,\tilde{}\;|\psi^2 \tilde{u_{b}}^{\beta-1}D
\tilde{u_{b}}\big)\,dx \nonumber\\
\le  & \,2\int_{B_{r\sigma}}\big((h_n\ast [ADu_b])\,\tilde{}\;|\psi D\psi
\,\tilde{u_{b}}^{\beta}\big)\,dx+\,\frac{\beta}{1+\beta}\,\int_{B_{r\sigma}}\psi^2\tilde{u_{b}}^{1+\beta}
k_{n}\,dx+ \int_{B_{r\sigma}} \psi^{2}\tilde{u_b}^{\beta}(h_n*f)^{\tilde{}} dx. \label{sup2}
\end{align}

Next we choose another cut-off function $\varphi\in C^1([0,t_0-t_1])$ with the following properties: $0\le
\varphi\le 1$, $\varphi=0$ in $[0,(t_2-t_1)/2]$, $\varphi=1$ in
$[t_2-t_1,t_0-t_1]$, and $0\le \dot{\varphi}\le 4/(t_2-t_1)$.
Multiplying (\ref{sup2}) by $-(1+\beta)>0$ and by $\varphi^{2}$
and convolving the resulting inequality with $l$ yields
\begin{align}
\int_{B_{r\sigma}} & l\ast
\big(\varphi^2\partial_{s}(k_{n}\ast
[\psi^2\tilde{u_{b}}^{1+\beta}])\big)\,dx+\beta(1+\beta)\,l\ast\int_{B_{r\sigma}}\big((h_n\ast
[ADu_b])\,\tilde{}\;|\psi^2 \tilde{u_{b}}^{\beta-1}D
\tilde{u_{b}}\big)\varphi^2\,dx \nonumber\\
\le \, & \,2|1+\beta|\,l\ast\int_{B_{r\sigma}}\big((h_n\ast
[ADu_b])\,\tilde{}\;|\psi D\psi
\,\tilde{u_{b}}^{\beta}\big)\varphi^2\,dx+|\beta|\,l\ast\int_{B_{r\sigma}}\psi^2\tilde{u_{b}}^{1+\beta}
k_{n}\varphi^2\,dx \nonumber\\
& + |1+\beta| l*\int_{B_{r\sigma}} \psi^{2}\tilde{u_b}^{\beta}(h_n*f)^{\tilde{}}\varphi^2 dx, \label{sup3}
\end{align}
for a.a. $s\in (0,t_0-t_1)$. Lemma \ref{comm} implies
\begin{align}
\int_{B_{r\sigma}} l\ast &
\big(\varphi^2\partial_{s}(k_{n}\ast
[\psi^2\tilde{u_{b}}^{1+\beta}])\big)\,dx \ge \int_{B_{r\sigma}} \varphi^2
l\ast \big(\partial_{s}(k_{n}\ast
[\psi^2\tilde{u_{b}}^{1+\beta}])\big)\,dx\nonumber\\
& -2\int_0^s l(s-\tau)\dot{\varphi}(\tau)\vf(\tau)
\big(k_{n}\ast
\int_{B_{r\sigma}}\psi^2\tilde{u_{b}}^{1+\beta}\,dx\big)(\tau)\,d\tau.
\label{sup4}
\end{align}
Furthermore, in view of $l*k=1$, $k_{n}=k\ast h_n$  and
\[
k_{n}\ast [\psi^2\tilde{u_{b}}^{1+\beta}]\in
{}_0H^1_1([0,t_0-t_1];L_1(B_{r\sigma}))
\]
 we get
\begin{equation} \label{sup5}
l\ast \partial_{s}(k_{n}\ast
[\psi^2\tilde{u_{b}}^{1+\beta}])=\partial_s(l\ast
k_{n}\ast [\psi^2\tilde{u_{b}}^{1+\beta}])=h_n\ast
(\psi^2\tilde{u_{b}}^{1+\beta}).
\end{equation}
If we combine (\ref{sup3}) and (\ref{sup4}) with (\ref{sup5}), and pass to the limit with
$n$ (taking a suitable subsequence, if
necessary), we arrive at
\begin{align}
& \int_{B_{r\sigma}}\varphi^2\psi^2\tilde{u_{b}}^{1+\beta}\,dx+
\beta(1+\beta)\,l\ast\int_{B_{r\sigma}}\big(\tilde{A}D\tilde{u_{b}}|\psi^2
\tilde{u_{b}}^{\beta-1}D \tilde{u_{b}}\big)\varphi^2\,dx\nonumber\\
& \le
\,2|1+\beta|\,l\ast\int_{B_{r\sigma}}\big(\tilde{A}D\tilde{u_{b}}|\psi
D\psi
\,\tilde{u_{b}}^{\beta}\big)\varphi^2\,dx+|\beta|\,l\ast\int_{B_{r\sigma}}\psi^2\tilde{u_{b}}^{1+\beta}
k\varphi^2\,dx \nonumber\\
& \;\;\;\;+2\int_0^s l(s-\sigma)\dot{\varphi}(\sigma)\vf(\sigma)
\big(k\ast
\int_{B_{r\sigma}}\psi^2\tilde{u_{b}}^{1+\beta}\,dx\big)(\sigma)\,d\sigma   + |1+\beta|l*\int_{B_{r\sigma}} \psi^{2}\tilde{u_b}^{\beta}\tilde{f}\varphi^2\, dx,
\label{sup6}
\end{align}
for a.a.\ $s\in(0,t_0-t_1)$.

In order to estimate the second term on the left, it
is convenient to set $w=\tilde{u_{b}}^{\frac{\beta+1}{2}}$. Then we have $Dw=\frac{\beta+1}{2}
\tilde{u_{b}}^{\frac{\beta-1}{2}} D\tilde{u_{b}}$ and by the ellipticity assumption (H2)
\begin{align}
\beta(1+\beta)\,l\ast\int_{B_{r\sigma}}\big(\tilde{A}D\tilde{u_{b}}|\psi^2
\tilde{u_{b}}^{\beta-1}D \tilde{u_{b}}\big)\varphi^2\,dx & \,\ge \nu
\beta(1+\beta)\,l\ast\int_{B_{r\sigma}} \varphi^2
\psi^2\tilde{u_{b}}^{\beta-1}|D\tilde{u_{b}}|^2\,dx \nonumber\\
& \, = \,\frac{4\nu \beta}{1+\beta}\,l\ast\int_{B_{r\sigma}}\varphi^2
\psi^2|Dw|^2\,dx. \label{sup7}
\end{align}
As to the first term on the right of \eqref{sup6} note that (H1) and Young's inequality give
\begin{align}
2\big|\big(\tilde{A}D\tilde{u_{b}}|\psi D\psi
\,\tilde{u_{b}}^{\beta}\big)\varphi^2\big| & \le 2\Lambda\psi|D\psi|\,|D
\tilde{u_{b}}|\tilde{u_{b}}^\beta \varphi^2=2\Lambda\psi|D\psi|\,|D
\tilde{u_{b}}|\tilde{u_{b}}^{\frac{\beta-1}{2}}\tilde{u_{b}}^{\frac{\beta+1}{2}}\varphi^2\nonumber\\
& \le \,\frac{\nu |\beta|}{2}\, \psi^2\varphi^2 |D \tilde{u_{b}}|^2
\tilde{u_{b}}^{\beta-1}+\,\frac{2}{\nu |\beta|}\,\Lambda^2
|D\psi|^2\varphi^2 \tilde{u_{b}}^{\beta+1}\nonumber\\
& = \,\frac{2\nu |\beta|}{(1+\beta)^2}\,
\psi^2\varphi^2|Dw|^2+\,\frac{2}{\nu |\beta|}\,\Lambda^2
|D\psi|^2\varphi^2 w^2. \label{sup8}
\end{align}
Combining (\ref{sup6}), (\ref{sup7}), and (\ref{sup8}) we infer that
\begin{equation} \label{sup9}
\int_{B_{r\sigma}}\varphi^2\psi^2w^2\,dx+\,\frac{2\nu
|\beta|}{|1+\beta|}\,l\ast\int_{B_{r\sigma}}\varphi^2 \psi^2|Dw|^2\,dx
\le l\ast F,\quad\mbox{a.a.}\;s\in(0,t_0-t_1),
\end{equation}
where
\eqnsl{
F(s) =\, & \,\frac{2\Lambda^2|1+\beta|}{\nu |\beta|}\,  \int_{B_{r\sigma}}
|D\psi|^2\varphi^2 w^2\,dx
+|\beta|\varphi^2(s)k(s)\int_{B_{r\sigma}}\psi^2 w^2
\,dx \\
& +2\dot{\varphi}(s)\vf(s) \big(k\ast \int_{B_{r\sigma}}\psi^2
w^2\,dx\big)(s)
 + |1+\beta|\int_{B_{r\sigma}} \psi^{2}w^{2}\frac{|\tilde{f}|}{b}\varphi^2\, dx \geq 0,
}{defeF}
for a.a.\ $s\in(0,t_0-t_1)$.
Evidently, both terms on the left-hand side of (\ref{sup9}) are nonnegative. Dropping the second one, applying Young's inequality for convolutions  we obtain
\begin{equation} \label{sup10}
\Big(\int_{0}^{t_0-t_1} (\int_{B_{r\sigma}}
[\vf(s) \psi(x)w(s,x)]^2\,dx)^{p_0}\,ds\Big)^{1/p_{0}} \,\le
\|l\|_{L_{p_{0}}([0,t_0-t_1])} \int_0^{t_0-t_1} \!\!\!\!F(s)\,ds,
\end{equation}
 where $p_0$ comes from the assumption (\ref{ak1}).

Returning to (\ref{sup9}), we may also drop the first term, convolve
the resulting inequality with $k$ and evaluate at
$s=t_0-t_1$, thereby finding that
\begin{equation} \label{sup12}
\int_{0}^{t_0-t_1}\int_{B_{r\sigma}}\vf^2 \psi^2|Dw|^2\,dx\,ds \le
\,\frac{|1+\beta|}{2\nu |\beta|}\,\int_0^{t_0-t_1} \!\!\!\!F(s)\,ds.
\end{equation}

We now estimate the term $\int_0^{t_0-t_1}F(s)ds$. Firstly, we have

\eqq{
\int_{0}^{t_{0}-t_{1}}\int_{B_{r\sigma}}\abs{D \psi}^{2}w^{2}\,dx\,ds \leq \frac{4}{r^{2}\sigma^{2}(\rho-\rho')^{2}}\int_{0}^{t_{0}-t_{1}}\int_{B_{\rho r\sigma}}w^{2}\,dx\,ds.
}{pierw1}
Next, note that Lemma \ref{kernels} and Lemma \ref{ba} allow to estimate
\[
\vf^2(s)k(s) \leq k\left(\frac{t_{2}-t_{1}}{2}\right) \leq k_1\left(\jd (\rho-\rho') \sigma \eta \Phi(2r)  \right) \leq \frac{2 \max\{1,\eta^{-1}\}}{\sigma (\rho-\rho') } k_{1}(\vdr),
\]
and thus we obtain
\eqq{
\vf^2(s)k(s)\int_{B_{r\sigma}}\psi^{2}w^{2}dx \leq c(\eta,\delta)(\rho-\rho')^{-1}k_1(\vdr) \int_{B_{\rho r\sigma}}w^{2}dx.
}{fa}
We further have
\[
\dot{\vf}(s)\vf(s)(k*\int_{B_{r\sigma}}\psi^{2}w^{2}dx)(s)\leq \frac{8}{\sigma\eta(\rho-\rho')\vdr}(k*\int_{B_{\rho \sigma r}}w^{2}dx)(s),
\]
and consequently
\begin{align}\label{inha}
\int_{0}^{t_{0}-t_{1}}\dot{\vf}(s)\vf(s)\big( k* &\int_{B_{\rho \sigma r}} w^{2}dx\big)(s)\,ds  \leq  \frac{4}{\sigma\eta(\rho-\rho')\vdr} \big(1*k*\int_{B_{\rho \sigma r}}w^{2}dx\big)(t_0-t_1) \nonumber
\\
& \leq \frac{4}{\sigma\eta(\rho-\rho')\vdr} (1*k)\big(\rho\sigma\eta\vdr\big)\int_{0}^{t_{0}-t_{1}}\int_{B_{\rho \sigma r}}w^{2}\,dx\,ds \nonumber\\
& \leq \frac{4}{(\rho-\rho')}  \frac{(1*k)\big(\sigma\eta \vdr\big)}{\sigma\eta \vdr}\int_{0}^{t_{0}-t_{1}}\int_{B_{\rho \sigma r}}w^{2}\,dx\,ds \nonumber \\
& \leq  \frac{4 \overline{c} \max\{1,\eta^{-1}\}}{\sigma(\rho-\rho')} \ki(\vdr)
 \int_{0}^{t_{0}-t_{1}}\int_{B_{\rho \sigma r}}w^{2}\,dx\,ds,
\end{align}
where we used the fact that $k$ is nonincreasing and Lemma \ref{kernels}. By Lemma~\ref{fi}, $\ki(\Phi(2r))=r^{-2}/4$.

Finally, let us estimate the $f$ - term. Applying H\"older's inequality and recalling the definition of $b$ we have
\[
\int_{0}^{t_0-t_1} \int_{B_{r\sigma}} \vf^{2}\psi^{2}w^{2}\frac{|\tilde{f}|}{b}\, dx\, ds \leq
\frac{1}{b}\norm{f}_{L_{q_{1}}((0,T);L_{q_2}(\Omega))} \norm{\vf \psi w}_{L_{2q'_{1}}((0,t_0-t_1);L_{2q'_2}(B_{ r \sigma }))}^{2}
\]
\eqq{
= r^{\frac{N}{q_2}-2} (\vdr)^{\frac{1}{q_{1}}}\norm{\vf \psi w}_{L_{2q'_{1}}((0,t_0-t_1);L_{2q'_2}(B_{ r \sigma }))}^{2}.
}{inhabef}
In case, $q_1<\infty$ and $q_2 < \infty$, let us define
\eqq{
\tilde{q}_{2} = \frac{N}{2}+\frac{p'_0}{q_1} q_2, \hd \hd \tilde{q}_{1} = p'_0 +\frac{N }{2q_2} q_1.
}{qtilde}
Then $\tilde{q}_{2} < q_2$, $\tilde{q}_{1} < q_1$ and
\[
\frac{p'_0}{\tilde{q}_1} + \frac{N}{2\tilde{q}_2} = 1.
\]
In case $q_1=\infty$ we define $\tilde{q}_1 = \infty$ and $\tilde{q}_2 = \frac{N}{2}$, then $\tilde{q}_2 < q_2$. Analogously, if $q_2=\infty$ we define $\tilde{q}_2 = \infty$ and $\tilde{q}_1 = p'_0$, then $\tilde{q}_1 < q_1$.
In case $q_2 < \infty$, using H\"older's inequality and the interpolation inequality (\ref{sobolev}) with $a=2\tilde{q}'_1$, $b = 2\tilde{q}'_2$ we arrive at
\[
\norm{\vf \psi w}_{L_{2q'_{1}}((0,t_0-t_1);L_{2q'_2}(B_{ r \sigma }))} \leq \norm{\vf \psi w}_{L_{2\tilde{q}'_{1}}((0,t_0-t_1);L_{2\tilde{q}'_2}(B_{ r \sigma }))}^{\frac{\tilde{q}_2}{q_2}}\norm{\vf \psi w}_{L_{2}((0,t_0-t_1);L_{2}(B_{ r \sigma }))}^{1-\frac{\tilde{q}_2}{q_2}}
\]
\[
\leq C(N,\tilde{q}_2) \left[\norm{\vf \psi w}^{1-\theta}_{L_{2p_0}((0,t_0-t_1);L_{2}(B_{ r \sigma }))}\norm{ D( \vf\psi w)}^{\theta}_{L_{2}((0,t_0-t_1);L_{2}(B_{ r \sigma }))}\right]^{\frac{\tilde{q}_2}{q_2}}\norm{\vf \psi w}_{L_{2}((0,t_0-t_1);L_{2}(B_{ r \sigma }))}^{1-\frac{\tilde{q}_2}{q_2}},
\]
where $\theta = \frac{N}{2\tilde{q}_2}$. Recalling $\tilde{q}_2/q_2 = 1-d$ and using Young's inequality we get
\[
\norm{\vf \psi w}_{L_{2q'_{1}}((0,t_0-t_1);L_{2q'_2}(B_{ r \sigma }))} \leq C(N,\tilde{q}_2) \norm{\vf \psi w}_{L_{2}((0,t_0-t_1);L_{2}(B_{ r \sigma }))}^{d}
\]
\eqq{
\times \left[\norm{l}^{-\frac{1}{2}}_{L_{p_0}(0,\vdr)}\norm{\vf \psi w}_{L_{2p_0}((0,t_0-t_1);L_{2}(B_{ r \sigma }))} + \norm{ D( \vf\psi w)}_{L_{2}((0,t_0-t_1);L_{2}(B_{ r \sigma }))} \right]^{1-d}\norm{l}^{\frac{(1-\theta)(1-d)}{2}}_{L_{p_0}(0,\vdr)}.
}{fffa}
In case $q_2=\infty$, using the interpolation estimate and noticing that in this case $\tilde{q}_1/q_1 = 1-d$ we arrive at (\ref{fffa}) with $\theta = 0$.
In any case, inserting this estimate in (\ref{inhabef}) gives
\[
\int_{0}^{t_0-t_1} \int_{B_{r\sigma}} \vf^{2}\psi^{2}w^{2}\frac{|\tilde{f}|}{b}\, dx\, ds \leq C(N,\tilde{q}_2)  r^{\frac{N}{q_2}-2}(\vdr)^{\frac{1}{q_1}} \norm{\vf \psi w}_{L_{2}((0,t_0-t_1);L_{2}(B_{ r \sigma }))}^{2d}
\]
\eqq{
\times \left[\norm{l}^{-1}_{L_{p_0}(0,\vdr)}\norm{\vf \psi w}^2_{L_{2p_0}((0,t_0-t_1);L_{2}(B_{ r \sigma }))} + \norm{ D( \vf\psi w)}^2_{L_{2}((0,t_0-t_1);L_{2}(B_{ r \sigma }))} \right]^{1-d}\norm{l}^{(1-\theta)(1-d)}_{L_{p_0}(0,\vdr)}.
}{fff}
 Hence, combining (\ref{defeF}),    (\ref{pierw1}), (\ref{fa}), (\ref{inha}) and the last estimate we obtain
\begin{align}\label{estiF1}
&\int_{0}^{t_{0}-t_{1}}\hspace{-0.5cm} F(s)\,ds  \leq c  (1+\abs{1+\beta})\Big(\frac{1}{r^{2}(\rho-\rho')^{2}}\int_{0}^{t_{0}-t_{1}}\hspace{-0.2cm}\int_{B_{\rho \sigma r}}\hspace{-0.3cm}w^{2}\,dx\,ds + r^{\frac{N}{q_2}-2} (\vdr)^{\frac{1}{q_{1}}}\norm{w}_{L_{2}((0,t_0-t_1);L_{2}(B_{ \rho \sigma r  }))}^{2d} \nonumber \\&
\times \left[\norm{l}^{-1}_{L_{p_0}(0,\vdr)}\norm{\vf \psi w}^2_{L_{2p_0}((0,t_0-t_1);L_{2}(B_{ r \sigma }))} + \norm{ D( \vf\psi w)}^2_{L_{2}((0,t_0-t_1);L_{2}(B_{ r \sigma }))} \right]^{1-d}\norm{l}^{(1-\theta)(1-d)}_{L_{p_0}(0,\vdr)}  \Big),
\end{align}
where $c=c(\eta,\delta,\Lambda,\nu,\overline{c},N,\tilde{q}_2)$.
Using $\rho\sigma \le 1$ and positivity as well as monotonicity of $l$ we have
\eqq{
\norm{l}^{p_0}_{L_{p_0}((0,\rho\sigma\eta\Phi(2r)))} \leq \norm{l}^{p_0}_{L_{p_0}((0,\eta\Phi(2r)))} \leq \max\{1,\eta\} \norm{l}^{p_0}_{L_{p_0}((0,\Phi(2r)))},
}{estiPhieta}
where  we applied the substitution $\tilde{s}:=s/\eta$ in the case $\eta > 1$.
Hence, estimates (\ref{sup10}) - (\ref{pierw1}) imply
\[
\norm{l}^{-1}_{L_{p_0}(0,\vdr)}\norm{\vf \psi w}^2_{L_{2p_0}((0,t_0-t_1);L_{2}(B_{ r \sigma }))} + \norm{ D( \vf\psi w)}^2_{L_{2}((0,t_0-t_1);L_{2}(B_{ r \sigma }))}
\]
\eqq{
\leq c(\nu,\eta) \left(|1+\beta|\int_{0}^{t_{0}-t_{1}} F(s)\,ds + \frac{1}{r^{2}\sigma^{2}(\rho-\rho')^{2}}\int_{0}^{t_{0}-t_{1}}\int_{B_{\rho r\sigma}}w^{2}\,dx\,ds\right).
}{estiF2}
Combining (\ref{estiF1}) and (\ref{estiF2}) and using Young's inequality gives
\[
\norm{l}^{-1}_{L_{p_0}(0,\vdr)}\norm{\vf \psi w}^2_{L_{2p_0}((0,t_0-t_1);L_{2}(B_{ r \sigma }))} + \norm{ D( \vf\psi w)}^2_{L_{2}((0,t_0-t_1);L_{2}(B_{ r \sigma }))}
\]
\eqq{
\leq c (1+\abs{1+\beta})^\zeta \int_{0}^{t_{0}-t_{1}}\int_{B_{\rho \sigma r}}w^{2}\,dx\,ds \left[\frac{1}{r^{2}(\rho-\rho')^{2}} + \norm{l}^{(1-\theta)\frac{1-d}{d}}_{L_{p_0(0,\vdr)}} (r^{-2+\frac{N}{q_2}}\vdr^{\frac{1}{q_1}})^{\frac{1-d}{d}}
\right],
}{q1q2}
where $\zeta = \frac{2(1-d)}{d}$,  $c=c(\eta,\delta,\Lambda,\nu,\overline{c},N,\tilde{q}_2)$.
If $q_1$ and $q_2$ are finite, we apply Lemma \ref{scaling}, and recall that $\theta = \frac{N}{2\tilde{q}_2}$  thereby obtaining
\[
\norm{l}^{(1-\theta)\frac{1-d}{d}}_{L_{p_0}(0,\vdr)} (r^{-2+\frac{N}{q_2}}\vdr^{\frac{1}{q_1}})^{\frac{1-d}{d}}\leq c(p_0,\overline{c}) \Big(r^{2}(\vdr)^{-\frac{1}{p'_0}}\Big)^{(1-\frac{N}{2\tilde{q}_2})\frac{1-d}{d}} (r^{-2+\frac{N}{q_2}}\vdr^{\frac{1}{q_1}})^{\frac{1-d}{d}}.
\]
We note that since $\tilde{q}_2/q_2 = 1-d$, we have
\[
\left[2(1-\frac{N}{2\tilde{q}_2}) + (-2+\frac{N}{q_2})\right]\left[\frac{1-d}{d}\right] = \frac{N}{q_2}\left(1-\frac{q_2}{\tilde{q}_2}\right)\frac{1-d}{d} = \frac{N}{q_2}\left(1-\frac{1}{1-d}\right)\frac{1-d}{d} = -\frac{N}{q_2}.
\]
Similarly, using $\tilde{q}_1/q_1 = 1-d$ we may calculate
\[
\left[-\frac{1}{p'_0}(1-\frac{N}{2\tilde{q}_2})+\frac{1}{q_1})\right]\frac{1-d}{d} = \frac{1}{q_1}\left[1-\frac{q_1}{\tilde{q}_1}\right]\frac{1-d}{d} = -\frac{1}{q_1}.
\]
Hence
\eqq{
 \norm{l}^{(1-\theta)\frac{1-d}{d}}_{L_{p_0}(0,\vdr)} (r^{-2+\frac{N}{q_2}}\vdr^{\frac{1}{q_1}})^{\frac{1-d}{d}}\leq c(p_0,\overline{c}) r^{-\frac{N}{q_2}}(\vdr)^{-\frac{1}{q_1}} \leq c(p_0,\overline{c}) r^{-2},
}{q1q22}
where in the last estimate we apply Proposition \ref{estiprop}. We note that if either $q_1=\infty$ or $q_2 = \infty$ (in the second case $\theta = 0$) from (\ref{q1q2}) we similarly obtain (\ref{q1q22}).
Thus, using the properties of $\varphi$
\eqnsl{ &
\norm{l}^{-1}_{L_{p_0}(0,\vdr)}\norm{ \psi w}^2_{L_{2p_0}((t_2-t_1,t_0-t_1);L_{2}(B_{ r \sigma }))} + \norm{ D( \psi w)}^2_{L_{2}((t_2-t_1,t_0-t_1);L_{2}(B_{ r \sigma }))}
  \\
& \leq c(\eta,\delta,\Lambda, \nu,\overline{c},N,p_0,q_1,q_2)  \frac{(1+\abs{1+\beta})^{\zeta}}{r^{2}(\rho-\rho')^{2}}\int_{0}^{t_{0}-t_{1}}\int_{B_{\rho \sigma r}}w^{2}\,dx\,ds .}{estiwwl2p}

Let us denote
\eqq{
\kappa= \frac{2p_0+N(p_0-1)}{2+N(p_0-1)}, \hd \hd \vartheta = \frac{N(p_0-1)}{N(p_0-1)+2p_0}.
}{defkappatheta}
Applying the embedding (\ref{sobolev}) with $a = 2\kappa $ and $b = 2\kappa $ we obtain
\[
\norm{\psi w}_{L_{2\kappa}((t_{2}-t_{1},t_{0}-t_{1});L_{2\kappa}(B_{r\sigma}))}
\leq C   \norm{D(\psi w)}^{\vartheta}_{L_{2}((t_{2}-t_{1},t_{0}-t_{1});L_{2}(B_{r\sigma}))}  \norm{\psi w}_{L_{2p_0}((t_{2}-t_{1},t_{0}-t_{1});L_{2}(B_{r\sigma}))}^{1-\vartheta}
\]
\eqq{
\leq C\left[\norm{l}^{\frac{1}{2}}_{L_{p_0}(0,\vdr)}\norm{ \psi w}_{L_{2p_0}((t_2-t_1,t_0-t_1);L_{2}(B_{ r \sigma }))} + \norm{ D( \psi w)}_{L_{2}((t_2-t_1,t_0-t_1);L_{2}(B_{ r \sigma }))}\right]\norm{l}^{\frac{1-\vartheta}{2}}_{L_{p_0}(0,\vdr)},
}{interapara}
where $C>0$ depends only on $N,p_0$. Combining (\ref{estiwwl2p})  and (\ref{interapara}) gives
\eqq{
\norm{\psi w}_{L_{2\kappa}((t_{2}-t_{1},t_{0}-t_{1});
L_{2\kappa}(B_{r\sigma}))}
 \leq c(\eta,\delta,\Lambda, \nu,\overline{c},N,p_0,q_1,q_2)\frac{1}{K(r)} \frac{(1+\abs{1+\beta})^{\zeta/2}}{\rho-\rho'}\left(\int_{0}^{t_{0}-t_{1}}\int_{B_{\rho \sigma r}}w^{2}dxd\tau\right)^{\frac{1}{2}},
}{mainMoser1}
where
\eqq{
K(r):= r\norm{l}_{L_{p}((0,\Phi(2r)))}^{\frac{\vartheta-1}{2}}.
}{defKr}
We denote  $\gamma=|1+\beta|$ and from the properties of $\psi$ we have
\[
\norm{\psi w}_{L_{2\kappa}((t_{2}-t_{1},t_{0}-t_{1});
L_{2\kappa}(B_{r\sigma}))} \geq \| u_b^{-1}\|_{L_{\kappa \gamma }(V_{\rho'})}^{\frac{\gamma}{2}}
\]
and
\[
\left( \int_{0}^{t_{0}-t_{1}} \int_{B_{\rho\sigma r}} w^{2}dxds \right)^{\frac{1}{2}} = \| u_b^{-1} \|_{L_{\gamma}(V_{\rho})}^{\frac{\gamma}{2}}.
\]
Therefore, (\ref{mainMoser1}) leads to the estimate
\[
\|u_b^{-1}\|_{L_{\gamma\kappa}(V_{\rho'})} \leq \left(\frac{C^{2}(1+\gamma)^{\zeta}}{(\rho-\rho')^{2}(K(r))^2}\right)^{\frac{1}{\gamma}} \|u_b^{-1}\|_{L_{\gamma}(V_{\rho})},\quad 0<\rho'<\rho\le 1.
\]
Note that $\kappa>1$ and $\zeta>2$. We apply the first  Moser iteration lemma (Lemma~\ref{moserit1}) to get
\[
\esssup_{V_{\vsi}}u_b^{-1}\leq \left(\frac{M_{0}}{(1-\vsi)^{\frac{\zeta\kappa}{\kappa-1}}(K(r))^{2\frac{\kappa}{\kappa-1}}}\right)^{\frac{1}{\gamma}}\|u_b^{-1}\|_{L_{\gamma}(V_{1})}, \hd \gamma\in (0,1], \hd \vsi \in (0,1),
\]
where $M_{0}=M_{0}( \eta,\delta,\Lambda, \nu,\overline{c},N,p_0,q_1,q_2)$.
Using (\ref{defkappatheta}) and Lemma \ref{scaling}  we obtain
\eqq{
(K(r))^{2\frac{\kappa}{\kappa-1}} = \norm{l}_{L_{p_0}((0,\Phi(2r)))}^{\frac{-p_0}{p_0-1}}r^{N}r^{\frac{2p_0}{p_0-1}} \geq c r^{N}\vdr.
}{Krdopotegi}
Hence,
\[
\esssup_{V_{\vsi}}u_b^{-1}\leq \left(\frac{M_{0}}{(1-\vsi)^{\frac{\zeta\kappa}{\kappa-1}}r^{N}\vdr}\right)^{\frac{1}{\gamma}}\|u_b^{-1}\|_{L_{\gamma}(V_{1})}, \hd \gamma\in (0,1], \hd \vsi \in (0,1),
\]
where $M_{0}=M_{0}(\eta,\delta,\Lambda, \nu,\overline{c},N,p_0,q_1,q_2)$. Thus, if we take
$\vsi=\sigma'/\sigma$ and notice that $V_{\vsi}=U_{\sigma'}$, $\nuNj(U_{1})=\eta\Phi(2r)r^{N}$ and
\[ \frac{1}{1-\vsi}\,=\,\frac{\sigma}{\sigma-\sigma'}\,\le
\frac{1}{\sigma-\sigma'},\]
then we obtain for $r \leq r^{*}$
\[
\esup_{U_{\sigma'}}{u_b^{-1}} \le
\Big(\frac{M_0 \nuNj(U_{1})^{-1}}{(\sigma-\sigma')^{\tau_0}}\Big)^{1/\gamma}
\|u_b^{-1}\|_{L_{\gamma}(U_\sigma)},\quad \gamma\in (0,1],
\]
where $M_{0}=M_{0}( \eta,\delta,\Lambda, \nu,\overline{c},N,p_0,q_1,q_2)$ and $\tau_0=\tau(N,p_0,d)>0$. Hence the proof is complete.

\end{proof}

${}$

\begin{satz} \label{superest2}
Let $\Omega\subset \iR^N$ be a bounded
domain. Suppose the assumptions (H1)--(H3) and (\ref{ak1}) are satisfied and $\kappa$ is given by (\ref{defkappatheta}). Let
$\eta>0$ and $\delta\in (0,1)$ be fixed,
and let $r^{*}\in (0,r_0/2)$ be the number provided by Lemma \ref{scaling}.
Then for any $\tilde{p}\in(0,\kappa)$,  for any $t_0\in [0,T)$
and $r\in (0, r^{*}]$ with $t_0+\eta \vdr \le T$, any ball
$B_r(x_0)\subset\Omega$ and any
nonnegative weak supersolution $u$ of (\ref{MProb}) in $(0,t_0+\eta
\vdr )\times B_r(x_0)$ with $u_0\ge 0$ in $B_r(x_0)$,  there holds
\eqq{
\|u_b\|_{L_{\tilde{p}}(U_{\sigma'}')}\le \left( \frac{C \nuNj(U'_1)^{-1} }
{(\sigma-\sigma')^{\gamma_0}}\right)^
{\frac{1}{\gamma}-\frac{1}{\tilde{p}}}
\!\!\!\|u_b\|_{L_{\gamma}(U'_\sigma)},\;\; \delta\le \sigma'<\sigma\le 1,\;
0<\gamma\le \frac{\tilde{p}}{\kappa}.
}{fin3}
Here $U'_\sigma=(t_0,t_0+\sigma\eta \vdr)\times B_{\sigma r}(x_0)$, $\sigma\in (0,1]$,
$C=C(\nu,\Lambda,\delta,\eta,N,\tilde{p},p_0,q_1,q_2,\overline{c})$ and
$\gamma_0=\gamma_0(N, p_0)$.
\end{satz}
\begin{proof} At first, we adapt some steps from the previous proof. We then follow again the approach from \cite{base} and \cite{harnackdistr} and use the key Lemma \ref{scaling}. Let us note that by H\"older's inequality, it is enough to prove (\ref{fin3}) for $\tilde{p}>1$. To shorten the notation, we write again $B_{r}:=B_{r}(x_{0})$.

Let $r \in (0,r^{*})$ and let us fix $\sigma'$, $\sigma$ such that $\delta\le \sigma'<\sigma\le 1$. For $\rho\in (0,1]$ we set
$V'_\rho=U'_{\rho\sigma}$. For any fixed $0<\rho'<\rho\le 1$, let
$t_1=t_0+\rho'\sigma\eta\vdr$ and $t_2=t_0+\rho\sigma\eta\vdr$, so $0\le
t_0<t_1<t_2$. We shift the time by means of ${s}=t-t_0$ and set
$\tilde{g}(s)=f(s+t_0)$, $s\in (0,t_2-t_0)$, for functions $g$
defined on $(t_0,t_2)$.

Proceeding as in the proof of Theorem \ref{superest1}, now with $\beta\in (-1,0)$, we note that (\ref{sup1}) simplifies to
\[
-\tilde{u_{b}}^{\beta}\partial_{s}(k_{n}\ast \tilde{u_{b}}) \ge
-\,\frac{1}{1+\beta}\,\partial_{s}
(k_{n}\ast\tilde{u_{b}}^{1+\beta}),\quad
\mbox{a.a.}\;(s,x)\in (0,t_2-t_0)\times B_r,
\]
hence we obtain with $\psi\in C_0^1(B_{r\sigma})$ as above
\begin{align}
-\,\frac{1}{1+\beta}\,& \int_{B_{r\sigma}}\psi^2\partial_{s}
(k_{n}\ast\tilde{u_{b}}^{1+\beta})\,dx+|\beta|\int_{B_{r\sigma}}\big((h_n\ast
[ADu_b])\,\tilde{}\;|\psi^2 \tilde{u_{b}}^{\beta-1}D
\tilde{u_{b}}\big)\,dx \nonumber\\
\le  & \,2\int_{B_{r\sigma}}\big((h_n\ast [ADu_b])\,\tilde{}\;|\psi D\psi
\,\tilde{u_{b}}^{\beta}\big)\,dx - \int_{B_{r\sigma}}\psi^2 (h_n *f)\,\tilde{}\;\tilde{u_{b}}^{\beta}\,dx ,\quad\quad
\mbox{a.a.}\;s\in(0,t_2-t_0). \label{L1}
\end{align}

We choose a cut-off function $\varphi\in C^1([0,t_2-t_0])$ such that $0\le
\varphi\le 1$, $\varphi=1$ in $[0,t_1-t_0]$, $\varphi=0$ in
$[t_1-t_0+(t_2-t_1)/2,t_2-t_0]$, and $0\le -\dot{\varphi}\le
4/(t_2-t_1)$. We multiply (\ref{L1}) by $1+\beta>0$ and by
$\varphi^{2}$. Then, applying Lemma \ref{comm2} to the first term we arrive at
\begin{align}
-\int_{B_{r\sigma}}  &
\partial_{s}(k_{n}\ast
[\varphi^2\psi^2\tilde{u_{b}}^{1+\beta}]\big)\,dx+|\beta|(1+\beta)\,
\int_{B_{r\sigma}}\big(\tilde{A}D\tilde{u_{b}}|\psi^2 \tilde{u_{b}}^{\beta-1}D
\tilde{u_{b}}\big)\varphi^2\,dx \nonumber\\
\le & \,\int_0^s
\dot{k}_{n}(s-\tau)\big(\varphi^2(s)-\varphi^2(\tau)\big)
\big(\int_{B_{r\sigma}}\psi^2\tilde{u_{b}}^{1+\beta}\,dx\big)(\tau)\,d\tau +\mathcal{R}_n(s) \nonumber\\
& \;\,+2(1+\beta)\,\int_{B_{r\sigma}}\big(\tilde{A}D\tilde{u_{b}}|\psi D\psi
\,\tilde{u_{b}}^{\beta}\big)\varphi^2\,dx- (1+\beta)\int_{B_{r\sigma}}\psi^2(h_n *f)\,\tilde{}\;\tilde{u_{b}}^{\beta}\varphi^2\,dx ,\quad
\mbox{a.a.}\;s\in(0,t_2-t_0), \label{L2}
\end{align}
where
\begin{align*}
\mathcal{R}_n(s)= &\,\,-|\beta|(1+\beta)\, \int_{B_{r\sigma}}\big((h_n\ast
[ADu_b])\,\tilde{}\;-\tilde{A}D\tilde{u_{b}}|\psi^2 \tilde{u_{b}}^{\beta-1}D
\tilde{u_{b}}\big)\varphi^2\,dx\\
&\,+2(1+\beta)\,\int_{B_{r\sigma}}\big((h_n\ast
[ADu_b])\,\tilde{}\;-\tilde{A}D\tilde{u_{b}}|\psi D\psi
\,\tilde{u_{b}}^{\beta}\big)\varphi^2\,dx,\quad
\mbox{a.a.}\;s\in(0,t_2-t_0).
\end{align*}

We set again $w=\tilde{u_{b}}^{\frac{\beta+1}{2}}$ and estimate
as in the proof of Theorem \ref{superest1}, using (H1), (H3) and
(\ref{sup8}), to the result
\begin{align}
-\int_{B_{r\sigma}}  &
\partial_{s}(k_{n}\ast
[\varphi^2\psi^2w^2]\big)\,dx+\,\frac{2\nu
|\beta|}{1+\beta}\,\int_{B_{r\sigma}}\varphi^2
\psi^2|Dw|^2\,dx \nonumber\\
\le & \,\int_0^s
\dot{k}_{n}(s-\tau)\big(\varphi^2(s)-\varphi^2(\tau)\big)
\big(\int_{B_{r\sigma}}\psi^2w^2\,dx\big)(\tau)\,d\tau\nonumber +\mathcal{R}_n(s)\\
& \;\,+\,\frac{2\Lambda^2(1+\beta)}{\nu |\beta|}\, \int_{B_{r\sigma}}
|D\psi|^2\varphi^2 w^2\,dx + (1+\beta)\abs{\int_{B_{r\sigma}} \psi^2(h_n *f)\,\tilde{}\;\frac{w^2}{\tilde{u_{b}}}\varphi^2\,dx} ,\quad
\mbox{a.a.}\;s\in(0,t_2-t_0). \label{L3}
\end{align}
Remember that $k_{n}=k\ast h_n$. If we denote the right-hand side of (\ref{L3}) by $F_n(s)$ and introduce
\[
W(s)=\int_{B_{r\sigma}}\varphi(s)
^2\psi(x)^2w(s,x)^2\,dx,
\]
it follows from (\ref{L3}) that
\[
G_n(s):=\partial_s [k* (h_n\ast W)](s)+F_n(s)\ge 0,\quad\quad
\mbox{a.a.}\;s\in(0,t_2-t_0).
\]
We know that $h_n$ is nonnegative for every $n\in \iN$. Hence, applying (\ref{sup5}), we have
\[
0\le h_n\ast W =l\ast \partial_s[ k* (h_n\ast W)]\le
l\ast G_n+l\ast [-F_n(s)]_+
\]
a.e. in $(0,t_2-t_0)$, where $[y]_+$ stands for the positive part
of $y\in \iR$. Then, for any
$t_*\in[t_2-t_0-(t_2-t_1)/4,t_2-t_0]$, Young's
inequality leads to
\begin{equation} \label{L4}
\|h_n\ast W\|_{L_{p_0}((0,t_*))}\le
\|l\|_{L_{p_{0}}((0,t_*))}\big(\|G_n\|_{L_1((0,t_*))}+
\|[-F_n]_+\|_{L_1((0,t_*))}\big).
\end{equation}
Positivity of $G_n$ implies
\eqq{
\|G_n\|_{L_1((0,t_*))}=(k_{n}\ast
W)(t_*)+\int_0^{t_*}\!\!\!F_n(s)\,ds.
}{Gnsuma}
Furthermore, $\mathcal{R}_n\rightarrow 0$ in $L_1((0,t_2-t_0))$ as
$n\to \infty$ and since $k_{n}$ and $\varphi$ are nonincreasing, there holds
$[-F_n]_+ \leq [-\mathcal{R}_n]_{+}$ and   $  \|[-F_n]_+\|_{L_1((0,t_*))}\to 0$ as
$n\to\infty$. Further,
\eqnsl{
\int_0^{t_*}\!\!&\!\int_0^s
\dot{k}_{n}(s-\tau)\big(\varphi^2(s)-\varphi^2(\tau)\big)
\big(\int_{B_{r\sigma}}\psi^2w^2\,dx\big)(\tau)\,d\tau\,ds\\
=&\,\int_0^{t_*}
k_{n}(t_*-\tau)\big(\varphi^2(t_*)-\varphi^2(\tau)\big)
\big(\int_{B_{r\sigma}}\psi^2w^2\,dx\big)(\tau)\,d\tau\\
&\,-2\int_0^{t_*}\!\!\!\dot{\varphi}(s)\vf(s)\int_0^s
k_{n}(s-\tau)
\big(\int_{B_{r\sigma}}\psi^2w^2\,dx\big)(\tau)\,d\tau\,ds\\
\le&\,-2\int_0^{t_*}\!\!\!\dot{\varphi}(s)\int_0^s
k_{n}(s-\tau)
\big(\int_{B_{r\sigma}}\psi^2w^2\,dx\big)(\tau)\,d\tau\,ds,
}{kndot}
because $\varphi$ is nonincreasing. Since
$k_{n}\ast W\to k\ast W$ in
$L_1((0,t_2-t_0))$, we may choose
$t_*\in[t_2-t_0-(t_2-t_1)/4,t_2-t_0]$ and a
subsequence such that $(k_{n_m}\ast W)(t_*)\to (k\ast
W)(t_*)$ as $m\to \infty$. Then (\ref{Gnsuma}) leads to
\[
\limsup_{m\rightarrow \infty} \|G_{n_{m}}\|_{L_{1}((0,t_*))}\leq (k*W)(t_*)+\limsup_{m\rightarrow \infty} \int_0^{t_*}\!\!\!F_{n_{m}}(s)\,ds.
\]
Applying the previous estimate we get
\[
\limsup_{m\rightarrow \infty} \int_0^{t_*}\!\!\!F_{n_{m}}(s)\,ds \leq 2
\limsup_{m\rightarrow \infty}  \int_0^{t_*}\!\!\!- \dot{\varphi}(s)\int_0^s
k_{n_{m}}(s-\tau)
\big(\int_{B_{r\sigma}}\psi^2w^2\,dx\big)(\tau)\,d\tau\,ds
\]
\[
+\frac{2\Lambda^2(1+\beta)}{\nu |\beta|}\, \int_{0}^{t_*}\int_{B_{r\sigma}}
|D\psi|^2\varphi^2 w^2\,dx ds +\limsup_{m\rightarrow \infty} \int_{0}^{t_*} |\mathcal{R}_{n_{m}}(s)|ds
\]
\[
+ (1+\beta)\limsup_{m\rightarrow \infty} \int_{0}^{t_*} \int_{B_{r\sigma}}\abs{\psi^2(h_{n_m} *f)\,\tilde{}\;\frac{w^2}{\tilde{u_{b}}}\varphi^2}\,dx ds
\]
\[
\leq 2\int_0^{t_*}\!\!\!- \dot{\varphi}(s)\int_0^s
k(s-\tau)
\big(\int_{B_{r\sigma}}\psi^2w^2\,dx\big)(\tau)\,d\tau\,ds
+\frac{2\Lambda^2(1+\beta)}{\nu |\beta|}\, \int_{0}^{t_*}\int_{B_{r\sigma}}
|D\psi|^2\varphi^2 w^2\,dx ds
\]
\[
+(1+\beta)\int_{0}^{t_*} \int_{B_{r\sigma}}\psi^2\frac{\abs{f}}{b}w^{2}\varphi^2\,dx ds.
\]
Since  $  \|[-F_n]_+\|_{L_1((0,t_*))}\to 0$, the estimates above together with (\ref{L4}) imply
\[
\|W\|_{L_{p_0}((0,t_*))} \leq \|l\|_{L_{p_0}((0,t_*))} \Big((k\ast
W)(t_*)+\|F\|_{L_1((0,t_*))}\Big),
\]
where
\begin{align}
F(s) & =\,\frac{2\Lambda^2(1+\beta)}{\nu |\beta|}\, \int_{B_{r\sigma}}
|D \psi|^2\varphi^2 w^2\,dx-2\dot{\varphi}(s)\big(k\ast
\int_{B_{r\sigma}}\psi^2w^2\,dx\big)(s)
\nonumber\\
& \quad\quad +(1+\beta)\int_{0}^{t_*} \int_{B_{r\sigma}}\psi^2\frac{\abs{f}}{b}w^{2}\varphi^2\,dx ds.
\label{Fsdef}
\end{align}
Hence,
\begin{equation} \label{L5}
\Big(\int_{0}^{t_2-t_0} (\int_{B_{r\sigma}}
[\vf(s)\psi(x)w(s,x)]^2\,dx)^{p_{0}}\,ds\Big)^{1/p_0}\le
\|l\|_{L_{p_0}((0,t_{2}-t_{0}))}\Big((k\ast
W)(t_*)+\|F\|_{L_1((0,t_2-t_0))}\Big).
\end{equation}

On the other hand, we can integrate (\ref{L3}) over $(0,t_*)$, then apply (\ref{kndot}) and
pass to the limit with $m$ for the same subsequence as
above, thereby getting
\begin{equation} \label{L6}
\int_{0}^{t_2-t_0}\!\!\!\int_{B_{r\sigma}}
\vf^{2}\psi^2|Dw|^2\,dx\,ds\le\,\frac{1+\beta}{2\nu
|\beta|}\,\Big((k\ast
W)(t_*)+\|F\|_{L_1((0,t_2-t_0))}\Big).
\end{equation}
Observe that due to $t_*\in [(\frac{3}{4}\rho+\frac{1}{4}\rho')\sigma \eta \Phi(2r), \rho\sigma \eta \Phi(2r)]$ and $\varphi=0$ on $[\frac{1}{2}(\rho+\rho')\sigma \eta \Phi(2r), \rho\sigma \eta \Phi(2r)]$, we have
\eqnsl{
(k*W)(t_{*})= & \int_{0}^{\frac{1}{2}(\rho+\rho')\sigma \eta \Phi(2r)} k(t_*-\tau) \varphi^2(\tau ) \left( \int_{B_{r\sigma}}
\psi^2 w^2\,dx\right) (\tau ) d\tau \\
& \leq k\left(t_* -\frac{1}{2}(\rho+\rho')\sigma \eta \Phi(2r)\right) \int_{0}^{t_{2}-t_{0}} \int_{B_{\rho r\sigma} } w^{2} dx d\tau \\
& \leq k_{1}\left( \frac{1}{4}(\rho-\rho')\sigma \eta \Phi(2r)\right) \int_{0}^{t_{2}-t_{0}} \int_{B_{\rho r\sigma} } w^{2} dx d\tau \\
& \leq \frac{4 \max\{1,\eta^{-1}\}}{\delta}\frac{k_1(\vdr)}{(\rho-\rho')}\int_{0}^{t_{2}-t_{0}}\int_{B_{\rho \sigma r}}w^{2}dxd\tau,
}{estikW}
where we used  Lemma \ref{kernels} and Lemma \ref{ba}.

Next, we estimate the $L_1$-norm of (\ref{Fsdef}). Using properties of $\vf$,  Young's inequality for convolutions, monotonicity of $k$ and Lemma \ref{kernels} we arrive at
\[
\int_{0}^{t_{2}-t_{0}} - \dot{\varphi}(s) \left( k* \int_{B_{r\sigma} } \psi^{2} w^{2}dx   \right)(s)ds \leq \frac{4}{t_{2}-t_{1}} \int_{0}^{t_{2}-t_{0}}   k* \int_{B_{r\sigma} } \psi^{2} w^{2}dx   ds
\]
\[
\leq  \frac{4}{(\rho - \rho ') \sigma \eta \Phi(2r)} (1*k)(\rho\sigma\eta\vdr)  \int_{0}^{t_{2}-t_{0}}  \int_{B_{ r\sigma} } \psi^2(x) w^{2}(x,s)dx ds
\]
\[
\leq  \frac{4\overline{c}}{(\rho - \rho ')}\frac{\max\{1,\eta^{-1}\}}{\delta} k_1(\vdr)\int_{0}^{t_{2}-t_{0}}  \int_{B_{ r\sigma} } \psi^2(x) w^{2}(x,s)dx ds.
\]
Thus, by employing (\ref{zn1}), we get
\eqq{\int_{0}^{t_{2}-t_{0}} - \dot{\varphi}(s) \left( k* \int_{B_{r\sigma} } \psi^{2} w^{2}dx   \right)(s)ds \leq
\frac{\overline{c}\max\{1,\eta^{-1}\}}{\delta(\rho - \rho ')^{2} r^{2}} \int_{0}^{t_{2}-t_{0}}    \int_{B_{\rho r\sigma} }  w^{2}dx   ds.}{estivphidot}

 Finally,
\eqq{
\int_{0}^{t_{2}-t_{0}}\int_{B_{r\sigma}}\abs{D \psi}^2 \varphi^2 w^{2}dxds \leq \frac{4}{r^{2}\sigma^{2}(\rho-\rho')^{2}}\int_{0}^{t_{2}-t_{0}}\int_{B_{\rho r\sigma}}w^{2}dxds.
}{estinabpsi}

Hence, from (\ref{Fsdef}), (\ref{estivphidot}), (\ref{estinabpsi}) and the estimate analogous to (\ref{fff}), we obtain
\begin{align}\label{estiFlj}
&\int_{0}^{t_{2}-t_{0}} F(s)\,ds  \leq c  \frac{(2+\beta)}{|\beta|}\Big(\frac{1}{r^{2}(\rho-\rho')^{2}}\int_{0}^{t_{2}-t_{0}}\hspace{-0.2 cm}\int_{B_{\rho \sigma r}}\hspace{-0.3cm} w^{2}\,dx\,ds + r^{\frac{N}{q_2}-2} (\vdr)^{\frac{1}{q_{1}}}\norm{w}_{L_{2}((0,t_2-t_0);L_{2}(B_{ \rho \sigma r  }))}^{2d} \nonumber \\&
\cdot \left[\norm{l}^{-1}_{L_{p_0}(0,\vdr)}\norm{\vf \psi w}^2_{L_{2p_0}((0,t_2-t_0);L_{2}(B_{ r \sigma }))} + \norm{ D( \vf\psi w)}^2_{L_{2}((0,t_2-t_0);L_{2}(B_{ r \sigma }))} \right]^{1-d}\norm{l}^{(1-\theta)(1-d))}_{L_{p_0}(0,\vdr)}  \Big),
\end{align}
where $\theta = 0$ in case $q_2=\infty$ and otherwise $\theta = \frac{N}{2\tilde{q}_2}$, where  $\tilde{q}_2$ is given by (\ref{qtilde}) and $c=c(\eta,\delta,\Lambda,\nu,\overline{c},N,\tilde{q}_2)$.

Using (\ref{L5}), (\ref{L6})(\ref{estikW}) and (\ref{zn1}) we find that
\[
\norm{l}^{-1}_{L_{p_0}(0,\vdr)}\norm{\vf \psi w}^2_{L_{2p_0}((0,t_2-t_0);L_{2}(B_{ r \sigma }))} + \norm{ D( \vf\psi w)}^2_{L_{2}((0,t_2-t_0);L_{2}(B_{ r \sigma }))}
\]
\eqq{
 \leq c(\nu,\eta,\delta)\left( \frac{1+\beta}{|\beta|r^{2}(\rho-\rho')}\int_{0}^{t_{2}-t_{0}}\int_{B_{\rho \sigma r}}w^{2}dx ds + \frac{1+\beta}{|\beta|}\int_{0}^{t_{2}-t_{0}} F(s)\,ds \right). }{estildpphiw}
Combining (\ref{estiFlj}) and (\ref{estildpphiw}) and applying Young's inequality we obtain, as in the proof of Theorem \ref{superest1}
\[
\norm{l}^{-1}_{L_{p_0}(0,\vdr)}\norm{\vf \psi w}^2_{L_{2p_0}((0,t_2-t_0);L_{2}(B_{ r \sigma }))} + \norm{ D( \vf\psi w)}^2_{L_{2}((0,t_2-t_0);L_{2}(B_{ r \sigma }))}
\]
\[
\leq c \frac{(2+\beta)^\zeta}{|\beta|^\zeta} \int_{0}^{t_{2}-t_{0}}\int_{B_{\rho \sigma r}}w^{2}\,dx\,ds \left[\frac{1}{r^{2}(\rho-\rho')^{2}} + \norm{l}^{(1-\theta)\frac{1-d}{d}}_{L_{p_0(0,\vdr)}} (r^{-2+\frac{N}{q_2}}\vdr^{\frac{1}{q_1}})^{\frac{1-d}{d}}
\right],
\]
where $\zeta = \frac{2(1-d)}{d}$ and $c=c(\eta,\delta,\Lambda,\nu,\overline{c},N,\tilde{q}_2)$.
Proceeding, exactly as in the proof of Theorem \ref{superest1} and using the properties of $\varphi$ we arrive at
\eqq{
\|w\|_{L_{2\kappa}((0,t_{1}-t_{0})\times B_{\rho'r\sigma})} \leq  c(\eta,\delta,\Lambda,\nu,\overline{c},N,p_0,q_1,q_2)\frac{(2+\beta)^{\zeta/2}}{|\beta|^{\zeta/2}(\rho-\rho')K(r)}\|w\|_{L_{2}((0,t_{2}-t_{0})\times B_{\rho r\sigma})}
}{estiVrho}
with $\kappa$, $K$ given by (\ref{defkappatheta}) and (\ref{defKr}), respectively.

If we denote $\gamma:=1+\beta$  we may write
\[
\|w\|_{L_{2\kappa}((0,t_{1}-t_{0})\times B_{\rho'r\sigma})} =  \| u_b \|^{\frac{\gamma}{2}}_{L_{\gamma \kappa }(V_{\rho'}')} \hd \m{ and } \hd\|w\|_{L_{2}((0,t_{2}-t_{0})\times B_{\rho r\sigma})} =  \| u_b \|^{\frac{\gamma}{2}}_{L_{\gamma  }(V_{\rho}')}
\]

 and therefore (\ref{estiVrho}) yields
\[
\| u_b \|_{L_{\gamma \kappa}(V'_{\rho'})} \leq \left(\frac{C}{|\beta|^{\zeta}(\rho-\rho')^{2}(K(r))^{2}}\right)^{\frac{1}{\gamma}} \| u_b \|_{L_{\gamma  }(V'_{\rho})},
\]
where $C=C(\eta,\delta,\Lambda,\nu,\overline{c},N,p_0,q_1,q_2)$.

For any $\tilde{p} \in (0, \kappa)$ we have $\frac{\tilde{p}}{\ka}<1$, so  for any
$\gamma = 1+\beta \in (0,\frac{\tilde{p}}{\ka}]$ we get the estimate
\eqq{
\| u_b \|_{L_{\gamma \kappa }(V'_{\rho'})} \leq \left(\frac{C}{(\rho-\rho')^{2}(K(r))^{2}}\right)^{\frac{1}{\gamma}} \| u_b \|_{L_{\gamma  }(V'_{\rho})}.
}{mainestiVprim}
where $C=C(\eta,\Lambda, \delta,\nu, N,\tilde{p}, p_{0},q_1,q_2,d,\overline{c})$. Note that $C$ does not depend on $\gamma$, because $1+\beta \le \frac{\tilde{p}}{\kappa}$, implies that $\beta$ is cut-off from zero.

Next, we multiply (\ref{mainestiVprim}) by $(\eta \om_{N}r^{N}\vdr)^{-\frac{1}{\gamma \kappa}}$, where $\om_{N}$ is the measure of the unit ball in $\R^{N}$, and we have
\eqns{
\left(\int_{V_{\rho'}'} |u_b|^{\kappa \gamma}\frac{1}{\eta \om_{N}r^{N}\vdr}dxdt\right)^{\frac{1}{\gamma \kappa}} &  \\
\leq \left(\frac{C}{(\rho-\rho')^{2}(K(r))^{2}}\right)^{\frac{1}{\gamma}} & \left(\int_{V_{\rho}'} |u_b|^{ \gamma}\frac{1}{\eta \om_{N}r^{N}\vdr}dxdt\right)^{\frac{1}{\gamma}} \left(\eta \om_{N} r^{N}\vdr\right)^{\frac{(\kappa-1)}{\kappa \gamma}}.
}
Hence, if we denote by $d\muf_{N+1} $ the measure $ (\eta \om_{N} r^{N}\vdr)^{-1} dxdt $, we obtain for $0<\rho'<\rho\leq 1, \hd \gamma\in (0,\frac{\tilde{p}}{\kappa}]$
\[
\|u_b\|_{L_{\gamma\kappa}(V'_{\rho'},d\muf_{N+1})} \leq \left(\frac{C}{(\rho-\rho')^{2}(K(r))^{2}(\eta \om_{N}r^{N}\vdr)^{-\frac{\kappa-1}{\kappa}}}\right)^{\frac{1}{\gamma}}
\|u_b\|_{L_{\gamma}(V'_{\rho},d\muf_{N+1})},
\]
for $0<\rho'<\rho\leq 1$ and $\gamma\in (0,\frac{\tilde{p}}{\kappa}]$. Note that $\muf_{N+1}(V_{1}')= \muf_{N+1}(U_{\sigma}')= \sigma^{N+1}\leq 1$ so that we may apply the second Moser lemma (see Lemma~\ref{moserit2}) to the result
\[
\|u_b\|_{L_{\tilde{p}}(V'_{\vs},d\muf_{N+1})} \leq \left(\frac{C^{\frac{\kappa(\kappa+1)}{\kappa-1}} \cdot 2^{\frac{2\kappa^{3}}{(\kappa-1)^{3}}} }{(1-\vs)^{\frac{2\kappa(\kappa+1)}{\kappa-1}}[(K(r))^{2}(\eta \om_{N}r^{N}\vdr)^{-\frac{\kappa-1}{\kappa}}]^{(\frac{\kappa(\kappa+1)}{\kappa-1})}}\right)^{\frac{1}{\gamma}-\frac{1}{\tilde{p}}}
\|u_b\|_{L_{\gamma}(V'_{1},d\muf_{N+1})},
\]
where  $ \vs\in (0,1)$ and $ \gamma\in (0,\frac{\tilde{p}}{\kappa}]$.  We would like to come back with normalization. To this end we multiply by  $(\eta \om_{N}r^{N}\vdr)^{\frac{1}{\tilde{p}}} =(\eta \om_{N}r^{N}\vdr)^{\frac{1}{\gamma}}(\eta \om_{N}r^{N}\vdr)^{\frac{1}{\tilde{p}}-\frac{1}{\gamma}}$. Then we have
\[
\|u_b\|_{L_{\tilde{p}}(V'_{\vs})} \leq
\left(\frac{C^{\frac{\kappa(\kappa+1)}{\kappa-1}} \cdot 2^{\frac{2\kappa^{3}}{(\kappa-1)^{3}}} }{(1-\vs )^{\frac{2\kappa(\kappa+1)}{\kappa-1}}\eta \om_{N}r^{N}\vdr [(K(r))^{2}(\eta \om_{N}r^{N}\vdr)^{-\frac{\kappa-1}{\kappa}}]^{(\frac{\kappa(\kappa+1)}{\kappa-1})}}\right)^{\frac{1}{\gamma}-\frac{1}{\tilde{p}}}
\|u_b\|_{L_{\gamma}(V'_{1})}.
\]
Observe that from (\ref{Krdopotegi}) and Lemma \ref{scaling} we get for $r \leq r^{*}$,
\eqns{
\left[(K(r))^{2}\left(\eta \om_{N}r^{N}\vdr\right)^{-\frac{\kappa-1}{\kappa}}\right]^{\frac{\kappa}{\kappa-1}} & =
(K(r))^{\frac{2\kappa}{\kappa-1}}\frac{1}{\eta \om_{N}r^{N}\vdr} \\
& = \frac{r^{\frac{2p_0}{p_0-1}}}{\eta \om_{N}\vdr \| l\|_{L_{p_0}((0,\vdr))}^{\frac{p_0}{p_0-1}}}\geq \frac{C(\overline{c},p_0)}{\eta \om_{N}}.
}
Therefore, for $r \leq r^{*}$, we have
\[
\|u_b\|_{L_{\tilde{p}}(V'_{\vs})} \leq
\left(\frac{C }{(1-\vs )^{\gamma_{0}}\eta \om_{N}r^{N}\vdr }\right)^{\frac{1}{\gamma}-\frac{1}{\tilde{p}}}
\|u_b\|_{L_{\gamma}(V'_{1})} \hd \m{ for } \hd \vs \in (0,1), \hd \gamma\in (0,\frac{\tilde{p}}{\kappa}],
\]
where $C=C(\eta,\Lambda, \delta,\nu, N, \tilde{p},p_0,q_1,q_2,\overline{c})$ and $\gamma_{0}= \gamma_{0}( N, p_0)$.  If we take $\vs=\sigma'/\sigma$, then $V_{\vs}'=U_{\sigma'}'$. Hence, if we  notice that $\nuNj(U_{1}') = \eta \om_{N} r^{N} \vdr$, then
\eqnsl{
\|u_b\|_{L_{\tilde{p}}(U_{\sigma'}')}\le \Big(\frac{C  \nuNj(U_{1}')^{-1}}{(\sigma-\sigma')^{\gamma_{0}}  }\Big)^{1/\gamma-1/\tilde{p}}
\|u_b\|_{L_{\gamma}(U'_\sigma)}, \hd \delta< \sigma'<\sigma\leq 1,  \quad 0<\gamma\le \tilde{p}/\kappa
}{L11}
and the proof is finished.

\end{proof}
\subsection{Logarithmic estimates} \label{logsection}
In this section we prove the weak $L_1$ estimates for the logarithm of positive supersolutions, which are required for the application of the Bombieri-Giusti lemma.
Theorem \ref{localweakHarnack}  follows from Theorem \ref{logest}, Theorem \ref{superest1} and Theorem \ref{superest2} by suitable application of Bombieri-Giusti lemma (Lemma \ref{abslemma}). Since, the proof is the same as the proof of Theorem 1.1 in \cite[section 3.4]{harnackdistr} we omit it. \\
The proof of Theorem \ref{logest}, in general, follows the reasoning from \cite{base} and \cite{harnackdistr}. However, dealing with the problem with unbounded inhomogeneity $f$, which is harmless in a classical parabolic problem, requires a significant change of the approach in our case. Concerning the result itself, here we assume that not only $r$ but also the number $\tau$ is small enough. It is a technical assumption, which does not influence the final regularity result (Theorem \ref{holder}). In what follows,
by $r^{*}\in (0,r_0/2)$ we mean the minimum of the number $r^{*}$ provided by Lemma \ref{scaling} and
$\Phi^{-1}(\tilde{t}_0)$, where $\tilde{t}_0$ comes from
\eqref{ak2}.

\begin{satz} \label{logest}
Let  $T>0$ and $\Omega\subset \iR^N$ be a bounded
domain. Suppose the assumptions (H1)--(H3) and (\ref{ak1}),(\ref{ak2}) are satisfied. Let $\delta,\,\eta\in(0,1)$ be fixed. Then there exists a $\tau^{*}=\tau^*(q_1, p_0, \delta,N,\nu,\Lambda,\overline{c}) > 0$ such that for any $0<r \leq r^{*}$, $0< \tau \leq \tau^{*}$, for any $t_0\ge
0$ with $t_0+\tau \vr\le T$, any ball
$B_r(x_0)\subset\Omega$, and any weak supersolution $u\geq 0$ of (\ref{MProb}) in $(0,t_0+\tau \vr)\times
B_r(x_0)$ with $u_0\ge 0$ in $B_r(x_0)$, there is a constant $c=c(u_b)$ such that
\begin{equation} \label{logestleft}
\nuNj\big(\{(t,x)\in K_-: \log u_b(t,x)>c(u_b)+\lambda\}\big)\le C
\vr \nuN(B_r)\lambda^{-1},\quad \lambda>0,
\end{equation}
and
\begin{equation} \label{logestright}
\nuNj\big(\{(t,x)\in K_+: \log u_b(t,x)<c(u_b)-\lambda\}\big)\le C
\vr \nuN(B_r)\lambda^{-1},\quad \lambda>0,
\end{equation}
where $K_-=(t_0,t_0+\eta \tau \vr)\times B_{\delta r}(x_0)$ and
$K_+=(t_0+\eta \tau \vr,t_0+\tau \vr)\times
B_{\delta r}(x_0)$. Here the constant $C$ depends only on $\delta, \eta,
\tau, N, \nu, p_0, \overline{c},\tilde{c}$, and $\Lambda$.
\end{satz}
\begin{proof} To shorten the notation we write again $B_{r} = B_{r}(x_{0})$. Let $r \in (0,r^{*}]$ and fix $1 < p \leq p_0$. Then, in particular $\Phi(r^{*}) \leq 1$. Since $u_0\ge 0$ in $B_r$ and $u$ is a nonnegative weak
supersolution we may assume without loss of generality that $u_0=0$. Furthermore, as in \cite{base} and \cite{harnackdistr}, by making a suitable time-shift, we may restrict ourselves to the case $t_0=0$ and discuss the problem on $J= [0,\tau \Phi(r)]$. We thus have
\begin{equation} \label{log1}
\int_{B_r} \Big(v \partial_t(k_{n}\ast u_b)+(h_n\ast
[ADu_b]|Dv)\Big)\,dx \ge \,\int_{B_r} h_n *f v dx,\quad\mbox{a.a.}\;t\in J,\,n\in \iN,
\end{equation}
for any nonnegative test function $v\in \oH^1_2(B_r)$.
We define the cut-off function $\psi\in C^1_0(B_r)$ such that supp$\,\psi\subset B_r$, $\psi=1$ in
$B_{\delta r}$, $0\le \psi \le 1$, $|D \psi|\le 2/[(1-\delta)r]$ and
the domains $\{x\in B_r:\psi(x)^2 \ge \hat{b}\}$ are convex for all $\hat{b}\le
1$. Then, for $t\in J$ we choose the test function $v=\psi^2 u_b^{-1}$.
 Then
\[ Dv=2\psi D\psi \,u_b^{-1}-\psi^2 u_b^{-2}Du_b.\]
Inserting this relation in (\ref{log1}) we obtain for a.a. $t\in
J$
\begin{align} \label{log1a}
-\int_{B_r} \psi^2 u_b^{-1}\partial_t & (k_{n}\ast
u_b)\,dx+\int_{B_r}\big(ADu_b|u_b^{-2}Du_b\big)\psi^2\,dx\nonumber\\ &\le
2\int_{B_r}\big(ADu_b|u_b^{-1}\psi D\psi\big)\,dx - \int_{B_r} (h_n *f)\cdot u_b^{-1} \psi^2\,dx +\mathcal{R}_n(t),
\end{align}
where
\[
\mathcal{R}_n(t)=\int_{B_r}\big(h_n\ast [ADu_b]-ADu_b|Dv\big)\,dx.
\]

By the assumption (H1) and Young's inequality, we may estimate
\[
\big|2\big(ADu_b|u_b^{-1}\psi D\psi\big)\big|\le 2\Lambda \psi
|D\psi|\,|Du_b| u_b^{-1}\le \frac{\nu}{2}\,\psi^2|Du_b|^2
u_b^{-2}+\frac{2}{\nu}\,\Lambda^2|D\psi|^2.
\]
Using this, assumption (H2) and the estimate $|D \psi|\le 2/[(1-\delta)r]$ in
(\ref{log1a}) leads to
\begin{equation} \label{log2}
-\int_{B_r} \psi^2 u_b^{-1}\partial_t(k_{n}\ast
u_b)\,dx+\frac{\nu}{2}\,\int_{B_r} |Du_b|^2 u_b^{-2} \psi^2\,dx \le \frac{8
\Lambda^2 \nuN(B_r)}{\nu (1-\delta)^2 r^2}\,+\mathcal{R}_n(t)- \int_{B_r} (h_n *f)\cdot u_b^{-1} \psi^2\,dx,
\end{equation}
for a.a. $t\in J$. Setting $w=\log u_b$, we have $Dw=u_b^{-1} Du_b$. The weighted Poincar\'e
inequality from Proposition \ref{WeiPI} with weight $\psi^2$ yields
\begin{equation} \label{log3}
\int_{B_r} (w-W)^2 \psi^2 dx \le \frac{8 r^2 \nuN(B_r)}{\int_{B_r} \psi^2
dx}\,\int_{B_r} |Dw|^2 \psi^2 dx,\quad \mbox{a.a.}\;t\in J,
\end{equation}
where
\[ W(t)=\,\frac{\int_{B_r} w(t,x) \psi(x)^2 dx}{\int_{B_r} \psi(x)^2
dx}\,,\quad\quad \mbox{a.a.}\;t\in J.
\]
The estimates (\ref{log2}) and (\ref{log3}) imply that
\begin{align*}
-\int_{B_r} \psi^2 u_b^{-1}\partial_t(k_{n}\ast u_b)
\,dx &+\,\frac{\nu \int_{B_r} \psi^2 dx}{16r^2 \nuN(B_r)}\,\int_{B_r} (w-W)^2
\psi^2 dx \\
& \le \frac{8 \Lambda^2 \nuN(B_r)}{\nu (1-\delta)^2
r^2}\,+\mathcal{R}_n(t) + \frac{1}{b}\int_{B_r} |h_n *f|\psi^2\,dx,
\end{align*}
which entails that
\begin{align}
\frac{-\int_{B_r} \psi^2 u_b^{-1}\partial_t(k_{n}\ast u_b)
\,dx}{\int_{B_r} \psi^2 dx} & +\,\frac{\nu}{16r^2 \nuN(B_r)}\,\int_{{B_{\delta r}}} (w-W)^2 dx
\nonumber\\
&\le \frac{C_1}{r^2}\,+S_n(t) + \frac{1}{\int_{B_r} \psi^2 dx}\frac{1}{b}\int_{B_r} |h_n *f|\psi^2\,dx ,
\quad \mbox{a.a.}\;t\in J,
\label{log4}
\end{align}
with some constant
$C_1=C_1(\delta,N,\nu,\Lambda)$ and
$S_n(t)=\mathcal{R}_n(t)/\int_{B_r} \psi^2 dx$.

The fundamental identity (\ref{fundidentity}) with $H(y)=-\log y$
reads (with the spatial variable $x$ being suppressed)
\begin{align*}
-u_b^{-1}\partial_t(k_{n} & \ast
u_b)=-\partial_t(k_{n}\ast \log u_b)+(\log
u_b-1)k_{n}(t)\nonumber\\
&+\int_0^t \Big(-\log u_b(t-s)+\log
u_b(t)+\frac{u_b(t-s)-u_b(t)}{u_b(t)}\Big)[-\dot{k}_{n}(s)]\,ds.
\end{align*}
Recalling $w=\log u_b$, we thus have
\begin{align} \label{log5}
-u_b^{-1}\partial_t(k_{n}  \ast u_b)= & \,
-\partial_t(k_{n}\ast
w)+(w-1)k_{n}(t)\nonumber\\
& \,\,+\int_0^t
\Psi\big(w(t-s)-w(t)\big)[-\dot{k}_{n}(s)]\,ds,
\end{align}
where $\Psi(y)=e^y-1-y$. Due to convexity of $\Psi$, it follows from
Jensen's inequality that
\[
\frac{\int_{B_r} \psi^2 \Psi\big(w(t-s,x)-w(t,x)\big)\,dx}{\int_{B_r}
\psi^2 dx} \ge \Psi \Big( \frac{\int_{B_r} \psi^2
\big(w(t-s,x)-w(t,x)\big)\,dx}{\int_{B_r} \psi^2 dx}\Big).
\]
Using this and (\ref{log5}) we obtain
\begin{align}
\frac{-\int_{B_r} \psi^2 u_b^{-1}\partial_t(k_{n}\ast u_b)
\,dx}{\int_{B_r} \psi^2 dx} & \ge -\partial_t(k_{n}\ast
W)+(W-1)k_{n}(t)\nonumber\\
&\quad +\int_0^t
\Psi\big(W(t-s)-W(t)\big)[-\dot{k}_{n}(s)]\,ds \nonumber\\
& = -e^{-W} \partial_t(k_{n}\ast e^W), \label{log6}
\end{align}
where in the last equality we applied again (\ref{log5}), but with $u_b$
replaced by $e^W$. From (\ref{log4}) and (\ref{log6}) we conclude
that
\begin{align}
\frac{\nu}{16r^2 \nuN(B_r)}\,\int_{B_{\delta r}} (w-W)^2 dx & \le e^{-W}
\partial_t(k_{n}\ast
e^W)+\,\frac{C_1}{r^2} \nonumber\\
&\quad\; +S_n(t)+\frac{1}{\int_{B_r} \psi^2 dx}\frac{1}{b}\int_{B_r} |h_n *f|\psi^2\,dx,\quad \mbox{a.a.}\;t\in J.
\label{log7}
\end{align}
The last inequality will be the starting point for our subsequent considerations.

At first, we choose
\eqq{
c(u_b) = \log \left(\frac{(k*e^{W})(\eta \tau \vr)}{A\eta \tau \vr\, k_1(\eta \tau \vr)}\right),
}{cwahl}
where $A$ is a positive constant depending only on $\tilde{c}$, which comes from (\ref{ak2}), and will be chosen later.
This definition makes sense, since $k\ast e^W\in C(J)$.
The latter is a consequence of $k\ast u_b\in C(J;L_2(B_r))$
and
\[
e^{W(t)}\le \,\frac{\int_{B_r} u_b(t,x) \psi(x)^2 dx}{\int_{B_r} \psi(x)^2
dx}\,,\quad\quad \mbox{a.a.}\;t\in J,
\]
where we again apply Jensen's inequality.

Similarly as in \cite{base} and \cite{harnackdistr}, to prove (\ref{logestleft}) and (\ref{logestright}), we use the inequalities
\begin{align}
\nuNj(\{(t,x) & \in K_-:\; w(t,x)>c(u_b)+\lambda\})\nonumber\\
\le &\;\nuNj(\{(t,x)\in K_-:
w(t,x)>c(u_b)+\lambda\;\,\mbox{and}\,\;W(t)\le c(u_b)+\lambda/2 \})\nonumber\\
&\; +\nuNj(\{(t,x)\in K_-:\,W(t)> c(u_b)+\lambda/2
\})=:I_1+I_2,\quad \lambda>0,\label{mainleft}\\
\nuNj(\{(t,x) & \in K_+:\; w(t,x)<c(u_b)-\lambda\})\nonumber\\
\le &\;\nuNj(\{(t,x)\in K_+:
w(t,x)<c(u_b)-\lambda\;\,\mbox{and}\,\;W(t)\ge c(u_b)-\lambda/2 \})\nonumber\\
&\; +\nuNj(\{(t,x)\in K_+:\,W(t)< c(u_b)-\lambda/2
\})=:I_3+I_4,\quad \lambda>0.\label{mainright}
\end{align}
We will estimate each of the four terms $I_j$ separately,
beginning with $I_1$ and $I_3$. Here, in contrast to $I_2$ and $I_4$, the reasoning does not differ much from the one carried out in \cite{base} and \cite{harnackdistr}. We nevertheless present a proof for the sake of completeness. Set $J_-:=(0,\eta
\tau \vr)$, $J_+:=(\eta\tau \vr,\tau
\vr)$.

{\bf Estimate of $I_{1}$.}
We introduce the following notation: $J_1(\lambda)=\{t\in
J_-:\,c(u_b)-W(t)+\lambda/2\ge 0\}$ and $\Omega^-_t(\lambda)=\{x\in
B_{\delta r}:\,w(t,x)>c(u_b)+\lambda\},\,t\in J_1(\lambda)$, where $c(u_b)$
is given by (\ref{cwahl}). Then
\[ w(t,x)-W(t)>c(u_b)-W(t)+\lambda\ge \lambda/2,\quad x\in
\Omega^-_t(\lambda), \hd t\in J_1(\lambda),\] and (\ref{log7}) imply that
a.e.\ in $J_1(\lambda)$
\begin{align}
&\frac{\nu}{16r^2
\nuN(B_r)}\, \,\nuN\big(\Omega^-_t(\lambda)\big)
\nonumber \\
&\le
\frac{1}{(c(u_b)-W+\lambda)^2}\,\Big(e^{-W}
\partial_t(k_{n}\ast
e^W)+\,\frac{C_1}{r^2}\,+S_n+\frac{1}{\int_{B_r} \psi^2 dx}\frac{1}{b}\int_{B_r} |h_n *f|\psi^2\,dx\Big). \label{log8}
\end{align}
Set $\chi(t,\lambda)=\nuN\big(\Omega^-_t(\lambda)\big)$, if $t\in
J_1(\lambda)$, and $\chi(t,\lambda)=0$ in case $t\in J_-\setminus
J_1(\lambda)$. Further, we introduce $H(y)=(c(u_b)-\log y+\lambda)^{-1},\,0<y\le
y_*:=e^{c(u_b)+\lambda/2}$. We have $H'(y)= (c(u_b)-\log
y+\lambda)^{-2}y^{-1}$ and
\[ H''(y)=\,\frac{1}{(c(u_b)-\log y+\lambda)^2 y^2}\,\Big(\frac{2}{c(u_b)-\log
y+\lambda}-1\Big),\quad 0<y\le y_*.\] Thus $H$ is
concave in $(0,y_*]$ whenever $\lambda\ge 4$. In
what follows, we will assume that $\lambda\ge 4$.

We extend $H$ to a $C^1$ function on $(0,\infty)$ in such a way that the extension, denoted by $\bar{H}$, is concave, $0\le\bar{H}'(y)\le
\bar{H}'(y_*),\,y_*\le y \le 2 y_*$, and $\bar{H}'(y)=0,\,y\ge 2
y_*$. Note that for $y\in(0,y_*]$
\begin{equation} \label{log8aa}
y\bar{H}'(y)=\,\frac{1}{(c(u_b)-\log y+\lambda)^2}\,\le
\,\frac{1}{(c(u_b)-\log y_*+\lambda)^2}\,\le \,\frac{4}{\lambda^2}\,\le
\,\frac{1}{\lambda},
\end{equation}
while in case $y\in[y_*,2y_*]$ we may simply estimate
\[
y\bar{H}'(y)\le 2y_*\bar{H}'(y_*)\le \,\,\frac{2}{\lambda}.
\]
Therefore
\begin{equation} \label{log8a}
 0\le y\bar{H}'(y)\le
\,\frac{2}{\lambda},\quad y>0.
\end{equation}

Next, we shall show that
\begin{equation} \label{log8b}
\bar{H}(y)\le \,\frac{3}{\lambda},\quad y>0.
\end{equation}
Indeed, by monotonicity of $\bar{H}$ and since
$\bar{H}'(y)=0$ for all $y\ge 2 y_*$, the claim follows if the
inequality is valid for all $y\in [y_*,2y_*]$. For such $y$ we
have by (\ref{log8aa}) and by concavity of $\bar{H}$
\[
\bar{H}(y)\le \bar{H}(y_*)+\bar{H}'(y_*)(y-y_*)\le
\bar{H}(y_*)+y_*\bar{H}'(y_*)\le \,\frac{3}{\lambda}.
\]
Moreover, we have
\[ e^{W(t)} H'(e^{W(t)})=\,\frac{1}{(c(u_b)-W(t)+\lambda)^2
},\quad \mbox{a.a.}\;t\in J_1(\lambda).\] Since $\bar{H}'\ge 0$
and
\[ e^{-W}
\partial_t(k_{n}\ast
e^W)+C_1 r^{-2}+S_n +\frac{1}{\int_{B_r} \psi^2 dx}\frac{1}{b}\int_{B_r} |h_n *f|\psi^2\,dx\ge 0,
\quad \mbox{a.a.}\;t\in J_-,
\]
by virtue of (\ref{log7}), we
infer from (\ref{log8}) and (\ref{log8a}) that
\begin{align} \label{log9}
&\frac{\nu}{16r^2 \nuN(B_r)}\,\,\chi(t,\lambda)  \le
e^W\bar{H}'(e^{W})\Big(e^{-W}\partial_t(k_{n}\ast
e^W)+\,\frac{C_1}{r^2}\,+S_n+\frac{1}{\int_{B_r} \psi^2 dx}\frac{1}{b}\int_{B_r} |h_n *f|\psi^2\,dx\Big)\nonumber\\
& \le \bar{H}'(e^{W})\partial_t(k_{n}\ast
e^W)+\,\frac{2C_1}{\lambda r^2}\,+\,\frac{2|S_n(t)|}{\lambda}+\frac{2}{\lambda\int_{B_r} \psi^2 dx}\frac{1}{b}\int_{B_r} |h_n *f|\psi^2\,dx,
\quad\mbox{a.a.}\; t\in J_-.
\end{align}
The concavity of $\bar{H}$, together with the fundamental identity
(\ref{fundidentity}), yields
\begin{align*}
\bar{H}'(e^{W})\partial_t(k_{n}\ast e^W) & \le
\partial_t\Big(k_{n}\ast \bar{H}\big(e^W\big)\Big)
+ \Big(-\bar{H}(e^{W})+\bar{H}'(e^{W})
e^W\Big)k_{n}\\
& \le \partial_t\Big(k_{n}\ast
\bar{H}\big(e^W\big)\Big)+\,\frac{2}{\lambda}\,k_{n},\quad
\mbox{a.a.}\; t\in J_-,
\end{align*}
which, applied in (\ref{log9}), gives a.e.\ in $J_-$
\begin{equation} \label{log10}
\frac{\nu}{16r^2 \nuN(B_r)}\,\,\chi(t,\lambda) \le
\partial_t\Big(k_{n}\ast \bar{H}\big(e^W\big)\Big)
+\,\frac{2}{\lambda}\,k_{n}+ \,\frac{2C_1}{\lambda
r^2}\,+\,\frac{2|S_n(t)|}{\lambda}+ \frac{2}{\lambda\int_{B_r} \psi^2 dx}\frac{1}{b}\int_{B_r} |h_n *f|\psi^2\,dx\,.
\end{equation}
Using (\ref{log8b}) we have
\[ \Big(k_{n}\ast
\bar{H}\big(e^W\big)\Big)(\eta\rho)\le
\,\frac{3}{\lambda}\,\int_0^{\eta\rho}k_{n}(t)\,dt.
\]
Thus, integrating (\ref{log10}) over $J_-=(0,\eta \rho)$, sending $n\to \infty$, we arrive at
\begin{align*}
\int_{J_1(\lambda)} & \nuN   \big(\Omega^-_t(\lambda)\big) \,dt = \int_0^{\eta\rho} \chi(t,\lambda)\,dt \\
&
\leq \frac{16 r^{2}\nuN(B_r)}{\nu}\left(\frac{5}{\lambda} \eta\rho \frac{(1*k)(\eta \rho)}{\eta \rho} + \frac{2C_{1}}{\lambda r^{2}} \eta \rho +\frac{2}{\lambda\int_{B_r} \psi^2 dx}\frac{1}{b}\int_{0}^{\eta \rho}\int_{B_r} |f|\psi^2\,dxdt
\right).
\end{align*}
By H\"older's inequality,
\[
\int_{0}^{\eta\rho}\frac{1}{\int_{B_r} \psi^2 dx }\int_{B_r}\frac{|f|}{b}\psi^2\,dx dt \leq (\delta r)^{-N}\frac{1}{b}\norm{f}_{L_{q_{1}}(0,T);L_{q_2}(\Omega)}(\eta\rho)^{\frac{1}{q'_1}}r^{\frac{N}{q'_2}}.
\]
Hence, recalling the definition of $b$ and using Lemma \ref{fixy} we find that
\eqq{
\int_{0}^{\eta\rho}\frac{1}{\int_{B_r} \psi^2 dx }\int_{B_r}\frac{|f|}{b}\psi^2\,dx dt \leq c(\delta,N,\eta,\tau,p_0,\overline{c})r^{-2}\Phi(r).
}{inhzz}
Thus, applying Lemma \ref{kernels} and (\ref{zn1}) we obtain
\[
I_{1}=\int_{J_1(\lambda)}\nuN   \big(\Omega^-_t(\lambda)\big) \,dt \leq c(\eta,\tau, \nu,\delta, N,p_0, \overline{c})\frac{\vr\nuN(B_r)}{\lambda},\quad\lambda\ge
4.
\]
For $\lambda <4$ we simply have
\[
I_{1} \leq |K_{-}|= (1-\eta) \tau \vr \delta^{N} \nuN(B_r) \leq 4\tau \delta^{N} \frac{\vr\nuN(B_r)}{\lambda},
\]
where we denoted by $|\cdot| $ the one-dimensional Lebesgue measure, in order to abbreviate the notation.
All in all we see that
\begin{equation} \label{I1est}
I_1\le \,c(\eta,\tau, \nu,\delta, N,p_0, \overline{c})\frac{ \vr\nuN(B_r)}{\lambda},\quad\lambda>0.
\end{equation}

{\bf Estimate of $I_{3}$.}
This estimate does not provide any new difficulties related to the consideration of a general kernel $k$ and an unbounded inhomogenity $f$. Thus we skip a large part of the proof. Here, as in \cite{base} and \cite{harnackdistr}, we shift the time by putting $s=t-\eta\rho$ and
denote the corresponding transformed functions as above by
$\tilde{W}$, $\tilde{w}$, ... and so forth. We set $\tilde{J}_+:=(0,(1-\eta)\rho)$, $J_2(\lambda)=\{s\in
\tilde{J}_+:\tilde{W}(s)-c(u_b)+\lambda/2\ge 0\}$ and
$\Omega_{s}^+(\lambda)=\{x\in B_{\delta r}:
\tilde{w}(s,x)<c(u_b)-\lambda\},\,s\in J_2(\lambda)$. Then we define
$\chi(s,\lambda)=\nuN\big(\Omega^+_{s}(\lambda)\big)$, if $s\in
J_2(\lambda)$, and $\chi(s,\lambda)=0$ in case $s\in
\tilde{J}_+\setminus J_2(\lambda)$. Performing the same calculations as in \cite{base}[Theorem 3.3.] and \cite{harnackdistr}[Theorem 3.3.] we arrive at
\begin{align*}
& \frac{\nu}{16r^2 \nuN(B_r)}\,\,\chi(s,\lambda)\\ &\le
-\partial_s\Big(k_{n}\ast
\bar{H}\big(e^{\tilde{W}}\big)\Big)
+\,\frac{6}{\lambda}\,k_{n}+\,\frac{4C_1}{\lambda
r^2}\,+\,\frac{4|\tilde{S}_n(s)|}{\lambda}+\frac{4}{\lambda\int_{B_r} \psi^2 dx}\frac{1}{b}\int_{B_r} (|h_n *f|)^{\tilde{}}\psi^2\,dx,
\end{align*}
for a.a. $s\in \tilde{J}_+$ and $\lambda \geq 1$. We integrate this estimate over
$\tilde{J}_+$, send $n\to \infty$, apply (\ref{inhzz}) and then Lemma~\ref{kernels} and (\ref{zn1}) to get
\[
I_{3}=\int_{J_2(\lambda)}\nuN\big(   \Omega^+_s(\lambda)\big) \,ds
  = \int_0^{(1-\eta)\rho} \!\!\!\!\chi(s,\lambda)\,ds
  \]
  \[
  \leq \frac{16r^2\nuN(B_r)}{\nu}\,\Big(\,\frac{6}{\lambda}\,\int_{0}^{(1-\eta)\rho} k(s)ds +\,\frac{4C_1(1-\eta)\rho}{\lambda
  r^2} + c(\delta,N,\eta,\tau,p_0,\overline{c})\frac{\vr}{\lambda r^{2}}\,\Big)
\]
\[
\leq  c(\tau , C_{1}, \nu,\delta,N, \eta, p_0, \overline{c}) \frac{\nuN(B_r) \vr}{\lambda},\quad \lambda\ge 1.
\]
 For $\lambda \leq 1$ we have
\[
I_{3}\leq |K_{+}|=(1-\eta)\tau \vr \delta^{N} \nuN(B_r)\leq \tau \delta^{N} \frac{\vr \nuN(B_r)}{\lambda},
\]
and therefore, we obtain
\begin{equation} \label{I3est}
I_3\le \, c(\tau , C_{1}, \nu,\delta, N, p_0, \overline{c}) \frac{\vr\nuN(B_r)}{\lambda},\quad\lambda>0.
\end{equation}

Now we pass to the estimates for $W$. Here, the presence of an unbounded inhomogeneity term forces us to significantly modify the approach from \cite{base} and \cite{harnackdistr}.
At first, we introduce for $\lambda>0$ the sets
$J_-(\lambda):=\{t\in J_-:\,W(t)> c(u_b)+\lambda \}$ and
$J_+(\lambda):=\{t\in J_+:\,W(t)< c(u_b)-\lambda \}$.

{\bf Estimate of $I_{2}$.}
Let us denote $\overline{w}:=e^{W}$. Multiplying (\ref{log7}) by $e^{W}$ we arrive at
\[
\partial_t (k_n * \overline{w}) + \frac{C_{1}}{r^{2}} \overline{w} + S_{n}(t)\overline{w} + \frac{1}{\int_{B_r}\psi^{2}dx}\frac{1}{b}\int_{B_r}|h_n*f|\psi^{2}dx\,\overline{w} \geq 0.
\]
Recalling  the definition of $b$ and the properties of $\psi$ we see that
\[
\frac{1}{\int_{B_r}\psi^{2}dx}\int_{B_{r}}\frac{(|(h_n*f)(t,x)|)}{b}dx \leq \frac{2\omega^{\frac{1}{q'_1}}_{N}r^{\frac{N}{q'_{2}}}||(h_{n}*f)(t,\cdot)||_{L_{q_2}(\Omega)}}{\omega_N \delta^{N}r^{N}r^{2}r^{-\frac{N}{q_{2}}}(\vdr)^{-\frac{1}{q_{1}}}\norm{f}_{L_{q_{1}}((0,T);L_{q_{2}}(\Omega))}}
\]
\[
\leq c(\delta,N)\frac{(\vdr)^{\frac{1}{q_1}}}{r^{2}}\frac{||(h_{n}*f)(t,\cdot)||_{L_{q_2}(\Omega)}}{{\norm{f}_{L_{q_{1}}((0,T);L_{q_{2}}(\Omega))}}}.
\]
Thus, we obtain
\eqq{
\partial_t (k_n * \overline{w}) +\theta_n(t) \overline{w} + S_{n}(t)\overline{w} \geq 0,
}{difeq}
where we denote
\eqq{
\theta_n(t):= \frac{C_1}{r^2}+c(\delta,N)r^{-2}(\vdr)^{\frac{1}{q_1}}\frac{\norm{(h_n*f)(t,\cdot)}_{L_{q_2}(\Omega)}}{\norm{f}_{L_{q_1}((0,T);L_{q_2}(\Omega))}}.
}{defth}
Note that $\theta_n \rightarrow \theta$ in $L_{q_1}(0,\eta\rho)$, where
\[
\theta(t):= \frac{C_1}{r^2}+c(\delta,N)r^{-2}(\vdr)^{\frac{1}{q_1}}\frac{\norm{f(t,\cdot)}_{L_{q_2}(\Omega)}}{\norm{f}_{L_{q_1}((0,T);L_{q_2}(\Omega))}}.
\]
Setting $\rho=\tau \vr$ and integrating the inequality \eqref{difeq} from $t$ to $\eta \rho$ gives
\[
(k_n * \overline{w})(\eta\rho) - (k_n * \overline{w})(t)  + \int_{t}^{\eta\rho}\theta_n(s)\overline{w}(s)ds + \int_{t}^{\eta\rho}S_{n}(s)\overline{w}(s)ds  \geq 0.
\]

We may choose a subsequence $n_m$ such that $(k_{n_m} * \overline{w})(\eta\rho)\rightarrow (k* \overline{w})(\eta\rho)$ for almost all $r$ (recall that $\rho=\tau \vr$) and $(k_{n_m} * \overline{w})(t)\rightarrow (k* \overline{w})(t)$ for almost all $t \in J_{-}$. We proceed for such $r$ and $t$. Then we obtain
\eqq{
(k* \overline{w})(t) + \int_{0}^{t}\theta(s) \overline{w}(s)ds \leq (k* \overline{w})(\eta \rho) +\int_{0}^{\eta\rho}\theta(s) \overline{w}(s)ds \hd \m{ for } \hd a.a \hd t \in (0,\eta\rho).
}{chibound}
On the other hand,
 (\ref{difeq}) may be rewritten as
 \[
 \partial_t (k_n * \overline{w}) +\theta_n(t) h_n*\overline{w} + \theta_n(t)(\overline{w} - h_n*\overline{w}) + S_{n}(t)\overline{w}=:g_n \geq 0.
 \]
 Convolving with $l$ gives
\eqq{
 h_n*\overline{w} + l*(\theta_n [h_n*\overline{w}]) + l*(\theta_n(\overline{w} - h_n*\overline{w})) + l*( S_{n}\overline{w}) = l*g_n.
}{lgn}
We multiply (\ref{lgn}) by $\theta_n$ and convolve with $l$ to the result
\begin{align*}
 l*(\theta_n [h_n*\overline{w}]) & =  l*(\theta_n [l*g_n]) - l*(\theta_n [l*(\theta_n [h_n*\overline{w}])])\\
 & \quad - l*(\theta_n[l*(\theta_n(\overline{w} - h_n*\overline{w}))])  - l*(\theta_n [l*( S_{n}\overline{w})]).
\end{align*}
Inserting this into (\ref{lgn}) leads to
\[
 h_n*\overline{w} +l*(\theta_n [l*g_n]) - l*(\theta_n l*(\theta_n h_n*\overline{w}))- l*(\theta_nl*(\theta_n(\overline{w} - h_n*\overline{w})))
 \]
 \eqq{
 - l*(\theta_n l*( S_{n}\overline{w}))+ l*(\theta_n(\overline{w} - h_n*\overline{w})) + l*( S_{n}\overline{w}) = l*g_n.
}{lgna}

Using only the fact that $g_n \geq 0$, we will show that for $n$ large enough, $\tau$ small enough and $t \in (0,\eta\tau \vr)$
\[
l*(\theta_n [l* g_n])(t) \leq  \frac{1}{2}l*g_n(t).
\]
Indeed, we have
\begin{align*}
l*(\theta_n [l* g_n])(t) & = \int_{0}^{t}l(t-y)\theta_n(y)\int_{0}^{y}l(y-\xi)g_n(\xi)d\xi dy \\
& \leq \norm{\theta_n}_{L_{q_{1}}(0,\eta\tau\Phi(r))} \left(\int_{0}^{t}\left[l(t-y)\int_{0}^{y}l(y-\xi)g_n(\xi)d\xi\right]^{q'_{1}} dy\right)^{\frac{1}{q'_1}}\\
&\leq \norm{\theta_n}_{L_{q_{1}}(0,\eta\tau\Phi(r))}  \int_{0}^{t}g_n(\xi)\left(\int_{\xi}^{t}[l(t-y)l(y-\xi)]^{q'_{1}} dy\right)^{\frac{1}{q'_1}} d\xi,
\end{align*}
where we applied H\"older's and Minkowski's inequality. We note that
\begin{align*}
&\left(\int_{\xi}^{t}[l(t-y)  l(y-\xi)]^{q'_{1}} dy\right)^{\frac{1}{q'_1}}  \leq
\left(\int_{\frac{\xi+t}{2}}^{t}(l(t-y))^{q'_{1}} dy\right)^{\frac{1}{q'_1}}l(\frac{\xi+t}{2} - \xi)\nonumber\\ &\quad\quad\quad\quad\quad\quad \quad\quad +
\left(\int_{\xi}^{\frac{\xi+t}{2}}(l(y-\xi))^{q'_{1}} dy\right)^{\frac{1}{q'_1}}l(t-\frac{\xi+t}{2})\\
&\quad\quad\quad\quad \leq 2l(\frac{t-\xi}{2})\left(\int_{0}^{\frac{t-\xi}{2}}(l(y))^{q'_{1}} dy\right)^{\frac{1}{q'_1}} \leq
 2l(\frac{t-\xi}{2})\left(\int_{0}^{t}(l(y))^{q'_{1}} dy\right)^{\frac{1}{q'_1}}\\
 & \quad\quad\quad\quad\leq 2l(\frac{t-\xi}{2})\left(\int_{0}^{(1-\eta)\tau \Phi(r)}(l(y))^{q'_{1}} dy\right)^{\frac{1}{q'_1}}.
\end{align*}
Using Remark \ref{ak31}, Lemma \ref{lpoint} and Lemma \ref{scaling} we deduce that
\begin{align*}
\left(\int_{\xi}^{t}[l(t-y)l(y-\xi)]^{q'_{1}} dy\right)^{\frac{1}{q'_1}} & \leq c(\overline{c},p_0)\,l(\frac{t-\xi}{2}) \tau^{\frac{1}{p'_0}-\frac{1}{q_1}} \norm{l}_{L_{q'_1}(0,\Phi(r))} \\& \leq c(\overline{c},p_0)\, \tau^{\frac{1}{p'_0}-\frac{1}{q_1}}  r^{2}(\Phi(r))^{-\frac{1}{q_1}} l(t-\xi).
\end{align*}
Furthermore, using that $\eta \in (0,1)$ we obtain for $\tau \in (0,1)$ and $n$ large enough that
\begin{align*}
\norm{\theta_n}_{L_{q_{1}}(0,\eta\tau\Phi(r))} &\leq 2\norm{\theta}_{L_{q_{1}}(0,\eta\tau\Phi(r))} \leq c(q_{1})\Big[\frac{C_{1}}{r^{2}}(\Phi(r))^{\frac{1}{q_{1}}} + c(\delta,N)\frac{(\vdr)^{\frac{1}{q_{1}}}}{r^{2}}\Big]\\
& \leq c(q_1,\delta,C_{1},N)\frac{(\vdr)^{\frac{1}{q_{1}}}}{r^{2}}.
\end{align*}
The previous estimates and $C_{1}=C_{1}(\delta,N,\nu,\Lambda)$ now lead to
\eqq{
l*(\theta_n l* g_n)(t) \leq  c(p_0, q_1,\delta,\nu,\Lambda, N, \overline{c}) \tau^{\frac{1}{p'_0}-\frac{1}{q_1}}\int_{0}^{t}l(t-\xi)g_n(\xi)d\xi.
}{iterest}
We recall that by (H3) $\frac{1}{p'_0} >\frac{1}{ q_1}$. Thus, there exists a $\tau^{*}\in (0,1)$ depending on $q_1,p_0,\delta,\nu,\Lambda, N, \overline{c}$ such that for sufficiently large $n$, any $\tau \in (0,\tau^{*}]$ and for $t \in (0,\eta\tau \Phi(r))$, there holds
\[
l*(\theta_n l* g_n)(t) \leq \frac{1}{2}(l*g_n)(t).
\]
Inserting this estimate in (\ref{lgna}) and using that $l*(\theta_n l*(\theta_n h_n*\overline{w}))\ge 0$ we arrive at
\begin{align}
 h_n*\overline{w} \geq & \;\frac{1}{2}l*g_n
 +l*(\theta_n l*( S_{n}\overline{w}))- l*( S_{n}\overline{w})\nonumber\\
 & \;- l*(\theta_n(\overline{w} - h_n*\overline{w}))+l*(\theta_n l*(\theta_n(\overline{w} - h_n*\overline{w}))),
\label{i2lim}
\end{align}
for sufficiently large $n$. Now, we would like to pass to the limit with $n$.
 Note that
 \[
 \norm{l*(\theta_n l*(\theta_n(\overline{w} - h_n*\overline{w})))}_{L_{1}(0,\eta\rho)}  \leq \norm{l}_{L_{1}(0,\eta\rho)}\norm{\theta_n}_{L_{q_1}(0,\eta\rho)}\norm{l*(\theta_n(\overline{w} - h_n*\overline{w}))}_{L_{q'_1}(0,\eta\rho)}
 \]
 \[
 \leq  \norm{l}_{L_{1}(0,\eta\rho)}\norm{\theta_n}^{2}_{L_{q_1}(0,\eta\rho)}\norm{l}_{L_{q'_1}(0,\eta\rho)}\norm{\overline{w} - h_n*\overline{w}}_{L_{q'_1}(0,\eta\rho)} \rightarrow 0\quad \m{ as } n\rightarrow \infty.
 \]
 Thus, the two last terms on the right-hand side of (\ref{i2lim}) tend to zero in $L_{1}((0,\eta\rho))$.
 Noting that $S_n \rightarrow 0$ in $L_{2}((0,\eta\rho))$ and $\theta_n \rightarrow \theta$ in $L_{q_1((0,\eta\rho))}$,
 from (\ref{lgn}) we obtain that $l*g_n \rightarrow \overline{w} + l*(\theta \overline{w})$ in $L_{1}((0,\eta\rho))$. Thus, passing to the limit with $n$ in the $L_1$- sense in (\ref{i2lim})  we arrive at
\[
 \overline{w} \geq \frac{1}{2}( \overline{w} + l*(\theta \overline{w})).
\]

Let us denote $\Psi:= \overline{w} + l*(\theta \overline{w})$. Then obviously $\overline{w} \leq \Psi$, and by (\ref{chibound}) $(k*\Psi)(t) \leq (k*\Psi)(\eta\rho)$ for almost all $t \in (0,\eta\rho)$.
Then, for $\lambda>0$ we obtain
\begin{align*}
e^\lambda | J_-(\lambda) | & =
e^\lambda\big|\{t\in J_-:\,
e^{W(t)}>e^{c(u_b)}e^{\lambda}\}\big|=\int_{J_-(\lambda)}e^\lambda \,dt\\
& \le \int_{J_-(\lambda)}e^{W(t)-c(u_b)} \,dt\le
\int_{J_-}e^{W(t)-c(u_b)} \,dt =\frac{A \eta \rho\, k_1(\eta\rho) }{(k*e^{W})(\eta \rho)} \int_{0}^{\eta \rho}e^{W(t)}dt. \\
& = A\eta\rho\,\frac{(1*\overline{w})(\eta\rho)}{(1*l)(\eta\rho) (k*\overline{w})(\eta\rho)} \leq 2A \eta\rho\,\frac{(l*k*\Psi)(\eta\rho)}{(1*l)(\eta\rho) (k*\Psi)(\eta\rho)}\\
& \leq 2A \eta\rho\,\frac{(1*l)(\eta\rho) (k*\Psi)(\eta\rho)}{(1*l)(\eta\rho) (k*\Psi)(\eta\rho)} = 2A \eta\rho.
\end{align*}

Hence, we have
\begin{equation} \label{I2est}
I_2=| J_-(\lambda/2)| \nuN( B_{\delta r})\le\,\frac{c(\tau,\delta,\eta,\tilde{c})}{\lambda}\,\vr\nuN(B_r),\quad \lambda>0.
\end{equation}

{\bf Estimate of $I_{4}$.} Finally, we come to $I_4$. For $m>0$ we introduce the family of functions $H_m$ on $\iR$ as follows: $H_m(y)=y$, $y\le m$, and $H_m(y)=m+(y-m)/(y-m+1)$, $y\ge m$.
Such construction provides that $H_m$ is increasing, concave, and bounded above by $m+1$.
Further, we have $H_m\in C^1(\iR)$, and by concavity
\begin{equation} \label{Hmprop}
0\le yH_m'(y)\le H_m(y)\le m+1,\quad y\ge 0.
\end{equation}
Then, the fundamental identity (\ref{fundidentity}) implies that
\[
\partial_t\Big(k_{n}\ast
H_m\big(e^W\big)\Big) \geq H_{m}'(e^{W})\partial_t\Big(k_{n}\ast
e^W\Big),
\]
and thus multiplying (\ref{log7}) by $e^W H_m'\big(e^W\big)$ and employing
 the last inequality we find that for a.a.\ $t\in J$
\begin{equation} \label{West1}
\partial_t\Big(k_{n}\ast
H_m\big(e^W\big)\Big)+\,\frac{C_1}{r^2}\,H_m\big(e^W\big)\ge - S_n
e^W H_m'\big(e^W\big) -\frac{H_m(e^{W})}{\int_{B_r} \psi^2 dx}\frac{1}{b}\int_{B_r} |h_n *f|\psi^2\,dx.
\end{equation}

We multiply (\ref{West1}) by a nonnegative $\Psi \in C^{1}(\overline{J_+})$ with $\Psi(\rho) = 0$ and integrate over $J_{+}$. Then
\[
-\Psi(\eta\rho)(k_n*H_{m}(e^{W}))(\eta\rho)-\int_{J_{+}} \dot{\Psi} [ k_n*H_{m}(e^{W})]dt
\]
\eqq{
+ \int_{J_{+}}\Psi \left(\,\frac{C_1}{r^2}\,H_m\big(e^W\big) + \frac{H_m(e^{W})}{\int_{B_r} \psi^2 dx}\frac{1}{b}\int_{B_r} |h_n *f|\psi^2\,dx\right)dt \geq -\int_{J_{+}}  S_n
e^W H_m'\big(e^W\big) \Psi dt.
}{etat}
Similarly as in the estimate of $I_2$, we have
\[
\frac{1}{b}\frac{H_m(e^{W})}{\int_{B_r} \psi^2 dx}\int_{B_r} (|h_n *f|)\psi^2\,dx \leq
\frac{c(\delta,N)r^{-2}(\vdr)^{\frac{1}{q_1}} H_m(e^{W})}{\norm{f}_{L_{q_1}((0,T);L_{q_2}(\Omega))}}\norm{h_n*f}_{L_{q_2}(\Omega)}.
\]
Applying this estimate in (\ref{etat}) and then sending $n\to \infty$ (and taking an appropriate subsequence if necessary) we arrive at
\[
-\Psi(\eta\rho)(k*H_{m}(e^{W}))(\eta\rho)-\int_{J_{+}} \dot{\Psi} [k*H_{m}(e^{W})]\,dt
\]
\[
+ \int_{J_{+}}\Psi(t) \left(\,\frac{C_1}{r^2}\,H_m\big(e^W\big)(t) + c(\delta,N)r^{-2}(\vdr)^{\frac{1}{q_1}}\frac{H_m(e^{W})(t)}{\norm{f}_{L_{q_1}((0,T);L_{q_2}(\Omega))}}\norm{f(t,\cdot)}_{L_{q_2}(\Omega)}\right)dt \geq 0.
\]

We perform a time-shift by the change of variables $s=t-\eta\tau
\vr=t-\eta \rho$, for $t \in J_{+}$. We denote $\tilde{g}(s)=g(s+\eta \rho)$,
$s\in (0,(1-\eta)\rho)$, for functions $g$ defined on $J_+$.
Then
\[
-\Psi(\eta\rho)(k*H_{m}(e^{W}))(\eta\rho)-\int_{0}^{\rho-\eta\rho} \dot{\tilde{\Psi}}[ k*H_{m}(e^{\tilde{W}})]\,ds - \int_{0}^{\rho-\eta\rho} \dot{\tilde{\Psi}}(s)  \int_{0}^{\eta\rho}k(s+\eta\rho-\sigma)H_{m}(e^{W(\sigma)})d\sigma ds
\]
\eqq{
+ \int_{0}^{\rho-\eta\rho}\tilde{\Psi}(s) \left(\,\frac{C_1}{r^2}\,H_m\big(e^{\tilde{W}}\big) + c(\delta,N)r^{-2}(\vdr)^{\frac{1}{q_1}}\frac{H_m(e^{\tilde{W}})}{\norm{f}_{L_{q_1}((0,T);L_{q_2}(\Omega))}}\norm{\tilde{f}(s,\cdot)}_{L_{q_2}(\Omega)}\right)ds \geq 0.
}{etas}
In order to abbreviate the notation, let us denote
\[
\tilde{u}_{m} = H_m\big(e^{\tilde{W}}\big), \hd \hd  u_m= H_m\big(e^{W}\big)
\]
and
\[
\Upsilon_{m}(s) = \int_0^{\eta\rho}\big[-\dot{k}(s+\eta\rho-\sigma)\big]
u_m(\sigma)\,d\sigma, \hd \hd \theta(s) = \frac{C_1}{r^2}+c(\delta,N)r^{-2}(\vdr)^{\frac{1}{q_1}}\frac{\norm{\tilde{f}(s,\cdot)}_{L_{q_2}(\Omega)}}{\norm{f}_{L_{q_1}((0,T);L_{q_2}(\Omega))}}.
\]
Observe that
\begin{equation} \label{Upsbound}
0\le \Upsilon_{m}(s) \leq \norm{u_m}_{L_{\infty}((0,\eta\rho))}[k(s)-k(s+\eta\rho)]\le \norm{u_m}_{L_{\infty}((0,\eta\rho))}k(s),
\end{equation}
thus $\Upsilon_{m} \in L_{1}((0,(1-\eta)\rho))$. Hence we may integrate by parts the third term on the left-hand side of (\ref{etas}) to the result
\[
-\int_{0}^{\rho-\eta\rho} \dot{\tilde{\Psi}}  [k*\tilde{u}_{m}]ds
+ \int_{0}^{\rho-\eta\rho}\tilde{\Psi} \theta \tilde{u}_{m}\, ds \geq   \int_{0}^{\rho-\eta\rho} \tilde{\Psi} \Upsilon_m\, ds.
\]

Now, we proceed as in the proof of \cite[Lemma 3.1]{Za2}. We take $\tilde{\Psi}(s) = \int_{s}^{\rho-\eta\rho}h_n(\sigma-s)\varphi(\sigma)d\sigma$ with a nonnegative $\varphi \in C^{1}([0,\rho-\eta\rho])$, $\varphi(\rho-\eta\rho) = 0$ and arbitrary $n \in \mathbb{N}$. Then, by Fubini's theorem
\[
-\int_{0}^{\rho-\eta\rho} \dot{\varphi} [k_n*\tilde{u}_{m}]\,ds
+ \int_{0}^{\rho-\eta\rho}\varphi [ h_n*(\theta \tilde{u}_{m})]\, ds \geq   \int_{0}^{\rho-\eta\rho} \varphi [h_n*\Upsilon_m]\, ds.
\]
Integrating first by parts and then using the fact that $\varphi$ is arbitrary, we arrive at
\eqq{
\partial_{s}(k_n*\tilde{u}_m) + h_{n}*(\theta\tilde{u}_{m}) \geq h_n*\Upsilon_m \quad\mbox{a.e. on}\; (0,(1-\eta)\rho).
}{Westt2}
\nic{
\begin{equation} \label{West2}
\partial_s\Big(k_{n}\ast
H_m\big(e^{\tilde{W}}\big)\Big)+\,\frac{C_1}{r^2}\,H_m\big(e^{\tilde{W}}\big)\ge
\Upsilon_{n,m}(s)- \tilde{S}_n e^{\tilde{W}}
H_m'\big(e^{\tilde{W}}\big) - \frac{1}{b}\frac{H_m(e^{\tilde{W}})}{\int_{B_r} \psi^2 dx}\int_{B_r} (|h_n *f|)^{\tilde{}}\psi^2\,dx,
\end{equation}
with the history term
\[
\Upsilon_{n,m}(s)=\int_0^{\eta\rho}\big[-\dot{k}_{n}(s+\eta\rho-\sigma)\big]
H_m\big(e^{W(\sigma)}\big)\,d\sigma, \hd s\in (0,(1-\eta)\rho).
\]

Applying the last estimate together with (\ref{Hmprop}) in (\ref{West2}) leads to
\begin{equation} \label{Westt1}
\partial_s\Big(k_{n}\ast
H_m\big(e^{\tilde{W}}\big)\Big)+\theta_n(s)\,H_m\big(e^{\tilde{W}}\big)\ge
\Upsilon_{n,m}(s)- \abs{\tilde{S}_n}
H_m\big(e^{\tilde{W}}\big),
\end{equation}
where $\theta_n$ is defined as in (\ref{defth}).
In order to abbreviate the notation further, let us denote
\[
\tilde{u}_{m} = H_m\big(e^{\tilde{W}}\big), \hd \hd  u_m= H_m\big(e^{W}\big).
\]
Then the inequality (\ref{Westt1}) may be written as follows
\eqq{
\partial_s (k_{n}*\tilde{u}_{m})(s) + \theta_n(s)\tilde{u}_m(s) \geq \Upsilon_{n,m}(s)- \abs{\tilde{S}_n(s)}\tilde{u}_m(s), \hd \hd s \in (0,(1-\eta)\rho),
}{Westt2}
where
\[
\tilde{u}_m \in L_{\infty}(0,(1-\eta)\rho), \hd \m{ and } \hd  \Upsilon_{n,m} \in L_{\infty}(0,(1-\eta)\rho),
\]
because
\[
|\Upsilon_{n,m}(s)| \leq \norm{u_m}_{ L_{\infty}(0,\eta\rho)}[k_{n}(s+\eta\rho)-k_{n}(s)]
\]
and $k_{n}\in H^{1}_{1}(0,T)$.
}

We would like to compare $\tilde{u}_{m}$ with the solution to
\eqq{
v(s) + l*(\theta v)(s) = (l*\Upsilon_{m})(s), \hd \hd s \in (0,(1-\eta)\rho).
}{Westt3}

Equation (\ref{Westt3}) has exactly one solution belonging to $L_{p_{0}}((0,(1-\eta)\rho))$ (see \cite[Chapter 9]{GLS} or use a simple fixed point argument).
\nic{Indeed, let us define operator $P$
\[
(Pv)(s) = (l*\Upsilon_{m})(s) - l*(\theta v)(s).
\]
$P$ maps $L_{p_{0}}(0,(1-\eta)\rho)$ into itself because
\[
\norm{l*\Upsilon_{m}}_{L_{p_{0}}(0,(1-\eta)\rho)} \leq \norm{l}_{L_{p_{0}}(0,(1-\eta)\rho)}\norm{\Upsilon_{m}}_{L_{1}(0,(1-\eta)\rho)}
\]
and
\[
 \norm{l*(\theta v)} _{L_{p_{0}}(0,(1-\eta)\rho)} \leq \norm{l}_{L_{p_{0}}(0,(1-\eta)\rho)}\norm{\theta v}_{L_{1}(0,(1-\eta)\rho)} \leq \norm{l}_{L_{p_{0}}(0,(1-\eta)\rho)}\norm{v}_{L_{p_{0}(0,(1-\eta)\rho)}}
\norm{\theta}_{L_{p'_{0}(0,(1-\eta)\rho)}}
\]
\[
\leq \norm{l}_{L_{p_{0}}(0,(1-\eta)\rho)}\norm{v}_{L_{p_{0}(0,(1-\eta)\rho)}}
\norm{\theta}_{L_{q_{1}(0,(1-\eta)\rho)}},
\]
where we applied (H3). Similarly, we have for $v_{1},v_{2} \in L_{p_{0}}(0,(1-\eta)\rho)$
\[
\norm{Pv_{1} - Pv_{2}}_{L_{p_{0}}(0,(1-\eta)\rho)} \leq \norm{l}_{L_{p_{0}}(0,(1-\eta)\rho)}
\norm{\theta}_{L_{q_{1}(0,(1-\eta)\rho)}}\norm{v_1-v_2}_{L_{p_{0}(0,(1-\eta)\rho)}},
\]
so $P$ is contractive on $L_{p_{0}}(0,(1-\eta)\rho)$ for small $r$ (recall that $\rho = \vr\tau)$.
Thus, there exists exactly one solution of (\ref{Westt3}) belonging to $L_{p_{0}}(0,(1-\eta)\rho)$.}

Convolving (\ref{Westt3}) with $k_{n}$ yields
\[
(k_{n}*v)(s) + h_{n}*1*(\theta v)(s) = (h_{n}*1*\Upsilon_{m})(s).
\]
Evidently, $h_{n}*1*(\theta v)$, $h_{n}*1*\Upsilon_{m} \in H_{1}^{1}((0,(1-\eta)\rho))$ and thus also $k_{n}*v\in H_{1}^{1}((0,(1-\eta)\rho))$ and we may differentiate to obtain
\eqq{
\partial_{s}(k_{n}*v)(s) + h_{n}*(\theta v)(s) = (h_{n}*\Upsilon_{m})(s).
}{Westt4}
Now, we subtract (\ref{Westt4}) from (\ref{Westt2})
\[
\partial_{s}(k_{n}*(\tilde{u}_{m}-v))(s) + \theta(s)(\tilde{u}_{m}-v)(s) \geq y_n,
\]
where
\[
y_{n}(s) := y_{n}(\tilde{u},v)(s):=h_{n}*(\theta v)(s) - \theta(s) v(s) + \theta(s)\tilde{u}_m(s)-h_n*(\theta \tilde{u}_m)(s).
\]
Note that $y_{n}\rightarrow 0$ in $L_{1}((0,(1-\eta)\rho))$.
To abbreviate the notation further, let us write $g := \tilde{u}_{m}-v.$ Here, we skip the index $m$ since $m$ is fixed in this part of the argument.
Then we may write
\eqq{
\partial_{s}(k_{n}*g)(s) + \theta(s)g(s) \geq y_{n}(s), \quad s \in (0,(1-\eta)\rho).
}{Westt5}
\nic{
Indeed, we have
\[
\tilde{S_{n}}\rightarrow 0 \hd \m{ in } \hd L_{2}(0,(1-\eta)\rho), \hd \hd  h_{n}*(\theta v)\rightarrow \theta v \hd \m{ in } \hd L_{1}(0,(1-\eta)\rho),\hd \theta_{n} \rightarrow \theta \m{ in } \hd L_{q_1}(0,(1-\eta)\rho)
\]
and
\[
\norm{\Upsilon_{n,m} - (h_n*\Upsilon_{m})}_{L_{1}(0,(1-\eta)\rho)} \leq
\norm{\Upsilon_{n,m} - \Upsilon_{m}}_{L_{1}(0,(1-\eta)\rho)} +
\norm{\Upsilon_{m} - (h_n*\Upsilon_{m})}_{L_{1}(0,(1-\eta)\rho)}\rightarrow 0.
\]
}
In order to compare $\tilde{u}_m$ with $v$, we would like to obtain from (\ref{Westt5}) that $g$ is nonnegative. However, since in (\ref{Westt5}) we deal with the quite delicate space $L_1$ and have an unbounded variable coefficient, we cannot apply known comparison arguments. Thus, we proceed very carefully.

Differentiating the convolution we have
\[
k_{n}(0)g(s) + \int_{0}^{s}\dot{k}_{n}(s-\xi)g(\xi)d\xi + \theta(s)g(s) \geq y_{n}(s),
\]
and since $k_{n}$ and $\theta$ are positive it follows that
\eqq{
g(s) \geq \frac{y_{n}(s)}{k_{n}(0)+\theta(s)} + \frac{1}{k_{n}(0)+\theta(s)}\int_{0}^{s}[-\dot{k}_{n}](s-\xi)g(\xi)d\xi.
}{Westt6}
Let us denote
\[
(K_{n}g)(s):= \frac{1}{k_{n}(0)+\theta(s)}\int_{0}^{s}[-\dot{k}_{n}](s-\xi)g(\xi)d\xi.
\]
Applying the positive operator $K_n$ $M>1$ times to (\ref{Westt6}) we arrive at
\eqq{
g(s) \geq \sum_{i=0}^{M}(K_{n}^i \frac{y_{n}}{k_{n}(0)+\theta(\cdot)})(s) + (K_{n}^{M+1}g)(s).
}{Westt7}

Consider now a fixed $n$. Since $\theta(s) \geq C_{1}r^{-2}$, we have
\begin{align*}
\norm{K_{n}g}_{L_{1}((0,(1-\eta)\rho))} & \leq \frac{1}{k_{n}(0)+C_{1}r^{-2}}\norm{[-\dot{k}_{n}]*g}_{L_{1}((0,(1-\eta)\rho))} \\
& \leq \frac{1}{k_{n}(0)+C_{1}r^{-2}}\norm{-\dot{k}_{n}}_{L_{1}((0,(1-\eta)\rho))}\norm{g}_{L_{1}((0,(1-\eta)\rho))}
\end{align*}
and
\[
\norm{-\dot{k}_{n}}_{L_{1}((0,(1-\eta)\rho))} = \int_{0}^{(1-\eta)\rho}[-\dot{k}_{n}](\xi)\,d\xi = k_{n}(0) - k_{n}((1-\eta)\rho) \leq k_{n}(0),
\]
which shows that
\[
\norm{K_{n}g}_{L_{1}((0,(1-\eta)\rho))} \leq  \frac{k_{n}(0)}{k_{n}(0)+C_{1}r^{-2}}\norm{g}_{L_{1}((0,(1-\eta)\rho))},
\]
and thus
\[
\norm{K^{i}_{n}g}_{L_{1}((0,(1-\eta)\rho))} \leq  \left(\frac{k_{n}(0)}{k_{n}(0)+C_{1}r^{-2}}\right)^{i}\norm{g}_{L_{1}((0,(1-\eta)\rho))}
\quad \mbox{for all}\;i \in \mathbb{N}.
\]
Therefore, for any fixed $n$, $K_{n}^{M}g\rightarrow 0$ in $L_{1}((0,(1-\eta)\rho))$ as $M\rightarrow \infty$. Passing to the limit in (\ref{Westt7}) we arrive at
\[
g(s) \geq \sum_{i=0}^{\infty}(K_{n}^i \frac{y_{n}}{k_{n}(0)+\theta(\cdot)})(s).
\]
Note that for every $i \in \mathbb{N}$
\[
\norm{K_{n}^i \frac{y_{n}}{k_{n}(0)+\theta(\cdot)}}_{L_{1}((0,(1-\eta)\rho))} \leq \norm{y_{n}}_{L_{1}((0,(1-\eta)\rho))}\frac{1}{k_{n}(0)+C_{1}r^{-2}}\left(\frac{k_{n}(0)}{k_{n}(0)+C_{1}r^{-2}}\right)^{i},
\]
and thus
\[
\norm{\sum_{i=0}^{\infty}(K_{n}^i \frac{y_{n}}{k_{n}(0)+\theta(\cdot)})}_{L_{1}((0,(1-\eta)\rho))} \leq \norm{y_{n}}_{L_{1}((0,(1-\eta)\rho))}\frac{1}{k_{n}(0)+C_{1}r^{-2}}\sum_{i=0}^{\infty}\left(\frac{k_{n}(0)}{k_{n}(0)+C_{1}r^{-2}}\right)^{i}
\]
\[
= \norm{y_{n}}_{L_{1}((0,(1-\eta)\rho))}\frac{1}{k_{n}(0)+C_{1}r^{-2}}\frac{1}{1-\frac{k_{n}(0)}{k_{n}(0)+C_{1}r^{-2}}}=
\norm{y_{n}}_{L_{1}((0,(1-\eta)\rho))} \frac{1}{C_{1}r^{-2}}.
\]
Hence
\[
\sum_{i=0}^{\infty}(K_{n}^i \frac{y_{n}}{k_{n}(0)+\theta(\cdot)}) \rightarrow 0 \quad \m{ in }\; L_{1}((0,(1-\eta)\rho))\quad \mbox{as}\;n\to \infty.
\]
This shows that $g\ge 0$ a.e.\ in $(0,(1-\eta)\rho)$,
which in turn gives the desired inequality
\begin{equation} \label{comparison}
 \tilde{u}_{m} \geq v \hd \m{ a.e. in } \hd (0,(1-\eta)\rho).
\end{equation}

Now, we would like to estimate the solution $v$ of (\ref{Westt3}) from below. We will use a similar reasoning as for the estimate of $I_2$.
Multiplying (\ref{Westt3}) by $\theta$ and convolving with $l$ we have
\[
l*(\theta v) + l*(\theta(l*\theta v)) = l*(\theta l* \Upsilon_{m}),
\]
so
\eqq{
v = l*\Upsilon_{m} - l*(\theta l* \Upsilon_{m}) + l*(\theta(l*\theta v)).
}{Westt8}
We claim that for sufficiently small $\tau$ there holds
\begin{equation} \label{halfestimate}
l*(\theta l* \Upsilon_{m})(s) \leq \frac{1}{2} l*\Upsilon_{m}(s),\quad s\in(0,(1-\eta)\tau \vr).
\end{equation}
In fact, proceeding as for the estimate of $I_2$ we arrive at (\ref{iterest}) with $g_n \geq 0$ replaced by $\Upsilon_m \geq 0$.
By (H3), $\frac{1}{p'_0} >\frac{1}{ q_1}$, and thus there exists a $\tau^{*}>0$ depending only on $q_1,\delta,\nu,\Lambda, N, \overline{c}$ such that \eqref{halfestimate} holds true for any $\tau \in (0,\tau^{*}]$.

On the other hand, it is not difficult to see (cf.\ also \cite[Chapter 9]{GLS}) that $v$ may be represented as
\[
v(s) = \sum_{i=0}^{\infty}(-1)^{i}\big(l*[\theta\cdot]\big)^{(i)} l* \Upsilon_{m},
\]
where the series converges not only in $L_{p_0}$, but even in $L_\infty((0,(1-\eta)\tau \vr))$. In fact, in view of
\eqref{Upsbound} and by \eqref{pc} we have
\[
0\le l\ast {\Upsilon}_m \le \norm{u_m}_{L_{\infty}((0,\eta\rho))},
\]
which, together with \eqref{halfestimate}, yields
\begin{align*}
0\le \big(l*[\theta\cdot]\big)^{(i)} l* \Upsilon_{m}\le \frac{1}{2^{i}}\,l* \Upsilon_{m}\le \frac{1}{2^{i}}\,
\norm{u_m}_{L_{\infty}((0,\eta\rho))},\quad i\in\iN.
\end{align*}
Writing the series as
\begin{equation} \label{seriesdiff}
v(s)=\sum_{i=0}^{\infty}\big(l*[\theta\cdot]\big)^{(2i)}\Big( l* \Upsilon_{m}-l*(\theta [l* \Upsilon_{m}])
\Big)
\end{equation}
and using \eqref{halfestimate} shows that $v$ is nonnegative, and thus from (\ref{Westt8}) (or alteratively from
\eqref{seriesdiff}) it follows that
\[
v(s) \geq \frac{1}{2} (l* \Upsilon_{m})(s),\quad s \in (0,(1-\eta)\tau \vr),
\]
provided that $\tau \leq \tau^{*}$. Combining the last
inequality and \eqref{comparison} yields
\[
H_{m}(e^{\tilde{W}})(s) \geq \frac{1}{2} (l* \Upsilon_{m})(s).
\]

By Fubini's theorem, we have
\eqq{
(l \ast
\Upsilon_{m})(s) = \int_{0}^{\eta\rho}H_m\big(e^{W(\xi)}\big)\int_{0}^{s} l(s-\sigma)\big[-\dot{k}(\sigma+\eta\rho-\xi)\big]d\sigma d\xi, \hd s\in (0,(1-\eta)\rho).
}{West33}
In order to estimate the inner integral we make use of the property that $r^{*}$ satisfies $\Phi(r^{*}) \leq \tilde{t}_0$, where $\tilde{t}_0$ comes from the assumption (\ref{ak2}), cf.\ the definition of $r^{*}$
prior to the statement of Theorem \ref{logest}. Then using estimate~(\ref{ak2prim}), we obtain 
\[
\int_{0}^{s} l(s-\sigma)\big[-\dot{k}(\sigma+\eta\rho-\xi)\big]d\sigma
=s \izj l(s(1-w))\big[-\dot{k}(sw+\eta\rho-\xi)\big]dw
\]
\[
\geq  s \,l(s) \izj \big[-\dot{k}(sw+\eta\rho-\xi)\big]dw =s\,l(s)\frac{1}{s}[k(\eta\rho-\xi) - k(s+\eta\rho-\xi)]
\]
\[
\geq  \tilde{c}\, l(s) k(\eta\rho-\xi)\frac{s}{s+\eta\rho-\xi}.
\]
It follows that
\[
H_{m}(e^{\tilde{W}(s)})
\geq \frac{\tilde{c}}{2}l(s)\int_0^{\eta\rho} H_m\big(e^{W(\xi)}\big) k(\eta\rho-\xi)\,d\xi \,\frac{s}{s+\eta\rho}.
\]
Evidently, $H_m(y)\nearrow y$ as $m\to \infty$ for all $y\in \iR$.
Thus, by sending $m\to \infty$ and applying Fatou's lemma we conclude
that
\begin{equation} \label{West5}
e^{\tilde{W}(s)}\ge \frac{\tilde{c}\, s\, l(s)}{2\rho}\big(k\ast
e^{W}\big) (\eta\rho),\quad \mbox{a.a.}\;s\in (0,(1-\eta)\rho).
\end{equation}

Let $\lambda>0$. We employ (\ref{West5}) to estimate as follows.
\[
\lambda |J_+(\lambda)| = \int_{J_+(\lambda)} \lambda dt
 \le \int_{J_+(\lambda)} \big(c(u_b)-W(t)\big)\, dt= \int_{J_+(\lambda)-\eta\rho} \big(c(u_b)-\tilde{W}(s) \big)\,ds\\
\]
\[
\leq \int_{J_+(\lambda)-\eta\rho} \Big[\log \left(\frac{(k*e^{W})(\eta \rho)}{ A\eta\rho\,k_1(\eta\rho)}\right)- \log \left(\frac{\tilde{c}(k*e^{W})(\eta \rho)sl(s)}{2\rho}\right)\Big]\,ds
= \int_{J_+(\lambda)-\eta\rho} \log \left(\frac{2(1*l)(\eta\rho)}{A\tilde{c}\eta sl(s)}\right)ds.
\]
We note that for any  $A \leq \frac{2}{\tilde{c}}$ the expression under the integral is nonnegative on the whole interval $(0,(1-\eta)\rho)$. Indeed, using Lemma \ref{ba} we have for any $s \in (0,(1-\eta)\rho)$
\[
\frac{(1*l)(\eta \rho)}{\eta sl(s)} \geq \frac{(1*l)(\eta \rho)}{\eta (1*l)(\rho)} =
\frac{k_1(\rho)}{\eta k_1(\eta\rho)} \geq 1.
\]
We choose $A=\frac{2}{ \tilde{c}}$ and continue estimating as follows.
\[
\lambda |J_+(\lambda)| \leq \int_{0}^{(1-\eta)\rho} \log \left(\frac{ (1*l)(\eta\rho)}{\eta sl(s)} \right)ds \leq
\int_{0}^{(1-\eta)\rho} \log \left(\frac{ \overline{c}(1*l)(\eta\rho)}{\eta (1*l)(s)} \right)ds
\]
\[
= \rho c(\overline{c},\eta) +(1-\eta)\rho \log \left((1*l)(\eta \rho)\right) - \int_{0}^{(1-\eta)\rho}\log \left((1*l)(s)\right)ds
\]
\[
=\rho c(\overline{c},\eta) + (1-\eta)\rho \log \left( (1*l)(\eta \rho)\right) - (1-\eta)\rho \log \big( (1*l)((1-\eta) \rho)\big) + \int_{0}^{(1-\eta)\rho} \frac{sl(s)}{(1*l)(s)}ds
\]
\[
\leq  \rho c(\overline{c},\eta) +(1-\eta)\rho \log \left(\frac{k_1((1-\eta) \rho)}{k_1(\eta \rho)}\right) + (1-\eta)\rho,
\]
where in the second inequality we applied (\ref{ak1}) with $p=1$.
Applying Lemma (\ref{ba})  we arrive at

\eqnsl{
\lambda|J_{+}(\lambda)| \leq
c(\eta,\tau,\bar{c}) \vr.
}{estidwadwa}
Hence,
\begin{equation} \label{I4est}
I_4=|J_+(\lambda/2)|\nuN(
B_{\delta r})\le\,\frac{c(\eta,\tau,\bar{c})
\delta^N}{\lambda}\,\vr\nuN(B_r),\quad \lambda>0.
\end{equation}

Finally, combining (\ref{mainleft}), (\ref{mainright}), and
(\ref{I2est}), (\ref{I4est}), (\ref{I1est}), (\ref{I3est})
we obtain the assertion.
\end{proof}

Having established Theorem \ref{superest1}, Theorem \ref{superest2} and Theorem \ref{logest}, the final step of the proof of Theorem \ref{localweakHarnack} is now the same as in the proof of Theorem 1.1 in \cite[Section 3.4]{harnackdistr}.
 $\square$


\section{Proof of the H\"older regularity}
In this section we prove Theorem \ref{holder}. The methods utilized in this section are analogous to those used in \cite{harnackdistr}, i.e.\ we will deduce the H\"older regularity of weak solutions to (\ref{MProb}) from the weak Harnack estimate. However, since we consider a more general setting, some modifications are necessary. The proof heavily relies on the key assumption (\ref{ak1}), which is crucial for Proposition \ref{estiprop} and Lemma \ref{fixy}. Furthermore, to control the behaviour of the truncated derivative of $k$ we use the additional assumption (\ref{asholder}).

In order to simplify the notation we define $\bvr:=\Phi(2r)$. Furthermore, in this section, by $l$ we denote an integer number, not to confuse with the kernel $l$ associated with the kernel $k$.

Let $u$ be a bounded weak solution to
\eqq{
\partial_t (k*(u-u_{0}))-\mbox{div}\,\big(A(t,x)Du\big)=f \m{ in } (0,2\eta\bvr)\times B(x_{0},2r),
}{osc4}
where we assume that $u_{0} \in L_{\infty}(B(x_{0},2r))$, $r\in (0,r^{*}]$  and $r^{*}$ comes from Theorem \ref{localweakHarnack}. Note that this ensures that $\bvr \leq 1$.
As before, we also often write $B_r(x)$ instead of $B(x,r)$.
We define $F(t,x) = f(t,x) + u_{0}(x)k(t)$ and normalize the equation by setting
\[
v=\frac{u}{2D}, \quad G = \frac{F}{2D},
\]
\eqq{
D=\|u\|_{L_{\infty}((0,2\eta\bvr)\times B_{2r}(x_{0}))} +  r^{2-\frac{N}{q_2}}(\bvr)^{-\frac{1}{q_1}}\|F\|_{L_{q_1}((\frac{\eta}{2}\bvr,2\eta\bvr);L_{q_2}( B_{2r}(x_{0})))}.}{osc3}
Then $v$ is a weak solution to
\[
\partial_t (k* v)-\mbox{div}\,\big(A(t,x)Dv\big)=G \m{ in } (0,2\eta\bvr)\times B(x_{0},2r)
\]
and
\eqq{
\|v\|_{L_{\infty}((0,2\eta\bvr)\times B_{2r}(x_{0}))} \leq \frac{1}{2}, \hd \|G\|_{L_{q_1}((\frac{\eta}{2}\bvr,2\eta\bvr);L_{q_2}( B_{2r}(x_{0})))}\leq \frac{(\bvr)^{\frac{1}{q_1}}}{2 r^{2-\frac{N}{q_2}}},\hd  \eosc_{(0,2\eta\bvr)\times B_{2r}(x_{0})} v \leq 1.
}{vges}

We emphasize that in (\ref{vges}), similarly as in \cite{harnackdistr}, it is crucial to take a sub-cylinder that has positive
distance from the initial time for the $G$-term.
We fix $(t_{1},x_{1}) \in (\eta\bvr,2\eta\bvr) \times B_{r}(x_{0})$ and $\theta \in (0,1)$, $\theta \leq 2\tau^{*}$, where $\tau^{*}$ comes from Theorem \ref{localweakHarnack}. Let further $Q_{dom}:=(0,2\eta\bvr) \times B_{2r}(x_{0})$. In order to prove a suitable oscillation estimate we will consider a family of nested cylinders.
If $r_0$ given by (\ref{r0}) is infinite the construction of these cylinders does not differ from the one introduced in \cite{harnackdistr}. Otherwise, we have to slightly modify the approach, having in mind that the domain of $\Phi$ is bounded. The modification is mainly of technical nature.
Thus, if $r_0$ is finite, we choose $\tilde{r} = \tilde{r}(\eta,\theta) > 0$ such that
\eqq{
\tilde{r} = \frac{1}{16}\sqrt{\frac{\theta}{\eta}}r_0.
}{rtilde}
Then, we consider the family of nested cylinders for $r \in (0,r^{**})$, where $r^{**}:=\min\{\tilde{r},r^{*}\}$,
\[
Q(\rho) = (t_{1}-\theta \overline{\Phi}(\rho r),t_{1}) \times B_{\rho r}(x_{1}), \hd \rho = 2^{-l}, \hd l \in \mathbb{Z},\hd l \geq -3-\left\lfloor \log_{2}\left(\sqrt{\frac{\eta}{\theta}}\right)\right\rfloor.
\]
In case $r_0 = \infty$, we just consider $r \in (0,r^{*})$ and the entire family of cylinders indexed with $l \in \mathbb{Z}$.
   Let $\tilde{l} \geq 0$ denote the integer that corresponds to the largest of those cylinders $Q(2^{-l})$ that are properly  contained in  $Q_{dom}$. Then
\eqq{
|x_{1}-x_0|+2^{-\tl}r < 2r \m{\hd  and \hd  } t_{1}-\theta\overline{\Phi}(2^{-\tl}r) > 0.
}{alter2}

\no Applying Proposition \ref{philambda} one can show, in the same way as in \cite{harnackdistr}, that there exists $\gamma = \gamma(\theta,\eta) \geq 0$ such that
\eqq{
\tilde{l} \leq \gamma(\theta,\eta).
}{ltildees}
\nic{Indeed, if $\tl$ corresponds to the largest cylinder contained in $Q_{dom}$, it means that $Q(2^{-(\tl-1)})$ is not contained in $Q_{dom}$. Hence, either
\eqq{
|x_{1}-x_{0}|+2^{-(\tl-1)}r \geq  2r \m{\hd  or \hd  } t_{1}-\theta\overline{\Phi}(2^{-(\tl-1)}r) \leq  0.
}{alter1}
In the first case we get
\[
2^{-(\tl-1)} \geq  2-\frac{|x_{1}-x_{0}|}{r} > 1, \m{ thus } \tl< 1,
\]
while in the second case
\[
\overline{\Phi}(2^{-(\tl-1)}r) \geq  \frac{t_{1}}{\theta} > \frac{\eta \bvr}{\theta}.
\]
If $\tilde{l} \geq 1$, then applying Proposition (\ref{philambda}) we obtain
\[
\frac{\eta \bvr}{\theta} < \overline{\Phi}(2^{-(\tl-1)}r) \leq 2^{-2(\tl-1)}\overline{\Phi}(r)
\]
and thus

\[
\tl \leq 1 + \log_{4}\max\left\{\frac{\theta}{\eta},1 \right\}
\]
and we arrive at (\ref{ltildees}).
}

We next introduce the numbers $\kappa \in (0,1)$ and $l_{0}\geq \tl, l_{0} \in \mathbb{N}$, that will both be chosen later and define
\[
a_{l}:=\essinf_{Q_{dom}\cap Q(2^{-l})}v, \hd \hd  b_{l} = a_{l}+2^{-(l-l_{0})\kappa} \hd \m{ for } \hd l \leq l_{0}.
\]
Then, by definition, for all $j \leq l_{0}$, $j \in \mathbb{Z}$ ($j \geq  -3-\left\lfloor \log_{2}\left(\sqrt{\frac{\eta}{\theta}}\right)\right\rfloor$ in case $r_0 < \infty$) we have
\eqq{a_{j} \leq v \leq b_{j} \hd \m{ a.e. in  } \hd Q_{dom}\cap Q(2^{-j}), \hd b_{j}-a_{j} = 2^{-(j-l_{0})\kappa},}{abclaim}
because $\eosc_{Q_{dom}} v \leq 1$ and $b_{j}-a_{j} \geq 1$ for $j \leq l_{0}$. We would like to construct a  nondecreasing sequence $a_{j}$ and a nonincreasing sequence $b_{j}$ such that the property (\ref{abclaim}) is satisfied for $j > l_{0}$.
The proof is carried out by induction. Firstly, we note that for $j\leq l_{0}$ the condition (\ref{abclaim}) trivially holds. Now, let $l \geq l_{0}$ and  assume that (\ref{abclaim}) holds for all $j\leq l$. From this, we will deduce that there exist $a_{l+1} \geq a_{l}$ and $b_{l+1}\leq b_{l}$ such that (\ref{abclaim}) holds  also for $j=l+1$.

We set $\tilde{t} = t_{1} - \theta \overline{\Phi}(2^{-l}r)$. Then $(\tilde{t},t_{1}) \times B(x_{1},2^{-l}r) = Q(2^{-l})$. At first, we will establish an estimate for
the memory term, analogous to \cite{harnackdistr}[Lemma 4.1].
\begin{lemma} Let $l\ge l_0$ and
suppose that (\ref{abclaim}) holds for all $j \leq l$  ($j \geq -3-\left\lfloor \log_{2}\left(\sqrt{\frac{\eta}{\theta}}\right)\right\rfloor$ in case $r_0 < \infty$), and put $m_{l} := \frac{a_{l}+b_{l}}{2}$. Then we have
\eqq{
|v(t,x)-m_{l}| \leq (b_{l}-m_{l})\left(2\cdot 2^{\kappa}\left(\frac{1}{\theta}\frac{\ki(t_{1}-\tilde{t})}{\ki(t_{1}-t)}\right)^{\frac{\kappa}{2}}-1\right)\quad \m{ for a.a. } (t,x) \in (0,\tilde{t})\times B_{2^{-l}r}(x_{1}).
}{history}
\end{lemma}
\begin{proof}
Fix $(t,x) \in (0,\tilde{t})\times B_{2^{-l}r}(x_{1})$ and notice that $\frac{t_{1}-t}{\theta} \geq \frac{t_{1}-\tilde{t}}{\theta} = \overline{\Phi}(2^{-l}r)$. Since $\overline{\Phi}$ is increasing, continuous and onto $[0,\infty)$ and by (\ref{rtilde}), there exists $l_{*} \leq l$ such that
\eqq{
\overline{\Phi}(2^{-l_{*}}r) \leq \frac{t_{1}-t}{\theta} < \overline{\Phi}(2^{-(l_{*}-1)}r).
}{lstar}
Indeed, if $r_0=\infty$, the claim is obvious. In case $r_0 < \infty$, it is enough to find $l_{*}-1 \geq -3-\left\lfloor \log_{2}\left(\sqrt{\frac{\eta}{\theta}}\right)\right\rfloor$ such that the second estimate in (\ref{lstar}) holds. If there exists $l_* > 1$ that satisfies (\ref{lstar}) the proof is finished. Otherwise, from Proposition \ref{philambda}
\[
\overline{\Phi}(2^{-(l_{*}-1)}r) \geq \bvr 2^{-2(l_{*}-1)}.
\]
On the other hand
\[
\frac{t_{1}-t}{\theta} \leq \frac{2\eta\bvr}{\theta}.
\]
Hence, the second inequality in (\ref{lstar}) is satisfied for $l_{*} \leq - \left\lceil\log_{2}\left(\sqrt{\frac{\eta}{\theta}}\right)\right\rceil$.
Thus (\ref{lstar}) is proven and in particular we have
\[
t_{1}-\theta \overline{\Phi}(2^{-(l_{*}-1)}r) <t < t_{1}.
\]

The inequalities above, combined with $B_{2^{-l}r}(x_{1})\subset B_{2^{-(l_{*}-1)}r}(x_{1})$, imply that $(t,x) \in Q(2^{-(l_{*}-1)})$. Since $l_{*}-1 < l$ we may apply the induction hypothesis to get
\[
v(t,x) - m_{l} \leq b_{l_{*}-1} - m_{l} \leq b_{l_{*}-1}-a_{l_{*}-1}+a_{l} - m_{l} = 2^{-(l_{*}-1-l_{0})\kappa} - \frac{1}{2} 2^{-(l-l_{0})\kappa}
\]
\[
=(b_{l}-m_{l})(2\cdot 2^{-(l_{*}-1-l)\kappa}-1).
\]
We recall that $t_{1}-\tilde{t} = \theta \overline{\Phi}(2^{-l}r)$ and $t_{1}-t \geq \theta \overline{\Phi}(2^{-l_{*}}r)$. Using this, together with (\ref{zn1}), Lemma \ref{ba} and the fact that $\ki$ is decreasing and $\theta \in (0,1)$, we have
\eqq{
2^{-2(l_{*}-l)} = \frac{\frac{1}{4}2^{2l}r^{-2}}{\frac{1}{4}2^{2l_{*}}r^{-2}}
=\frac{\ki(\overline{\Phi}(2^{-l}r))}{\ki(\overline{\Phi}(2^{-l_{*}}r))} \leq \frac{1}{\theta}
\frac{\ki(\theta\overline{\Phi}(2^{-l}r))}{\ki(\theta\overline{\Phi}(2^{-l_{*}}r))}
\leq   \frac{1}{\theta} \frac{\ki(t_{1}-\tilde{t})}{\ki(t_{1}-t)}.
}{twoes}
This way we obtain the upper bound. Analogously,
\[
v(t,x) - m_{l} \geq a_{l_{*}-1} - m_{l} \geq a_{l_{*}-1}-b_{l_{*}-1}+b_{l} - m_{l} = -2^{-(l_{*}-1-l_{0})\kappa} + \frac{1}{2} 2^{-(l-l_{0})\kappa}
\]
\[
=-(b_{l}-m_{l})(2\cdot 2^{-(l_{*}-1-l)\kappa}-1),
\]
and applying (\ref{twoes}) we obtain the lower bound.
This proves the lemma.
\end{proof}
Next, we will construct smaller cylinders inside $Q(2^{-l})$. As in \cite{harnackdistr},
we introduce $\theta_{1},\theta_{2}$ such that $\frac{1}{4} < \theta_{1}<\theta_{2}<1$ and we set
\eqq{
t_{**} = t_{1}-\theta\theta_{2}\overline{\Phi}(2^{-l}r), \hd  t_{*} = t_{1}-\theta\theta_{1}\overline{\Phi}(2^{-l}r), \hd Q^{-} :=(t_{**},t_{*})\times B_{2^{-(l+1)}r}(x_{1}).
}{ab3}
Applying Proposition \ref{philambda}, we can show, analogously as in \cite{harnackdistr}, that
\eqq{
\tilde{t} < t_{**}<t_{*}<t_{1}-\theta\overline{\Phi}(2^{-(l+1)}r).
}{tstars}
Thus, $Q^{-}$ and $Q(2^{-(l+1)}) = (t_{1}-\theta\overline{\Phi}(2^{-(l+1)}r),t_{1}) \times B_{2^{-(l+1)}r}(x_{1})$ are disjoint and contained in $Q(2^{-l})$.
\nic{
Only the last inequality in (\ref{tstars}) needs explanation. By the definition of $t_{*}$,  it is equivalent with
\[
\overline{\Phi}(2^{-(l+1)}r) < \theta_{1}\overline{\Phi}(2^{-l}r).
\]
Actually we will show that there exists $b \in (0,1)$ which depends only on $\theta_{1}$ such that
\eqq{
\overline{\Phi}(2^{-(l+1)}r) < b\theta_{1}\overline{\Phi}(2^{-l}r).
}{boxdist}
\nic{Since $\ki$ is decreasing the last one is equivalent with
\[
\ki(b\theta_{1}\overline{\Phi}(2^{-l}r)) < \ki(\overline{\Phi}(2^{-(l+1)}r)) = 2^{2l}r^{-2}.
\]
We note that
\[
\ki(b\theta_{1}\overline{\Phi}(2^{-l}r))  \leq (b\theta_{1})^{-1}\ki(\overline{\Phi}(2^{-l}r)) = (b\theta_{1})^{-1}\cdot 2^{2l}\frac{1}{4}r^{-2}.
\]
Since $\theta_{1} \in (\frac{1}{4},1)$, one may find $b \in (0,1)$ such that
\[
(b\theta_{1})^{-1}\cdot 2^{2l}r^{-2} < 4\cdot 2^{2l}r^{-2}
\]
i.e. for $b\in (1/4\theta_{1},1)$ the inequality (\ref{boxdist}) holds and we arrive at (\ref{tstars}).}

From Proposition \ref{philambda} we infer that
\[
\overline{\Phi}(2^{-(l+1)}r)  \leq \frac{1}{4}\overline{\Phi}(2^{-l}r)
\]
Since $\theta_{1} \in (\frac{1}{4},1)$, one may find $b \in (0,1)$ such that (\ref{boxdist}) holds.
}

We will now discuss the two cases $(A)$ and $(B)$:
\eqq{
(A) \hd \nuNj(\{(t,x)\in Q^{-}:v(t,x) \leq m_{l}\}) \geq \frac{1}{2}\nuNj(Q^{-})
}{A}
\eqq{
(B)\hd  \nuNj(\{(t,x)\in Q^{-}:v(t,x) \leq m_{l}\}) \leq \frac{1}{2}\nuNj(Q^{-}).
}{B}
In both cases, we apply the weak Harnack inequality for a certain shifted problem with cylinders $Q_{-} \supseteq Q^{-}-t_{**}$ and $Q_{+} \supseteq Q(2^{-(l+1)})-t_{**}$, where $\cdot-t_{**}$ denotes the shift only in time variable. This will lead to the required estimates in the cylinder $Q(2^{-(l+1)})$.

Suppose (A) holds. Set $w = b_{l}-v$. Then, the induction hypothesis implies $w \geq 0$ on $Q(2^{-l})$.   Moreover,
\eqq{
\partial_{t}(k*w)(t,x) - \divv (A(t,x)D w(t,x)) = b_{l}\cdot k(t) - G(t,x)
}{trzy1}
in a weak sense for $(t,x) \in (0,2\eta\bvr)\times B_{2r}(x_{0})$.
Let $t \in (t_{**},t_{1})$. We make a time-shift, setting $s=t-t_{**}$ and  $\ti{w}(s,x)= w(s+\tss,x)$. Then we have
\[
(k*w)\hspace{0.05cm} \ti{} \hspace{0.05cm}  (s,x)= (k*w)(s+\tss,x) = \left(\int_{0}^{\ti{t}}+\int_{\ti{t}}^{\tss}+\int_{\tss}^{s+\tss} \right)k(s+\tss - \tau )w(\tau,x)d\tau
\]
\[
= \int_{0}^{\ti{t}}  k(s+\tss - \tau )w(\tau,x)d\tau +\int_{0}^{\tss - \ti{t}}  k(s+(\tss -\ti{t})- p )w(p+\tit,x )dp + \int_{0}^{s} k(s-p )\ti{w}(p,x)dp,
\]
where in the second integral we substitute $p:=\tau - \tit$ and in the third $p:=\tau - \tss$ and \m{$s\in (0, t_{1}-\tss)$.} Differentiating with respect to $s$ yields
\begin{align*}
\partial_{s}(k*w)\hspace{0.05cm} \ti{} \hspace{0.05cm}  (s,x) & = \int_{0}^{\ti{t}}  \dot{k}(s+\tss - \tau )w(\tau,x)d\tau\\
&\quad +\int_{0}^{\tss - \ti{t}}  \dot{k}(s+(\tss -\ti{t})- p )w(p+\tit,x )dp + \partial_{s} (k*\ti{w})(s,x).
\end{align*}
Thus, from (\ref{trzy1}) we obtain

\[
\partial_{s}(k*\tilde{w})(s,x) - \divv(\tilde{A}(s,x)D\tilde{w}(s,x)) = -\int_{0}^{\tss - \ti{t}}  \dot{k}(s+(\tss -\ti{t})- p )w(p+\tit,x )dp
\]
\[
-\int_{0}^{\ti{t}}  \dot{k}(s+\tss - \tau )w(\tau,x)d\tau + b_{l}k(t_{**}+s) - \tilde{G}(s,x), \hd \hd s \in (0,t_{1}-t_{**}).
\]
Note that $(\tilde{t},t_{**}) \times B_{2^{-l}r}(x_{1}) \subset Q(2^{-l})$, because $\tit = t_{1}- \theta \overline{\Phi}(2^{-l}r)$ and $\tss<t_{1}$. Since $w \geq 0 $ on $Q(2^{-l})$, the first term on the RHS is nonnegative and we see that  $\tilde{w}$ satisfies (in a weak sense)
\[
\partial_{s}(k*\tilde{w})(s,x) - \divv(\tilde{A}(s,x)D\tilde{w}(s,x))
\]
\eqq{ \geq
-\int_{0}^{\ti{t}}  \dot{k}(s+\tss - \tau )w(\tau,x)d\tau + b_{l}k(t_{**}+s) - \tilde{G}(s,x)=:\Psi(s,x), \hd (s,x) \in (0,t_{1}-t_{**})\times B_{2^{-l}r}(x_{1}).
}{defPsi}
Hence
\eqq{
\partial_{s}(k*\tilde{w}) - \divv(\tilde{A}D\tilde{w}) \geq -\Psi^{-} \hd \m{ in \hd } (0,t_{1}-t_{**})\times B_{2^{-l}r}(x_{1}),
}{nad}
where $\Psi=\Psi^{+}- \Psi^{-}$ and $\Psi^{-}\geq 0$ denotes the negative part of $\Psi$.


Since $w$ is nonnegative in $Q(2^{-l})$,  $\tw(s,x)=w(t_{**}+s, x)\geq 0$ on $\Qdl-t_{**}$. Furthermore,
\begin{align*}
\Qdl-t_{**} & =(t_{1}-\theta \vkrl - t_{**}, t_{1}-t_{**})\times \blrj\\
& =(\theta(\theta_{2}-1) \vkrl , \theta \theta_{2}\vkrl)\times \blrj .
\end{align*}
In particular, $\tw \geq 0 $ on $(0, t_{1}-t_{**})\times \blrj $. Thus, $\tw$ is a nonnegative weak supersolution to (\ref{nad}) in $(0, t_{1}-t_{**})\times \blrj $. Therefore, we may apply Theorem~\ref{localweakHarnack} to $\tw$ with the parameters $u_{0}:=0$, $r:=2^{-l}r$, $\delta:=\frac{3}{4}$, $x_{0}:=x_{1}$, $p:=1$ and we obtain
\eqq{\frac{1}{\nuNj(Q_{-})}\int_{Q_{-}} \tw dxds \leq C\left[ \essinf_{Q_{+}} \tw + 2^{-l(2-\frac{N}{q_2})}r^{2-\frac{N}{q_2}}(\overline{\Phi}(2^{-l}r))^{-\frac{1}{q_1}} \| \Psi^{-} \|_{L_{q_1}((0,2\tau \vkrl);L_{q_2} (\blrj)}
\right],
}{zharnacka}
where
\[
Q_{-}:=Q_{-}(0,x_{1}, \dml, \frac{3}{4} )= \left(0,\frac{3}{4} \tau \vkrl \right)\times B_{2^{-(l+1)}\cdot \frac{3}{2}r} (x_{1}),
\]
\[
Q_{+}:=Q_{+}(0,x_{1}, \dml, \frac{3}{4} )= \left(\frac{5}{4} \tau \vkrl, 2\tau \vkrl \right)\times B_{2^{-(l+1)}\cdot \frac{3}{2}r} (x_{1})
\]
and $C=C(\Lambda, \tau , \nu, N,p_0,q_1,q_2,\overline{c},\tilde{c})$, provided $2\tau \vkrl \leq \theta \theta_{2}\vkrl $, i.e. $\tau \leq \tau^{*}$ and $2 \tau  \leq \theta \theta_{2}$. Since $\essinf_{Q_{+}}\tw = b_{l}- \esssup_{Q_{+}}\tv = b_{l}- \esssup_{Q_{+}+t_{**}} v$, we have
\begin{align}
&\frac{1}{\nuNj(Q_{-})} \int_{Q_{-}} \tw \,dx\,ds\nonumber\\
&\leq  C\left[ b_{l}-\esssup_{Q_{+}+t_{**}} v + 2^{-l(2-\frac{N}{q_2})}r^{2-\frac{N}{q_2}}(\overline{\Phi}(2^{-l}r))^{-\frac{1}{q_1}} \| \Psi^{-} \|_{L_{q_1}((0,2\tau \vkrl);L_{q_2} (\blrj)}
\right].
\label{ab1}
\end{align}
As in \cite{harnackdistr} we choose $\theta_2 \in (\max\{\theta_1,\frac{2}{3}\},1)$  and $\tau = \jd \theta\theta_{2}$. Then, due to Proposition \ref{philambda} we have
\eqq{\Qdll \subseteq Q_{+}+t_{**}.}{zawieranie}
Note that such choice of $\tau \in (0,\tau^{*}]$ is possible since we assumed $\theta \leq 2 \tau^{*}$.

Combining (\ref{ab1}) and (\ref{zawieranie}) we arrive at the following estimate
\begin{align}
&\frac{1}{\nuNj(Q_{-})}  \int_{Q_{-}}  \tw \,dx\,ds \nonumber\\
& \leq  C\left[ b_{l}-\esssup_{\Qdll} v +2^{-l(2-\frac{N}{q_2})}r^{2-\frac{N}{q_2}}(\overline{\Phi}(2^{-l}r))^{-\frac{1}{q_1}}  \| \Psi^{-} \|_{L_{q_1}((0,2\tau \vkrl);L_{q_2} (\blrj)}
\right],
\label{ab2}
\end{align}
where $C=C(\Lambda, \nu,N,\theta,\theta_2,p_0,q_1,q_2,\overline{c},\tilde{c})$.

Now, we would like to estimate from below the term on the left-hand side of (\ref{ab2}) by $(b_{l}-m_{l})$. Note that $\nuNj(Q_{-})=\left(\frac{3}{2}\right)^{N+1}\frac{\theta_{2}}{4(\theta_{2}-\theta_{1})}\nuNj(Q^{-})$ (see (\ref{ab3})) and from the assumption (A) we obtain
\begin{align*}
\jd \nuNj(Q_{-}) &= \left(\frac{3}{2}\right)^{N+1}\frac{\theta_{2}}{4(\theta_{2}-\theta_{1})} \jd \nuNj(Q^{-})\\
& \leq  \left(\frac{3}{2}\right)^{N+1}\frac{\theta_{2}}{4(\theta_{2}-\theta_{1})} \nuNj(\{(t,x)\in Q^{-}:\hd v(t,x)\leq m_{l} \}).
\end{align*}
Thus,
\begin{align*}
 \left(\frac{2}{3}\right)^{N+1} &  \frac{2(\theta_{2}-\theta_{1})}{\theta_{2}}\nuNj(Q_{-}) \leq  \nuNj(\{(s,x)\in Q^{-}-\tss:\hd \tv (s,x)\leq m_{l} \})\\
& = \nuNj(\{(s,x)\in Q^{-}-\tss:\hd b_{l}-m_{l} \leq \tw (s,x) \})\\
& =  (b_{l}-m_{l})^{-1}
\int_{ \{(s,x)\in Q^{-}-\tss:\hd b_{l}-m_{l}
\leq \tw (s,x) \}  } (b_{l}- m_{l}) dxds\\
&\leq (b_{l}-m_{l})^{-1}
\int_{ Q^{-}-\tss  } \tw(s,x) dxds
\leq  (b_{l}-m_{l})^{-1}
\int_{  Q_{-}  } \tw(s,x) dxds,
\end{align*}
provided
\eqq{Q^{-}-\tss \subseteq Q_{-}.}{pop1}
The inclusion (\ref{pop1}) holds iff $\frac{5}{8}\theta_{2}\leq \theta_{1}$. Assuming further that $\theta_{1}$ and $\theta_{2}$ also  satisfy the last condition we arrive at
\eqq{\jd (b_{l}-m_{l})\leq \left(\frac{3}{2}\right)^{N+1} \frac{\theta_{2}}{4(\theta_{2}-\theta_{1})} \frac{1}{\nuNj(Q_{-})} \int_{Q_{-}} \tw dxds.}{ab4}
The estimates (\ref{ab2}) and (\ref{ab4}) lead to
\[
\jd (b_{l}-m_{l}) \leq  C\left[ b_{l}-\esssup_{\Qdll} v+2^{-l(2-\frac{N}{q_2})}r^{2-\frac{N}{q_2}}(\overline{\Phi}(2^{-l}r))^{-\frac{1}{q_1}} \| \Psi^{-} \|_{L_{q_1}((0,2\tau \vkrl);L_{q_2} (\blrj)}
\right],
\]
and thus
\eqq{
\esssup_{\Qdll } v \leq b_{l}+2^{-l(2-\frac{N}{q_2})}r^{2-\frac{N}{q_2}}(\overline{\Phi}(2^{-l}r))^{-\frac{1}{q_1}}\| \Psi^{-} \|_{L_{q_1}((0,2\tau \vkrl);L_{q_2} (\blrj)}  - \frac{1}{2C} (b_{l}-m_{l}),
}{wazne2}
where $C=C(\Lambda, \nu , N, \theta, \theta_{1}, \theta_{2},p_0,q_1,q_2,\overline{c},\tilde{c})$.

Let us estimate $\Psi^{-}$, which was defined in (\ref{defPsi}). Employing the estimate (\ref{history}) we have
\[
\int_{0}^{\tilde{t}}-\dot{k}(s+t_{**}-\tau)w(\tau,x) d\tau =
\int_{0}^{\tilde{t}}-\dot{k}(s+t_{**}-\tau)(b_{l}-v(\tau,x))d\tau
\]
\[
= \int_{0}^{\tilde{t}}-\dot{k}(s+t_{**}-\tau)[b_{l}-m_{l}+m_{l}-v(\tau,x)]d\tau
\]
\[
\geq -(b_{l}-m_{l})\int_{0}^{\tilde{t}}-\dot{k}(s+t_{**}-\tau)\left[2\cdot 2^{\kappa}\left(\frac{1}{\theta}\frac{\ki(t_{1}-\tilde{t})}{\ki(t_{1}-\tau)}\right)^{\frac{\kappa}{2}}-2\right]d\tau
\]
\[
=  -(b_{l}-a_{l})\int_{0}^{\tilde{t}}-\dot{k}(s+t_{**}-\tau)\left[2^{\kappa}\left(\frac{1}{\theta}\frac{\ki(t_{1}-\tilde{t})}{\ki(t_{1}-\tau)}\right)^{\frac{\kappa}{2}}-1\right]d\tau
\]
\[
=-(b_{l}-a_{l})(t_{1}-\tilde{t})\int_{1}^{\frac{t_{1}}{t_{1}-\tilde{t}}}-\dot{k}(t_{**}-t_{1}+s+p(t_{1}-\tilde{t}))\left[ 2^{\kappa}\left(\frac{1}{\theta}\frac{\ki(t_{1}-\tilde{t})}{\ki(p(t_{1}-\tilde{t}))}\right)^{\frac{\kappa}{2}}-1\right]dp
\]
\[
\geq -(b_{l}-a_{l})(t_{1}-\tilde{t})\int_{1}^{\frac{t_{1}}{t_{1}-\tilde{t}}}-\dot{k}\left(\left[\frac{t_{**}-t_{1}}{t_{1}-\tilde{t}}+p\right](t_1 - \tilde{t})\right)[2^{\kappa}\left(\frac{p}{\theta}\right)^{\frac{ \kappa}{2}}-1]dp,
\]
where in the last inequality we used
\[
\ki(p(t_{1}-\tilde{t})) \geq p^{-1}\ki(t_{1}-\tilde{t}).
\]
which follows from Lemma \ref{ba} with $x=(t_1 - \tilde{t})p, y = p^{-1}$,  $p \geq 1$.
Since $\frac{t_{**}-t_{1}}{t_{1}-\tilde{t}} = -\theta_{2}$ and $t_{1}-\tilde{t} = \theta\overline{\Phi}(2^{-l}r)$, we thus obtain
\[
\int_{0}^{\tilde{t}}-\dot{k}(s+t_{**}-\tau)w(\tau,x) d\tau \geq -(b_{l}-a_{l})\theta\overline{\Phi}(2^{-l}r) \int_{1}^{\frac{t_{1}}{\theta\overline{\Phi}(2^{-l}r)}}-\dot{k}\left((p-\theta_{2})(\theta\overline{\Phi}(2^{-l}r))\right)[2^{\kappa}\left(\frac{p}{\theta}\right)^{\frac{ \kappa}{2}}-1]dp.
\]
We apply the assumption (\ref{asholder}) with $x = (p-\theta_2)$ and $y = \theta\overline{\Phi}(2^{-l}r)$ to get
\[
\int_{0}^{\tilde{t}}-\dot{k}(s+t_{**}-\tau)w(\tau,x) d\tau
\]
\[
\geq -(b_{l}-a_{l})c(\theta,\eta)k_1(\theta\overline{\Phi}(2^{-l}r))\int_{1}^{\frac{t_{1}}{\theta\overline{\Phi}(2^{-l}r)}}(p-\theta_2)^{-\beta-1}[2^{\kappa}\left(\frac{p}{\theta}\right)^{\frac{ \kappa}{2}}-1]dp
\]
\[
\geq -(b_{l}-a_{l})c(\theta,\eta)k_1(\theta\overline{\Phi}(2^{-l}r))\int_{1}^{\infty}(p-\theta_2)^{-\beta-1}\left[\left(\frac{4p}{\theta}\right)^{\frac{ \kappa}{2}}-1\right]dp,
\]
and the last expression tends to zero as $\kappa$ goes to zero by the dominated convergence theorem.

Finally, for $(s,x)\in (0,\theta\theta_{2} \vkrl)\times \blrj$
\eqq{
\int_{0}^{\tilde{t}}-\dot{k}(s+t_{**}-\tau)w(\tau,x) d\tau \geq -(b_{l}-a_{l})k_1(\overline{\Phi}(2^{-l}r)) \ve(\theta,\theta_{2},\eta,\kappa) \m{ where } \ve \rightarrow 0 \m{ as } \kappa \rightarrow 0.
}{hw}
Combining (\ref{hw}) with (\ref{zn1}) we obtain
\[
2^{-l(2-\frac{N}{q_2})}r^{2-\frac{N}{q_2}}(\overline{\Phi}(2^{-l}r))^{-\frac{1}{q_1}} \|(\int_{0}^{\tilde{t}}-\dot{k}(\cdot+t_{**}-\tau)w(\tau,\cdot) d\tau)^{-}\|_{L_{q_1}((0,\theta\theta_{2} \vkrl);L_{q_2}(\blrj))}
\]
\eqq{
\leq 2^{-2l}(b_{l}-a_{l})\ve(\theta,\theta_{2},\eta,\kappa).
}{estiHw}
Having $\Gf(s,x)= G(s+\tss,x)$,\hd  $s\in (0,t_{1}-\tss)$, $x\in \blrj$, we may estimate
\[
\| \Gf \|_{L_{q_1}((0,t_{1}-\tss);L_{q_2}(\blrj))}\leq \| G  \|_{L_{q_1}((\tilde{t},t_{1});L_{q_2} (\blrj))} = \| G \|_{L_{q_1,q_2}(Q(2^{-l}r))}.
\]
We would like to apply (\ref{vges}), thus it should be verified that
\eqq{\Qdl=(\tilde{t}, t_{1})\times \blrj \subseteq \left(\frac{\eta}{2}\Pk(r) ,2\eta\Pkr \right)\times B_{2r}(x_{0}).}{zawieranie1}
Obviously, $\blrj\subseteq B_{2r}(x_{0})$, because $x_{1}\in B_{r}(x_{0})$ and $l\geq l_{0} \geq \lf \geq 0 $. Thus, for $t_{1}\in (\eta \Pkr, 2\eta \Pkr)$ and $\tilde{t} = t_1-\theta\overline{\Phi}(2^{-l}r)$ the inclusion (\ref{zawieranie1}) holds if
\eqq{\theta\vkrl \leq \frac{\eta}{2} \Pkr.}{zaw2}
The parameters $\theta, \eta>0$ are fixed, so proceeding as earlier we deduce that there exists  $l_{0}\geq \lf$ sufficiently large such that (\ref{zaw2}) holds for $l\geq l_{0}$ and $l_{0}= l_{0}(\eta, \theta)$.
For such $l_{0}$, using (\ref{vges}) we obtain
\[
\| \Gf\|_{L_{q_1} ((0,t_{1}-\tss);L_{q_2}(\blrj)) }\leq \frac{(\bvr)^{\frac{1}{q_1}}}{2 r^{2-\frac{N}{q_2}}}.
\]
Thus, applying Lemma \ref{fixy} with $x=2\cdot 2^{-l}r$, $y=2^{l}$ we have
\[
2^{-l(2-\frac{N}{q_2})}r^{2-\frac{N}{q_2}} (\overline{\Phi}(2^{-l}r))^{-\frac{1}{q_{1}}}\| \Gf\|_{L_{q_1} ((0,t_{1}-\tss);L_{q_2}(\blrj)) } \leq 2^{-l(2-\frac{N}{q_2})-1} \left(\frac{\overline{\Phi}(r)}{\overline{\Phi}(2^{-l}r)}\right)^{\frac{1}{q_1}}
\]
\eqq{
\leq c\cdot 2^{-l(2-\frac{N}{q_2})-1}( 2^{2l})^{\frac{p'_0}{q_1}} \leq c \cdot 2^{-2l(1-\frac{N}{2q_2}-\frac{p'_0}{q_1})-1} = c\cdot2^{-2ld - 1},
}{estigf}
where we used relation (H3) and $c$ is a positive constant depending on $p_0$ and $\overline{c}$.

Now we pass to the estimate of the $b_{l}$-term in $\Psi$. Since $l \geq l_{0}\geq \tilde{l}$ we have
\[
t_{**}+s \geq \tss = t_{1}-\theta\theta_{2}\overline{\Phi}(2^{-l}r) \geq t_{1}-\theta\theta_{2}\overline{\Phi}(2^{-\tl}r) = \theta\overline{\Phi}(2^{-\tl}r)\left[\frac{t_{1}}{\theta\overline{\Phi}(2^{-\tl}r)}-\theta_{2}\right] \geq \theta\overline{\Phi}(2^{-\tl}r)[1-\theta_{2}],
\]
where the last inequality is a consequence of (\ref{alter2}). Next, $a_{l_{0}}\leq a_{l} \leq b_{l} \leq b_{l_{0}}$ implies $b_{l}\geq -|a_{l_{0}}| $ hence,  applying Lemma \ref{kernels} and Lemma \ref{ba} we may estimate as follows
\[
b_{l}k(t_{**}+s) \geq -|a_{l_{0}}|k_1(t_{**}+s) \geq -|a_{l_{0}}|k_1\left(\theta\overline{\Phi}(2^{-\tl}r)[1-\theta_{2}]\right) \geq -|a_{l_{0}}| \frac{1}{\theta(1-\theta_2)}k_1(\overline{\Phi}(2^{-\tl}r))
\]
\[
\geq -\jd c(\theta,\theta_{2})k_1(\overline{\Phi}(2^{-\tl}r))\geq -\jd c(\theta,\theta_{2})\ki(\overline{\Phi}(2^{-\tl}r)) \geq -  \jd c(\theta,\theta_{2})2^{2\tl-2}r^{-2},
\]
where we used (\ref{zn1}) and the normalization condition $|a_{l_{0}}| \leq \|v\|_{L_{\infty}(Q_{dom})} \leq \frac{1}{2}$  (see (\ref{vges})).
Thus,
\eqq{
\norm{b_{l}k(t_{**}+\cdot)}_{L_{q_1}(0,\theta\theta_2 \overline{\Phi}(2^{-l}r)); L_{q_2}(B_{2^{-l}r}(x_1))} \leq c(\theta,\theta_2)2^{2(\tilde{l}-1)}r^{-2}r^{\frac{N}{q_2}}2^{-\frac{lN}{q_2}}(\overline{\Phi}(2^{-l}r)))^{\frac{1}{q_1}}.
}{estibl}

Combining (\ref{defPsi}), (\ref{estiHw}), (\ref{estibl}) and (\ref{estigf}) we obtain
\[
2^{-l(2-\frac{N}{q_2})}r^{2-\frac{N}{q_2}}(\overline{\Phi}(2^{-l}r))^{-\frac{1}{q_1}}\|\Psi^{-}\|_{L_{q_1}(0,\theta\theta_2 \overline{\Phi}(2^{-l}r)); L_{q_2}(B_{2^{-l}r}(x_1))}
\]
\eqq{
\leq (b_{l}-a_{l})\ve(\mu, \theta, \eta, \kappa) + c(\theta,\theta_{2})2^{2(\tl-l)} + c(\overline{c},p_0)2^{-2ld - 1}.
}{ab8}
Inserting this result into (\ref{wazne2}) we obtain the bound on the essential supremum of $v$ on a smaller cylinder,
\eqq{
\esssup_{\Qdll}v \leq b_{l}-\frac{1}{4C}(b_{l}-a_{l}) + (b_{l}-a_{l})\ve(\mu, \theta, \eta, \kappa) + c(\theta,\theta_{2})2^{2(\tl-l)} +  c(\overline{c},p_0)2^{-2ld - 1}=:\beta_{l+1}.
}{estisup1}
We define
\[
a_{l+1} = a_{l}, \hd b_{l+1} = a_{l}+2^{-(l+1-l_{0})\kappa}.
\]
Then, (\ref{abclaim}) gives $b_{l+1}\leq b_{l}$ and
\[
a_{l+1}=a_{l}\leq \essinf_{\Qdl} v \leq \essinf_{\Qdll} v \leq v \hd \m{ on \hd } \Qdll.
\]
Thus, in order to show that in case (A) (\ref{abclaim}) holds for $j=l+1$, it is enough to prove $\beta_{l+1}\leq b_{l+1}$.
We will choose $\kappa$ small enough and $l_{0}$ big enough so that $\beta_{l+1}\leq b_{l+1}$. The calculations here are similar to those carried out in \cite{harnackdistr}, however, here the term coming from the estimate of the $G$-term is different. We have
\[
\beta_{l+1} \leq b_{l+1} \iff (b_{l}-a_{l})\left(1-\frac{1}{4C}+\ve(\mu, \theta, \eta, \kappa)\right)+c(\theta,\theta_{2})2^{2(\tl-l)} +  c(\overline{c},p_0)2^{-2ld - 1} \leq 2^{-(l+1-l_{0})\kappa}.
\]
Applying the induction hypothesis (\ref{abclaim}) for $j=l$, it is enough to show that
\[
2^{\kappa}\left(1-\frac{1}{4C}+\ve(\mu, \theta, \eta, \kappa)\right)+c(\theta,\theta_{2})2^{\kappa+ 2(\tl-l)+(l-l_{0})\kappa} +  c(\overline{c},p_0) 2^{\kappa -2ld -1+(l-l_{0})\kappa} \leq 1,
\]
i.e.
\[
2^{\kappa}\left(1-\frac{1}{4C}+\ve(\mu, \theta, \eta, \kappa)\right)+2^{\kappa+(l-l_{0})(\kappa-d) -dl_{0}}\left[ c(\theta,\theta_{2})2^{2\lf}2^{(d-2)l} +  c(\overline{c},p_0) 2^{-dl-1}\right] \leq 1.
\]
Since $l \geq l_{0}$, $\tilde{l}\leq \gamma(\theta,\eta)$  it is enough to show that for $\kappa \in (0,d)$
\[
2^{\kappa}\left(1-\frac{1}{4C}+\ve(\mu, \theta, \eta, \kappa)\right)+2^{\kappa -dl_{0}}\left[ c(\theta,\theta_{2})2^{2\gamma(\theta,\eta)} +  c(\overline{c},p_0) 2^{ -1}\right] \leq 1.
\]
We recall that here $C=C(\Lambda, \nu, N, \theta,\theta_{1}, \theta_{2},p_0,q_1,q_2,\overline{c},\tilde{c})$.
\nic{
Then by induction hypothesis
\[
\beta_{l+1} \leq b_{l+1} \iff (b_{l}-a_{l})(1-\frac{1}{4C}+\ve)+\frac{1}{2}c(\theta,\theta_{2})2^{-2(l-\tl)} + 2^{-2l -1} \leq 2^{-(l+1-l_{0})\kappa}.
\]
\eqq{
\iff 2^{\kappa}(1-\frac{1}{4C}+\ve)+2^{\kappa-(l-l_{0})(2-\kappa)}(\frac{1}{2}c(\theta,\theta_{2})2^{-2(l_{0}-\tl)}+2^{-2l_{0}-1}) \leq 1.
}{lessone}
We note that
\[
2^{\kappa-(l-l_{0})(2-\kappa)}(\frac{1}{2}c(\theta,\theta_{2})2^{-2(l_{0}-\tl)}+2^{-2l_{0}-1}) = 2^{l(\kappa-2)-l_{0}\kappa+\kappa}(\frac{1}{2}c(\theta,\theta_{2})2^{2\tl}+\frac{1}{2})
\]
and
\[
l(\kappa-2)-l_{0}\kappa+2l_{0} = (l-l_{0})(\kappa-2) \leq 0.
\]
Since $l \geq l_{0}$ and $\tilde{l}\leq \gamma(\theta,\eta)$ it is enough to show that
\[
2^{\kappa}(1-\frac{1}{4C}+\ve)+2^{\kappa-2l_{0}}(c(\theta,\theta_{2})2^{2\gamma}+\frac{1}{2}) \leq 1.
\]}
Since $\ve( \theta, \eta, \kappa)\rightarrow 0$ as $\kappa \rightarrow 0$ we may choose $\kappa$ small enough so that the first summand is smaller then $1-\frac{1}{8C}$. We fix such $\kappa$. Then we choose $l_{0}\geq \tl$ so large that the second summand is smaller then $\frac{1}{8C}$. In this way we obtain $\beta_{l+1}\leq b_{l+1}$, thus (\ref{abclaim}) is satisfied for $j=l+1$.

The reasoning in the case (B) is analogous.
We define $w=v-a_{l}$ and shifting as before we arrive at
\[
\partial_s (k*\tilde{w})(s,x) - \divv(\tilde{A}(s,x)D\tilde{w}(s,x))
\]
\[
\geq H(w)(s,x) - a_{l}k(t_{**}+s) + \tilde{G}(s,x)=:\Psi(s,x), \hd (s,x) \in (0,t_{1}-t_{**})\times B_{2^{-l}r}(x_{1}).
\]
We apply Theorem~\ref{localweakHarnack} with the same parameters and sets to $\tilde{w}$ and we obtain   the analogue of~(\ref{ab2}),
\eqq{
\frac{1}{\nuNj(Q_{-})} \int_{Q_{-}} \tw dxds \leq  C\left[ \essinf_{\Qdll} v -a_{l}+ 2^{-l(2-\frac{N}{q_2})}r^{2-\frac{N}{q_2}}(\overline{\Phi}(2^{-l}r))^{-\frac{1}{q_1}} \| \Psi^{-} \|_{L_{q_1}((0,2\tau \vkrl);L_{q_2} (\blrj)}
\right].
}{ab5}
Proceeding as earlier,  from the assumption (B) we obtain
\[
\jd \nuNj(Q_{-})= \left(\frac{3}{2}\right)^{N+1}\frac{\theta_{2}}{4(\theta_{2}-\theta_{1})} \jd \nuNj(Q^{-})\leq \left(\frac{3}{2}\right)^{N+1} \frac{\theta_{2}}{4(\theta_{2}-\theta_{1})}  \nuNj(\{(t,x)\in Q^{-}:\hd v(t,x)> m_{l} \}).
\]
Thus,
\[
 \left(\frac{2}{3}\right)^{N+1} \frac{2(\theta_{2}-\theta_{1})}{\theta_{2}}\nuNj(Q_{-}) \leq  \nuNj(\{(s,x)\in  Q^{-}-\tss:\hd \tv (s,x)> m_{l} \})
\]
\[
= \nuNj(\{(s,x)\in  Q^{-}- \tss:\hd m_{l}-a_{l} < \tw (s,x) \})
=  (m_{l}-a_{l})^{-1}
\int_{ \{(s,x)\in Q^{-}-\tss:\hd m_{l}-a_{l} < \tw (s,x) \}  } (m_{l}-a_{l}) dxds
\]
\[
\leq   (m_{l}-a_{l})^{-1}
\int_{  Q^{-}-\tss  } \tw(s,x) dxds
\leq   (m_{l}-a_{l})^{-1}
\int_{ Q_{-}  } \tw(s,x) dxds,
\]
where in the last inequality we applied (\ref{pop1}). Consequently, we get
\eqq{\jd (m_{l}-a_{l})\leq \left(\frac{3}{2}\right)^{N+1}\frac{\theta_{2}}{4(\theta_{2}-\theta_{1})} \frac{1}{\nuNj(Q_{-})} \int_{Q_{-}} \tw dxds.}{ab6}
From (\ref{ab5}) and (\ref{ab6}) we obtain
\[
\frac{1}{2}(m_{l}-a_{l}) \leq   C\left[ \essinf_{\Qdll} v -a_{l}+2^{-l(2-\frac{N}{q_2})}r^{2-\frac{N}{q_2}}(\overline{\Phi}(2^{-l}r))^{-\frac{1}{q_1}} \| \Psi^{-} \|_{L_{q_1}((0,2\tau \vkrl);L_{q_2} (\blrj)}
\right],
\]
Thus,
\eqq{
   a_{l} + \frac{1}{2C}(m_{l}-a_{l}) - 2^{-l(2-\frac{N}{q_2})}r^{2-\frac{N}{q_2}}(\overline{\Phi}(2^{-l}r))^{-\frac{1}{q_1}} \| \Psi^{-} \|_{L_{q_1}((0,2\tau \vkrl);L_{q_2} (\blrj)} \leq  \essinf_{\Qdll} v .
}{ab7}
Proceeding as in case $(A)$, we obtain the same estimate for the function $\Psi^{-}$ as in the previous case and we arrive at (\ref{ab8}). Hence,
\eqq{
\al_{l+1}:= a_{l} + \frac{1}{4C}(b_{l}-a_{l}) -  (b_{l}-a_{l})\ve(\mu, \theta, \eta, \kappa) - c(\theta,\theta_{2})2^{2(\tl-l)} -  c(\overline{c},p_0)2^{-2ld - 1} \leq \essinf_{\Qdll} v .
}{ab9}
In this case we set $b_{l+1}=b_{l}$ and $a_{l+1} = b_{l}-2^{-(l+1-l_{0})\kappa}$. Then using (\ref{abclaim}) for $j=l$ we get
\[
a_{l} = b_{l} - 2^{-(l-l_{0})\kappa} \leq b_{l} - 2^{-(l+1-l_{0})\kappa}= a_{l+1}
\]
and
\[
b_{l+1}-a_{l+1}= 2^{-(l+1-l_{0})\kappa}.
\]
Furthermore,
\[
v\leq \esssup_{\Qdll}{ v }\leq \esssup_{\Qdl} v \leq b_{l+1}=b_{l} \hd \m{ on } \hd \Qdll.
\]
Thus, if we show that $a_{l+1}\leq \al_{l+1}$, then  $a_{l+1}\leq v$ on $\Qdll$ and (\ref{abclaim}) holds for $j=l+1$ in case (B).  We note that
\[
a_{l+1} \leq \al_{l+1} \iff (b_{l}-a_{l})\left(1-\frac{1}{4C}+\ve(\mu, \theta, \eta, \kappa)\right)+c(\theta,\theta_{2})2^{2(\tl-l)} +  c(\overline{c},p_0)2^{-2ld - 1}\leq 2^{-(l+1-l_{0})\kappa},
\]
thus we obtained the same condition as in case (A). Proceeding further as in case (A) we deduce that (\ref{abclaim}) is satisfied for $j=l+1$ in case (B).
By induction  (\ref{abclaim}) holds for all $j \in \mathbb{Z}$ and
\eqq{
\eosc_{Q_{dom}\cap Q(2^{-j})} v \leq 2^{-(j-l_{0})\kappa} \hd \m{ for } \hd j\in \mathbb{Z}.
}{osc1}
Recalling that $u=2D v$ we get
\[
\eosc_{Q_{dom}\cap Q(2^{-j})}{u} \leq 2D 2^{-(j-l_{0})\kappa} = C_{H}2^{-j\kappa}D,
\]
where we denoted $C_{H}:=2^{l_{0}\kappa+1}$.
Hence, we arrive at the following oscillation estimate
\eqq{
\eosc_{ Q(2^{-j})}{u} \leq C_{H}2^{-j\kappa}D \hd \m{ for \hd }  j \geq\tl.
}{osc2}

We will make this estimate continuous via a standard argument. We define the cylinders
\[
\tQ(\rho r) = (t_{1}-\theta\overline{\Phi}(\rho r),t_{1}) \times B(x_{1},\rho r) \hd \m{ for } \hd  \rho \in (0,\rho_{0}), \hd \rho_{0} = 2^{-\gamma(\theta,\eta)},
\]
where  $\gamma(\theta,\eta)$ comes from (\ref{ltildees}). Then there exists $j_{*} \geq \tilde{l}$ such that $2^{-(j_{*}+1)} < \rho \leq 2^{-j_{*}}$. Then from (\ref{osc2}) we get
\[
\eosc_{\tQ(\rho r)} u \leq \eosc_{Q(2^{-j_{*}} )} u  \leq C_{H}2^{-j_{*}\kappa}D \leq C_{H}2^{\kappa}\rho^{\kappa}D=:\tilde{C}\rho^{\kappa}D.
\]
Let $\tilde{\rho} = \rho r$ and recall that $D$ was defined in (\ref{osc3}).  Then for every $(t_{1},x_{1}) \in (\eta\bvr,2\eta\bvr) \times B(x_{0},r)$ we have
\[
\eosc_{\tQ(\tilde{\rho})} u\leq \tilde{C}\left(\frac{\tilde{\rho}}{r}\right)^{\kappa}\left(\|u\|_{L_{\infty}((0,2\eta\bvr)\times B_{2r}(x_{0}))} + r^{2-\frac{N}{q_2}}(\vdr)^{-\frac{1}{q_1}}\|F\|_{L_{q_1}((\frac{\eta}{2}\bvr,2\eta\bvr);L_{q_2}( B_{2r}(x_{0})))}\right).
\]
We note that
\[
r^{2-\frac{N}{q_2}}(\vdr)^{-\frac{1}{q_1}}\|F\|_{L_{q_1}((\frac{\eta}{2}\bvr,2\eta\bvr);L_{q_2}( B_{2r}(x_{0})))}
\]
\[
\leq r^{2-\frac{N}{q_2}}(\vdr)^{-\frac{1}{q_1}}\|f\|_{L_{q_1}((\frac{\eta}{2}\bvr,2\eta\bvr);L_{q_2}( B_{2r}(x_{0})))}  + r^{2}\|u_{0}\|_{L_{\infty}(B(x_{0},2r))}k\left(\frac{\eta}{2}\bvr \right)
\]
\[
\leq r^{2}\|f\|_{L_{q_1}((\frac{\eta}{2}\bvr,2\eta\bvr);L_{q_2}( B_{2r}(x_{0})))} + c(\eta)\|u_{0}\|_{L_{\infty}(B(x_{0},2r))},
\]
where we used the estimate
\[
r^{2}k\left(\frac{\eta}{2}\bvr \right)  \leq c(\eta )r^{2}\ki\left(\vr \right) =c(\eta).
\]
Let us now fix $\theta=\tau^{*}$, where $\tau^{*}$ comes from Theorem \ref{localweakHarnack}.
We have proven that there exists $\tau^{*}>0$ and $r^{**}=r^{**}(\tau^{*},\eta) > 0$ such that  for each  $r\in (0,r^{**})$,  there exists $\gamma=\gamma(\eta, \tau^{*})>0$,   $\tilde{C}>0$ and $\kappa\in (0,1)$, which depend on $\Lambda, \nu, N, \tau^{*},\theta_{1}, \theta_{2},p_0,q_1,q_2,\overline{c},\tilde{c}$ such that for any weak solution of (\ref{osc4}) and any $(t_{1},x_{1})\in (\eta\bvr, 2\eta\bvr)\times B(x_{0},r)$
\eqnsl{\eosc_{(t_{1}- \tau^{*} \Pkro, t_{1} )\times B(x_{1},\rok)}{u} & \\
\leq \tilde{C}\left(\frac{\tilde{\rho}}{r}\right)^{\kappa}(\|u\|_{L_{\infty}((0,2\eta\bvr)\times B_{2r}(x_{0}))} +& r^{2}\|f\|_{L_{q_1}((\frac{\eta}{2}\bvr,2\eta\bvr);L_{q_2}( B_{2r}(x_{0})))} + \|u_{0}\|_{L_{\infty}(B(x_{0},2r))}),}{osc5}
where $\rok\in (0,\rho_{0}r)$, $\rho_{0}= 2^{-\gamma(\eta, \tau^{*})}$. Applying Proposition \ref{estiprop} we obtain
\eqq{\eosc_{(t_{1}- \tau^{*} c  {\rok}^{2p'_0}, t_{1} )\times B(x_{1},\rok)}{u} \leq \eosc_{(t_{1}- \tau^{*} \Pkro, t_{1} )\times B(x_{1},\rok)}{u},}{osc6}
hence (\ref{osc5}) implies
\[
\sup_{\begin{array}{l} x_{1}\in B(x_{0},r) \\ t_{1} \in (\eta \bvr, 2 \eta \bvr)  \end{array}} \esssup_{\begin{array}{l} 0<t_{1}-t_{2}< \theta c (2^{-\gamma(\eta, \tau^{*})}r)^{2p'_0}  \\ |x_{1}-x_{2}|<2^{-\gamma(\eta, \tau^{*})}r\end{array}} \frac{|u(t_{1},x_{1})-u(t_{2},x_{2}) |}{\left(|t_{1}-t_{2}|^{\frac{1}{2p'_0}}+|x_{1}-x_{2}| \right)^{\kappa} }
\]
\[
\leq \tilde{C}\left(\frac{1}{r}\right)^{\kappa}(\|u\|_{L_{\infty}((0,2\eta\bvr)\times B_{2r}(x_{0}))} + r^{2}\|f\|_{L_{q_1}((\frac{\eta}{2}\bvr,2\eta\bvr);L_{q_2}( B_{2r}(x_{0})))} + c(\eta)\|u_{0}\|_{L_{\infty}(B(x_{0},2r))}).
\]
Then a standard argument yields the H\"older continuity of the weak solution $u$ on the set $(\eta\bvr, 2\eta\bvr)\times B(x_{0},r)$.

To finish the proof of Theorem~\ref{holder}, for a given subset $V\subset \Om_{T}$ separated from the parabolic boundary of $\Om_{T}$ it is enough to choose a finite  covering of $V$ by a family of sets $(t_m+\eta\bvr, t_m+2\eta\bvr)\times B(x_{n},r)=:Q_{n,m}$, where $r\in (0,r^{**})$ and $\eta$ are sufficiently small. Then for $(t,x) \in Q_{n,m}$ we introduce the shifted time $s=t-t_m$ and set $\tilde{g}(s) = g(s+t_m)$ for $s \in (\overline{\Phi}(r)\eta,2\overline{\Phi}(r)\eta)$. Then $\tilde{u}(s,x)  = u(s+t_m,x)$ is a weak solution to
\eqq{
\partial_s (k*(\tilde{u}-u_0))(s,x) + \divv (\tilde{A}(s,x)D\tilde{u}(s,x)) = \tilde{f}(s,x) + \int_{0}^{t_m}\dot{k}(s+t_m-\tau)(u(\tau,x)-u_0(x))d\tau.
}{osc7}
Since $u$ is bounded the right-hand-side of (\ref{osc7}) is bounded on $Q_{n,m}$ and
\[
\norm{\int_{0}^{t_m}\dot{k}(\cdot+t_m-\tau)(u(\tau,\cdot) - u_0(x))d\tau }_{L_{\infty}((\frac{\eta}{2}\overline{\Phi}(r), 2\eta \overline{\Phi}(r))\times B(x_{n}, 2r) )}
\]
\[
\leq (\norm{u}_{L_{\infty}(\Omega_{T})} + \norm{u_0}_{L_{\infty}(\Omega)})k\left(\overline{\Phi}(r)\frac{\eta}{2}\right) \leq c(\eta)(\norm{u}_{L_{\infty}(\Omega_{T})} + \norm{u_0}_{L_{\infty}(\Omega)}) r^{-2}.
\]
Thus, $u$ is H\"older continuous on each  $Q_{n,m}$, hence $u$ is H\"older continuous on $V$ and the estimate (\ref{holderkoniec}) holds.
\section{Appendix}
Here we show that all the examples of kernels listed in the introduction satisfy the assumptions (\ref{pc}) - (\ref{asholder}). We begin with two simple remarks.
\begin{bemerk1}\label{remdif}
The inequality in the assumption (\ref{ak1}) may be deduced from the following stronger differential inequality: there exists $\vartheta \in (0,1)$ such that for every $p \in [1,p_0]$ and every $0\leq t \leq t_0$
\eqq{
-t\dot{l}(t) \leq \frac{\vartheta}{p}l(t).
}{diffak1}
Indeed, since $l \in L_{p}((0,t_0))$, the inequality in (\ref{ak1}) follows from
\[
l^{p}(t) \leq \overline{c}(l^{p}(t)+pt(l(t))^{p-1}\dot{l}(t)), \hd \hd t \in (0,t_0],
\]
which is satisfied due to (\ref{diffak1}) with $\overline{c} = \frac{1}{1-\vartheta}$.
\end{bemerk1}
\begin{bemerk1}
The inequality in assumption (\ref{ak2}) follows from convexity of $k$, together with $k(2t) \leq \theta k(t)$, for some $\theta \in (0,1)$ and all $t \in (0,\tilde{t}_{0})$. We note that all the kernels from our examples are completely monotone and thus convex.
\end{bemerk1}

\begin{bei}
We consider the $\mathscr{PC}$- pair
\[
(k,l) = \left(\frac{t^{-\al}}{\Gamma(1-\al)}e^{-\gamma t}, \frac{t^{\al-1}}{\Gamma(\al)}e^{-\gamma t} + \gamma \izt e^{-\gamma \tau}\frac{\tau^{\al-1}}{\Gamma(\al)}d\tau \right), \hd \hd \gamma \geq 0, \hd \al \in (0,1).
\]
Note that, $l \in L_{p_0}(0,t_0)$ for every $t_0 < \infty$ and every $p_0 < \frac{1}{1-\al}$. To show the estimate in (\ref{ak1}), we will prove the stronger estimate (\ref{diffak1}). Since $\dot{l}(t) = \frac{\al-1}{\Gamma(\al)}t^{\al-2}e^{-\gamma t}$, we have
\[
-t\dot{l}(t) \leq (1-\al)l(t),
\]
hence for any $p_0 < \frac{1}{1-\al}$, we may choose $\vartheta =(1-\al) p_0 < 1$ such that for any $p \in [1,p_0]$ and any $t > 0$ (\ref{diffak1}) holds.
Furthermore,
\[
-\dot{k}(t) = \frac{t^{-\al}e^{-\gamma t}}{\Gamma(1-\al)}(\al t^{-1}+\gamma) \geq \al \frac{k(t)}{t}
\]
and thus (\ref{ak2}) is satisfied.
It remains to show (\ref{asholder}). We have
\[
-\dot{k}(xy) = \frac{x^{-\al}y^{-\al}e^{-\gamma x y}}{\Gamma(1-\al)}(\frac{\al}{xy}+\gamma) \leq \frac{x^{-\al}y^{-\al}}{\Gamma(1-\al)}(\frac{\al}{xy}+\gamma).
\]
Since $1 \leq e^{\gamma}e^{-\gamma y}$ for $y \leq 1$,
\[
-\dot{k}(xy) \leq e^{\gamma} k(y)(\frac{\al x^{-\al-1}}{y} + \gamma x^{-\al}) \leq e^{\gamma}\frac{k_1(y)}{y} (\al x^{-\al-1} +\gamma x^{-\al}y),
\]
where we applied $k(y) \leq k_1(y)$ (see Lemma \ref{kernels}). Finally, for $xy \leq D$, we have $x^{-\al}y \leq D x^{-\al-1}$, which finishes the proof.
\end{bei}

\begin{bei} \label{distexamp}
The $\mathscr{PC}$ - pair $(k,l)$ discussed in \cite{harnackdistr} satisfies the assumptions (\ref{ak1}) - (\ref{asholder}). Let us recall the construction of the kernel $k$ associated with distributed order fractional time derivatives.
 Let $\{\alpha_{n}\}_{n=1}^{M}$ satisfy
\[
0<\al_{1}<\al_{2}<\dots <\al_{M}<1,
\]
$q_{n}$, $n=1, \dots, M$, be nonnegative numbers, and $w\in L_{1}((0,1))$ be nonnegative. We define the measure $\mu$ on the Borel sets in $\iR$ by
\eqq{d\mu=\sum_{n=1}^{M}q_{n}d\delta(\cdot - \al_{n})+w d \nu_{1},  }{bo1}
where $\delta(\cdot - \al_{n})$ is the Dirac measure at $\alpha_n$. Here we allow the first or the second component in the above
representation to vanish, but we always assume that $\mu \not \equiv 0$. Then we define
\begin{equation} \label{kintro}
 k(t):=\izj \frac{t^{-\al} }{\Gamma(1-\al)} \dd,\quad t>0.
\end{equation}
Then
\[
l(t) = \frac{1}{\pi}\izi e^{-pt} H(p)dp, \hd H(p) = \frac{\izj p^{\al} \sin (\pi \al)\dd}{(\izj p^{\al} \sin (\pi \al)\dd)^{2} + (\izj p^{\al} \cos (\pi \al)\dd)^{2}}.
\]
 From \cite[Lemma 2.5]{harnackdistr} we know that there exists a number $\gmb \in (0,1)$ such that the kernel $l$ belongs to $L_{p}((0,1))$ for any $p < \frac{1}{1-\gmb}$. In order to show the estimate in (\ref{ak1}) we note that, again from  \cite[Lemma 2.5]{harnackdistr}, for $t>0$ small enough
 \[
 \izt l^{p}(\tau)d\tau \leq c \frac{t}{(1*k)^{p}(t)}.
 \]
 By \cite[Lemma 2.1]{harnackdistr}, we also have for sufficiently small $t > 0$ that $l(t) > c \frac{1}{(1*k)(t)}$. Combining these estimates yields (\ref{ak1}). The  assumption (\ref{ak2}) holds for every $t \in (0,1)$, due to the simple estimate (38)  in \cite{harnackdistr}:
 \[
 \izj \frac{\al}{\Gamma(1-\al)}t^{-\al}\dd \geq c(\mu)\izj \frac{1}{\Gamma(1-\al)}t^{-\al}\dd \hd \m{ for } \hd t \in (0,1).
 \]
To show that the kernel $k$ satisfies the assumption (\ref{asholder}) we note that
 \[
 -\dot{k}(xy) = \izj \frac{\al}{\Gamma(1-\al)}x^{-\al-1}y^{-\al-1}\dd.
 \]
 We choose $\gamma \in (0,\frac{1}{2})$ such that $\int_{2\gamma}^{1}\mu(\al)d\al > 0$ and estimate as follows.
 \[
 -\dot{k}(xy) \leq \frac{1}{xy} \int_{\gamma}^{1} \frac{1}{\Gamma(1-\al)}x^{-\al}y^{-\al}\dd + \frac{1}{xy} \int_{0}^{\gamma}\frac{1}{\Gamma(1-\al)}x^{-\al}y^{-\al}\dd.
 \]
 Since $x >1$ we have
 \[
 \int_{\gamma}^{1} \frac{1}{\Gamma(1-\al)}x^{-\al}y^{-\al}\dd \leq x^{-\gamma}k(y) \leq x^{-\gamma}k_1(y).
 \]
Since for $y\in (0,1)$ there holds
\[
y^{-\gamma} = y^{\gamma}y^{-2\gamma}\left(\int_{2\gamma}^{1}\frac{1}{\Gamma(1-\al)} \dd \right) \left(\int_{2\gamma}^{1}\frac{1}{\Gamma(1-\al)} \dd\right)^{-1}
\]
\[
\leq y^{\gamma} \left(\int_{2\gamma}^{1}y^{-\al}\frac{1}{\Gamma(1-\al)} \dd\right) \left(\int_{2\gamma}^{1}\frac{1}{\Gamma(1-\al)} \dd \right)^{-1}
\leq y^{\gamma} c(\mu)k(y),
\]
we may estimate as
 \[
 \int_{0}^{\gamma}\frac{x^{-\al}y^{-\al}}{\Gamma(1-\al)}\dd \leq y^{-\gamma}\int_{0}^{\gamma}\frac{x^{-\al}}{\Gamma(1-\al)}\dd \leq c(\mu) k(y) y^{\gamma}\int_{0}^{\gamma}\frac{1}{\Gamma(1-\al)}\dd
 \leq c(\mu) D^{\gamma} k_1(y) x^{-\gamma},
 \]
 which leads to (\ref{asholder}).
\end{bei}

\begin{bei}
Let us consider the pair from Example 1.1 with switched kernels, i.e.
\[
(k,l) = \left(\frac{t^{\al-1}}{\Gamma(\al)}e^{-\gamma t} + \gamma \izt e^{-\gamma \tau}\frac{\tau^{\al-1}}{\Gamma(\al)}d\tau, \frac{t^{-\al}}{\Gamma(1-\al)}e^{-\gamma t} \right), \hd \hd \gamma \geq 0, \hd \al \in (0,1).
\]
Now, it is easy to see that $l \in L_{p_0}((0,t_0))$ for every $t_0 < \infty$ and every $p_0 < \frac{1}{\al}$.
Furthermore,
\[
-\dot{l}(t) = \frac{t^{-\al}e^{-\gamma t}}{\Gamma(1-\al)}(\al t^{-1}+\gamma),
\]
thus, for every $p_0 < \frac{1}{\al}$, there exists $t_0 = \frac{1}{2\gamma}[\frac{1}{p_0}-\al]$ and $\vartheta = \frac{1}{2}(\al p_0 +1) \in (0,1)$ such that, for any $t \in (0,t_0)$ and any $p \in (1,p_0]$
 \[
-t\dot{l}(t) = (\al+\gamma t)l(t) \leq \frac{1}{2}(\al+\frac{1}{p_0})l(t) =\frac{\vartheta}{p_0} \leq \frac{\vartheta}{p}l(t).
\]
Hence, by the Remark \ref{remdif} the assumption (\ref{ak1}) is satisfied.
To show (\ref{ak2}) we note that
\[
-t\dot{k}(t) = \frac{1-\al}{\Gamma(\al)}t^{\al-1}e^{-\gamma t}
\]
and for $t \in (0,1]$
\[
k(t) \leq \frac{t^{\al-1}e^{-\gamma t}}{\Gamma(\al)}(1+\frac{\gamma t e^{\gamma t}}{\al}) \leq -t\dot{k}(t) \frac{1+\frac{\gamma e^{\gamma}}{\al}}{1-\al}
\]
and (\ref{ak2}) is satisfied with $\tilde{c} = \frac{1-\al}{1+\frac{\gamma}{\al}e^{\gamma}}$.
It remains to show (\ref{asholder}), we have
\[
-\dot{k}(xy) = \frac{(1-\al)x^{\al-1}y^{\al-1}e^{-\gamma x y}}{xy\Gamma(\al)} \leq \frac{(1-\al)e^{\gamma}x^{\al-1}y^{\al-1}e^{-\gamma y}}{xy\Gamma(\al)} \leq (1-\al)e^{\gamma}\frac{k(y)}{xy} x^{\al-1}
\]
applying Lemma \ref{kernels}, we obtain that the assumption (\ref{asholder}) is satisfied.
\end{bei}

\begin{bei}
 We may also switch the kernels in Example 1.2, however in order to have $l \in L_{p}$ for some $p > 1$ we have to assume that $\supp \mu \subset [0,\al_{*}]$ for some $\al_* \in (0,1)$. So let us discuss the pair
 \[
 (k,l) = \left(\frac{1}{\pi}\izi e^{-pt} H(p)dp, \izj \frac{t^{-\al} }{\Gamma(1-\al)} \dd \right),
\]
where
\[
H(p) = \frac{\izj p^{\al} \sin (\pi\al)\dd}{(\izj p^{\al} \sin (\pi\al)\dd)^{2} + (\izj p^{\al} \cos (\pi\al)\dd)^{2}}
\]
and $\supp \mu \subset [0,\al_{*}]$ for some $\al_* \in (0,1)$. Then $l \in L_{p_0}((0,1))$ for every $p_0 < \frac{1}{\al_{*}}$. Furthermore,
\[
-t\dot{l}(t) =  \izj \frac{\al t^{-\al} }{\Gamma(1-\al)} \dd \leq \al_{*}l(t) \leq \frac{\vartheta}{p}l(t),
\]
for every $p \in (1,p_0]$ and $\vartheta = p_0 \al_{*}$. Hence, (\ref{ak1}) follows from Remark \ref{remdif}.
In the subsequent calculations we will use the following simple estimate from \cite[Remark 2.2.]{harnackdistr}: for some fixed $\gmb > 0$ which satisfies $\int_{\gmb}^{1}\dd > 0$,  there exists $c=c(\mu)$ such that for every $x \in (0,1]$
\eqq{
\izj x^{-\al}\dd \leq c \int_{\gmb}^{1}x^{-\al}\dd.
}{gml}
To show (\ref{ak2}) we note that
\[
-\pi t\dot{k}(t) = \izi pte^{-pt}H(p)dp = \frac{1}{t} \izi w e^{-w}H(\frac{w}{t})dw \geq \frac{1}{t} \int_{1}^{\infty}  e^{-w}H(\frac{w}{t})dw.
\]
Since
\[
\pi k(t) = \frac{1}{t} \int_{0}^{\infty}  e^{-w}H(\frac{w}{t})dw,
\]
it is enough to show that there exists $c>0$ such that
\eqq{
\int^{1}_{0}  e^{-w}H(\frac{w}{t})dw \leq c \int_{1}^{\infty}  e^{-w}H(\frac{w}{t})dw.
}{coscos}
From the definition of $H$ we have
\[
\int^{1}_{0}  e^{-w}H(\frac{w}{t})dw \leq \izj \frac{1}{\izj w^{\al}t^{-\al} \sin(\pi \al)\dd}dw \leq \frac{1}{1-\al_{*}}\frac{1}{\int_{\gmb}^{\al_{*}}t^{-\al} \sin(\pi \al)\dd} \leq c(\mu) \frac{1}{\int_{0}^{\al_{*}}t^{-\al} \dd},
\]
where in the last estimate we applied (\ref{gml}). Similarly, applying (\ref{gml}) we have
\[
\int_{1}^{\infty}  e^{-w}H(\frac{w}{t})dw \geq c(\mu) \int_{1}^{\infty}  e^{-w} \frac{\int_
{\gmb}^{\al_{*}}w^{\al}t^{-\al}\dd}{(\int_{0}^{\al_{*}}w^{\al}t^{-\al}\dd)^{2}}dw
\]
\[
\geq c(\mu) \int_{1}^{\infty}  e^{-w} \frac{1}{\int_{0}^{\al_{*}}w^{\al}t^{-\al}\dd}dw \geq c(\mu) \frac{1}{\int_{0}^{\al_{*}}t^{-\al} \dd}
\]
and we arrive at (\ref{coscos}). It remains to show (\ref{asholder}).
We note that here,
\[
k_{1}(y) = \frac{1}{(1*l)(y)} = \frac{1}{\izj \frac{y^{1-\al}}{\Gamma(2-\al)}\dd}.
\]
For $x \geq 1$ and $y <  1$, we may write
\[
-\pi \dot{k}(xy) = \izi pe^{-pxy}H(p)dp = (xy)^{-2}\izi we^{-w}H(\frac{w}{xy})dw
\]
\[
\leq (xy)^{-2}\izi we^{-w}\frac{1}{\izj w^{\al}(xy)^{-\al} \sin(\pi \al)\dd}dw \leq c(\mu)
(xy)^{-1} x^{\al_{*}-1}\izi we^{-w}\frac{1}{\int_{\gmb}^{\al_{*}} w^{\al}y^{1-\al} \dd}dw
\]
\[
\leq c(\mu) (xy)^{-1} x^{\al_{*}-1} \frac{1}{\int_{0}^{\al_{*}} y^{1-\al} \dd} \leq (xy)^{-1} x^{\al_{*}-1} k_{1}(y),
\]
where we applied (\ref{coscos}). This shows (\ref{asholder}).
\end{bei}


\end{document}